\documentclass[a4paper,11pt]{amsart}
\usepackage[utf8x]{inputenc}
\usepackage{amsthm,amssymb, amsmath}
\usepackage{graphicx}
\usepackage[all]{xy}

\usepackage
[pdfauthor={Giulio Tiozzo},
 pdftitle={Topological entropy of quadratic polynomials and dimension of sections of the Mandelbrot set},
 bookmarks=false]
{hyperref}

\newtheorem{theorem}{Theorem}[section]
\newtheorem{corollary}[theorem]{Corollary}
\newtheorem{conjecture}[theorem]{Conjecture}
\newtheorem{lemma}[theorem]{Lemma}
\newtheorem{proposition}[theorem]{Proposition}

\newtheorem{definition}[theorem]{Definition}

\title[Topological entropy of quadratic polynomials]{Topological entropy of quadratic polynomials and 
dimension of sections of the Mandelbrot set}
\author{Giulio Tiozzo}

\begin{document}

\begin{abstract}

Let $c$ be a real parameter in the Mandelbrot set, and $f_c(z):= z^2 + c$. 
We prove a formula relating the topological entropy of $f_c$ 
to the Hausdorff dimension of the set of rays landing on the real Julia set 
$J(f_c) \cap \mathbb{R}$,
and to the Hausdorff dimension of the set of rays landing on the real 
section of the Mandelbrot set, to the right of the given parameter $c$.
We then generalize the result by looking at the entropy of Hubbard trees: namely,
we relate the Hausdorff dimension of the set of external angles which land on a certain 
slice of a principal vein in the Mandelbrot set to the topological entropy of 
the quadratic polynomial $f_c$ restricted to its Hubbard tree.
\end{abstract}

\maketitle

\tableofcontents

\section{Introduction} 

Let us consider the family of quadratic polynomials 
$$f_c(z) := z^2 + c, \quad \textup{with }c \in \mathbb{C}.$$
The \emph{filled Julia set} $K(f_c)$ of a quadratic polynomial $f_c$
is the set of points which do not escape to infinity under iteration, and the \emph{Julia set}
$J(f_c)$ is the boundary of $K(f_c)$.
The \emph{Mandelbrot set} $\mathcal{M}$ is the connectedness locus of the quadratic family, i.e.
$$\mathcal{M} := \{ c\in \mathbb{C} \ : \ \textup{ the Julia set of }f_c \textup{ is connected} \}.$$

A fundamental theme in the study of parameter spaces in holomorphic dynamics is that 
the local geometry of the Mandelbrot set near a parameter $c$ reflects the geometry
of the Julia set $J(f_c)$, hence it is related to dynamical properties of $f_c$. In this paper 
we will establish an instance of this principle, by looking at the Hausdorff dimension of certain 
sets of external rays.

Recall that a measure of the complexity of a continuous map is its \emph{topological entropy}, which
is essentially defined as the growth rate of the number of itineraries under iteration
 (see section \ref{section:htop}). 

In our case, the map $f_c(z) = z^2 + c$ is a degree-two ramified cover of the Riemann sphere $\hat{\mathbb{C}}$, 
hence a generic point has exactly $2$ preimages, and the topological entropy of $f_c$ always equals $\log 2$, 
independently of the parameter \cite{Ly}.
If $c$ is real, however, then $f_c$ can also be seen as a real interval map, and its restriction to the real line
 also has a well-defined topological entropy, which we will 
denote by $h_{top}(f_c, \mathbb{R})$.  The dependence of $h_{top}(f_c, \mathbb{R})$ on $c$ is much more interesting: 
indeed, it is a continuous, decreasing function of $c$ \cite{MT}, 
and it is constant on baby Mandelbrot sets \cite{Do} (see Figure \ref{MTentropy}).


\begin{figure}[ht!]
\centering
\fbox{\includegraphics[width=0.75 \textwidth]{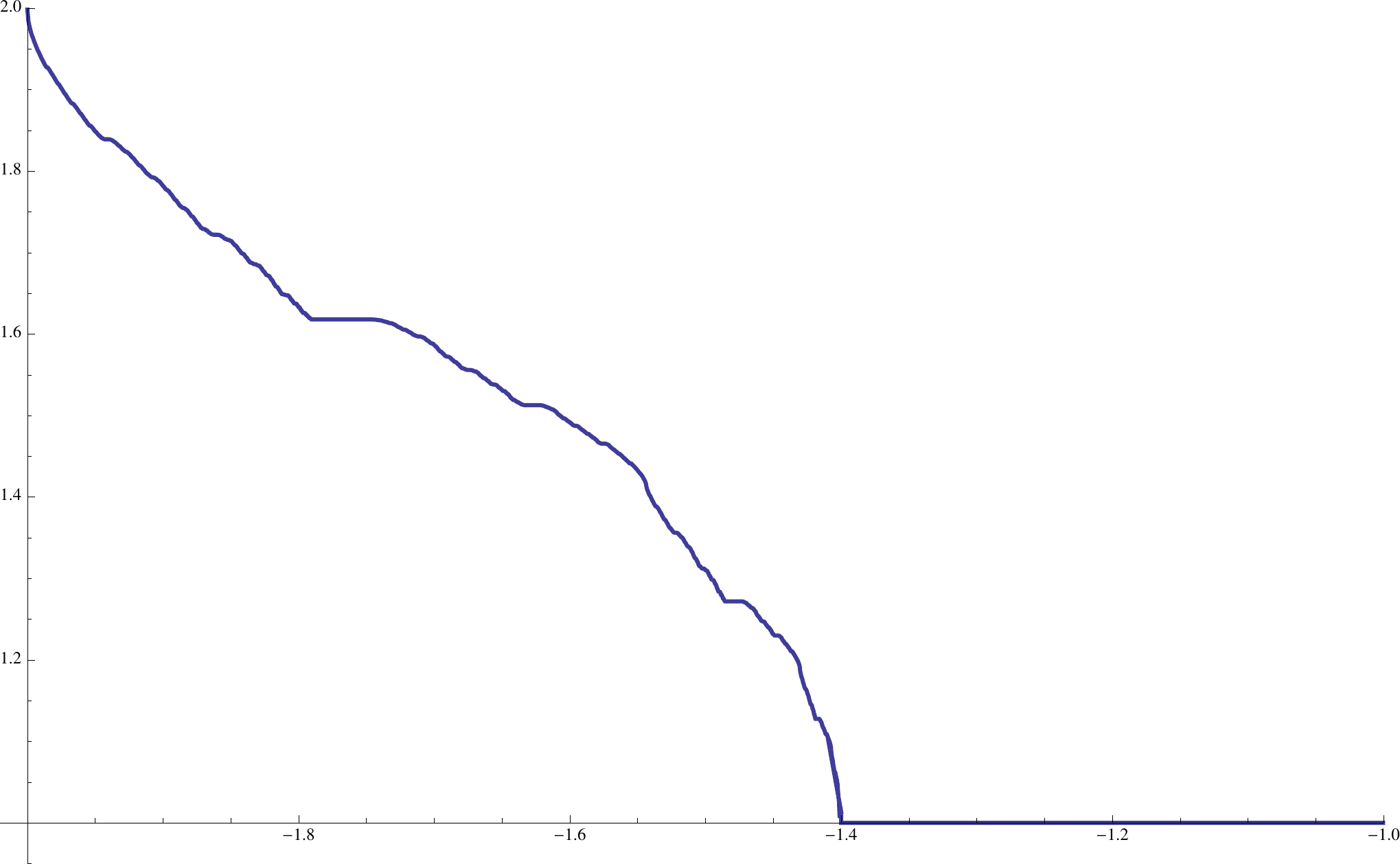}} 
\caption{Topological entropy of the real quadratic family $f_c(z) := z^2 + c$, as a 
function of $c$. For each value of $c \in [-2, -1]$, we plot the growth number
$e^{h_{top}(f_c)}$. }
\label{MTentropy}
\end{figure}

The Riemann map $\Phi_M : \hat{\mathbb{C}} \setminus \overline{\mathbb{D}} \to \hat{\mathbb{C}} \setminus \mathcal{M}$ 
uniformizes the exterior of the Mandelbrot set, and images of radial arcs are called \emph{external rays}. Each 
angle $\theta \in \mathbb{R}/\mathbb{Z}$ determines the external ray $R_M(\theta) := \Phi_M(\rho e^{2\pi i \theta})_{\{\rho > 1\}}$, 
which is said to \emph{land} if the limit as $\rho \to 1^+$ exists.

Given a subset $A$ of $\partial{\mathcal{M}}$, 
one can define the \emph{harmonic measure} $\nu_M$ as the probability that a random ray from infinity lands on $A$:
$$\nu_M(A) := \textup{Leb}(\{ \theta \in S^1 \ :  \ R_M(\theta) \textup{ lands on } A \}).$$
If one takes $A := \partial \mathcal{M} \cap \mathbb{R}$ to be the real slice of the boundary of $\mathcal{M}$, then the harmonic 
measure of $A$ is zero. 
However, the set of rays which land on the real axis has full Hausdorff dimension \cite{Za}. 
(By comparison, the set of rays which land on the main cardioid has zero Hausdorff dimension.)
As a consequence, it is more useful to look at Hausdorff dimension than harmonic measure; 
for each $c$, let us consider the section
$$P_c := \{ \theta \in S^1 :  R_M(\theta) \textup{ lands on }\partial{\mathcal{M}} \cap [c, 1/4] \}$$
of all parameter rays which 
land on the real axis, to the right of $c$. The function 
$$c \mapsto \textup{H.dim }P_c$$
increases from $0$ to $1$ as $c$ moves towards the tip of $\mathcal{M}$, 
reflecting the increased ``hairiness'' near the tip.
In the dynamical plane, one can consider the set of rays which land on the real slice of $J(f_c)$, 
and let $S_c$ be the set of external angles of rays landing on $J(f_c) \cap \mathbb{R}$. 
This way, we construct the function $c \mapsto \textup{H.dim }S_c$, which we want to compare to the 
Hausdorff dimension of $P_c$.

\begin{figure}
\centering
\fbox{\includegraphics[scale=0.2]{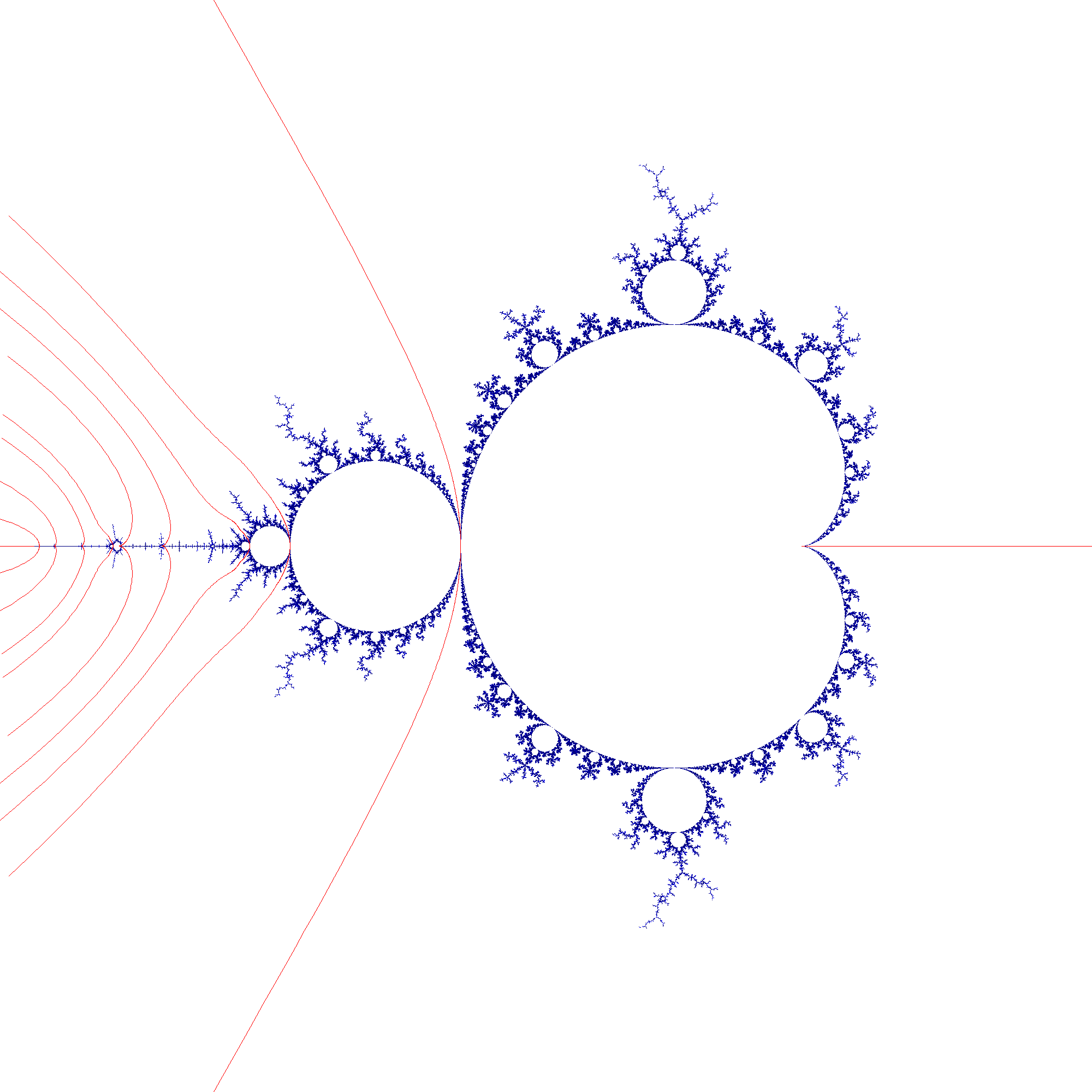}}
\caption{A few rays which land on the real slice of the Mandelbrot set.}
\end{figure}

The main result is the following identity:

\begin{theorem} \label{entropy_fmla} \label{realeq}
Let $c \in [-2, 1/4]$. Then we have
$$\frac{h_{top}(f_c,\mathbb{R})}{\log 2} = \textup{H.dim }S_c = \textup{H.dim }P_c.$$
\end{theorem}

The first equality establishes a relation between entropy, Hausdorff dimension and the Lyapunov 
exponent of the doubling map (in the spirit of the ``entropy formulas'' \cite{Ma}, \cite{Yo}, \cite{LeYo}), 
while the second equality can be seen as an instance of Douady's principle relating the local geometry of the Mandelbrot set
to the geometry of the corresponding Julia set. Indeed, we can replace $P_c$ with the set of angles of rays 
landing on $[c, c+\epsilon]$ in parameter space, as long as $[c, c+\epsilon]$ does not lie in a tuned copy of the Mandelbrot 
set. Note that the set of rays which possibly do not land has zero capacity, hence the result is independent of the MLC conjecture.

A first study of the dimension of the set of angles of rays landing on the real axis has been done in \cite{Za}, 
where it is proven that the set of angles of parameter rays landing on the real slice of $\mathcal{M}$ has dimension $1$. 
Zakeri also provides estimates on the dimension along the real axis, and specifically asks for dimension bounds  
for parameters near the Feigenbaum point ($-1.75 \leq c \leq c_{Feig}$, see \cite{Za}, Remark 6.9). 
Our result gives an identity rather than an estimate, and the dimension of $S_c$ can be exactly computed 
in the case $c$ is postcritically finite (see following examples).

Recall the dimension of $S_c$ also equals the dimension of the set $B_c$ of angles landing at biaccessible points
(Proposition \ref{biaccentro}). 
Smirnov \cite{Sm} first showed that such set has positive Hausdorff dimension for Collet-Eckmann maps. More recent 
work on biaccessible points is due, among others, to Zakeri \cite{Za2} and Zdunik \cite{Zd}.
The first equality in Theorem \ref{entropy_fmla} has also been established independently by Bruin-Schleicher \cite{BS}.

A precise statement of the asymptotic similarity between $\mathcal{M}$ and Julia sets near Misiurewicz points is proven in \cite{TanL}.

\vskip 0.5cm

\noindent \textbf{Examples}

\begin{enumerate}
\item If $c = 0$, then $f^n_c(z) = z^{2^n}$ has only one lap for each $n$, hence the entropy is zero. 
Moreover, the characteristic ray is $\theta = 0$, hence $P_c$ consists of only one element and it has zero dimension.
Moreover, the Julia set is a circle and the set of rays landing on the real axis $S_c = \{0, \frac{1}{2}\}$ consists of 
two elements, hence the dimension is $0$.

\item If $c = -2$, then $f_c$ is a 2-1 surjective map from $[-2, 2]$ to itself, hence the entropy is $\log 2$. 
The Julia set is a real segment, hence all rays land on the real axis and the Hausdorff dimension of $S_c$ is $1$. 
Similarly, the set of rays $P_c$ is the set of all parameter rays which land on the real axis, which has Hausdorff dimension $1$.

\item The \emph{basilica map} $f_c(z) = z^2 -1$ has a superattracting cycle of period $2$, and for each $n$, $f_c^n$ has $2n+1$ critical 
points, hence the entropy is $\lim_{n \to \infty} \frac{\log (2n+1)}{n} = 0$.
The angles of rays landing on the Hubbard tree are $\theta = \frac{1}{3}, \frac{2}{3}$, 
and the set of rays landing on the real Julia set is countable, hence it has dimension $0$. 
In parameter space, the only rays which land on the real axis to the right of $c = -1$ are $\theta = 0, 1/3, 2/3$, hence their dimension 
is still zero.

\item The \emph{airplane map} has a superattracting cycle of period $3$, and its characteristic angle is $\theta_c = \frac{3}{7}$. 
The set of angles whose rays land on the Hubbard tree is the set of binary numbers with expansion which does not contain any sequence of 
three consecutive equal symbols. It is a Cantor set which can be generated by the automaton in Figure \ref{automa}, 
and its Hausdorff dimension is $\log_2 \frac{\sqrt{5}+1}{2}$.

\begin{figure}[ht]
\centering
\begin{minipage}{0.45\textwidth}
\fbox{
\includegraphics[scale=1.8]{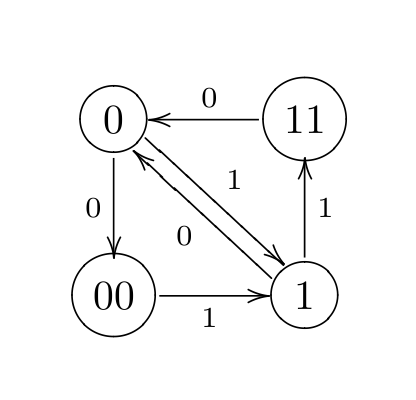}
}
\end{minipage}
\hfill
\begin{minipage}{0.45\textwidth}
\fbox{\includegraphics[scale=0.45]{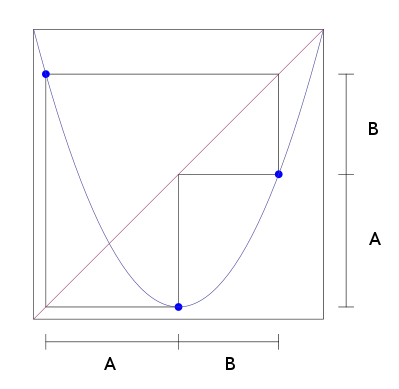}}
\end{minipage}
\caption{To the right: the combinatorics of the airplane map of period $3$. To the left: the automaton which 
produces all symbolic orbits of points on the real slice of the Julia set.}\label{automa}
\end{figure}

On the other hand, the topological dynamics of the real map is encoded by the right-hand side diagram: the interval $A$ is mapped onto 
$A \cup B$, and $B$ is mapped onto $A$. Then the number of laps of $f_c^n$ is given by the Fibonacci numbers, hence the topological 
entropy is the logarithm of the golden mean. It is harder to characterize explicitly the set of parameter rays which land on 
the boundary of $\mathcal{M}$ to the right of the characteristic ray: however, as a consequence of Theorem \ref{entropy_fmla}, the dimension of
 such set is also $\log_2 \frac{\sqrt{5}+1}{2}$.

\end{enumerate}

A more complicated example is the \emph{Feigenbaum point} $c_{Feig}$, the accumulation point of the period doubling cascades.
As a corollary of Theorem \ref{entropy_fmla}, we are able to answer a question of Zakeri (\cite{Za}, Remark 6.9):

\begin{corollary} The set of biaccessible angles for the Feigenbaum parameter $c_{Feig}$ has dimension zero:
$$\textup{H.dim }B_{c_{Feig}} = 0.$$
\end{corollary}

\subsection{The complex case}

The result of Theorem \ref{realeq} lends itself to a natural generalization for complex quadratic polynomials, 
which we will now describe.

In the real case, we related the entropy of the restriction of $f_c$ on an invariant interval 
to the Hausdorff dimension of a certain set of angles of external rays landing on the real slice 
of the Mandelbrot set. 

In the case of complex quadratic polynomials, the real axis is no longer invariant, but 
we can replace it with the Hubbard tree $T_c$ (see section \ref{section:Htree}). In particular, we define the polynomial $f_c$ 
to be \emph{topologically finite} if the Julia set is connected and locally connected
and the Hubbard tree is homeomorphic to a finite tree (see Figure \ref{vein}, left). We thus define 
the entropy $h_{top}(f_c\mid_{T_c})$ of the restriction of $f_c$ to the Hubbard tree, and 
we want to compare it to the Hausdorff dimension of some subset of parameter space.
Let $H_c$ be the set of external rays which land on $T_c$.

\begin{figure}[h!]
\centering

\begin{minipage}{0.44 \textwidth}
\fbox{ \includegraphics[scale=0.13]{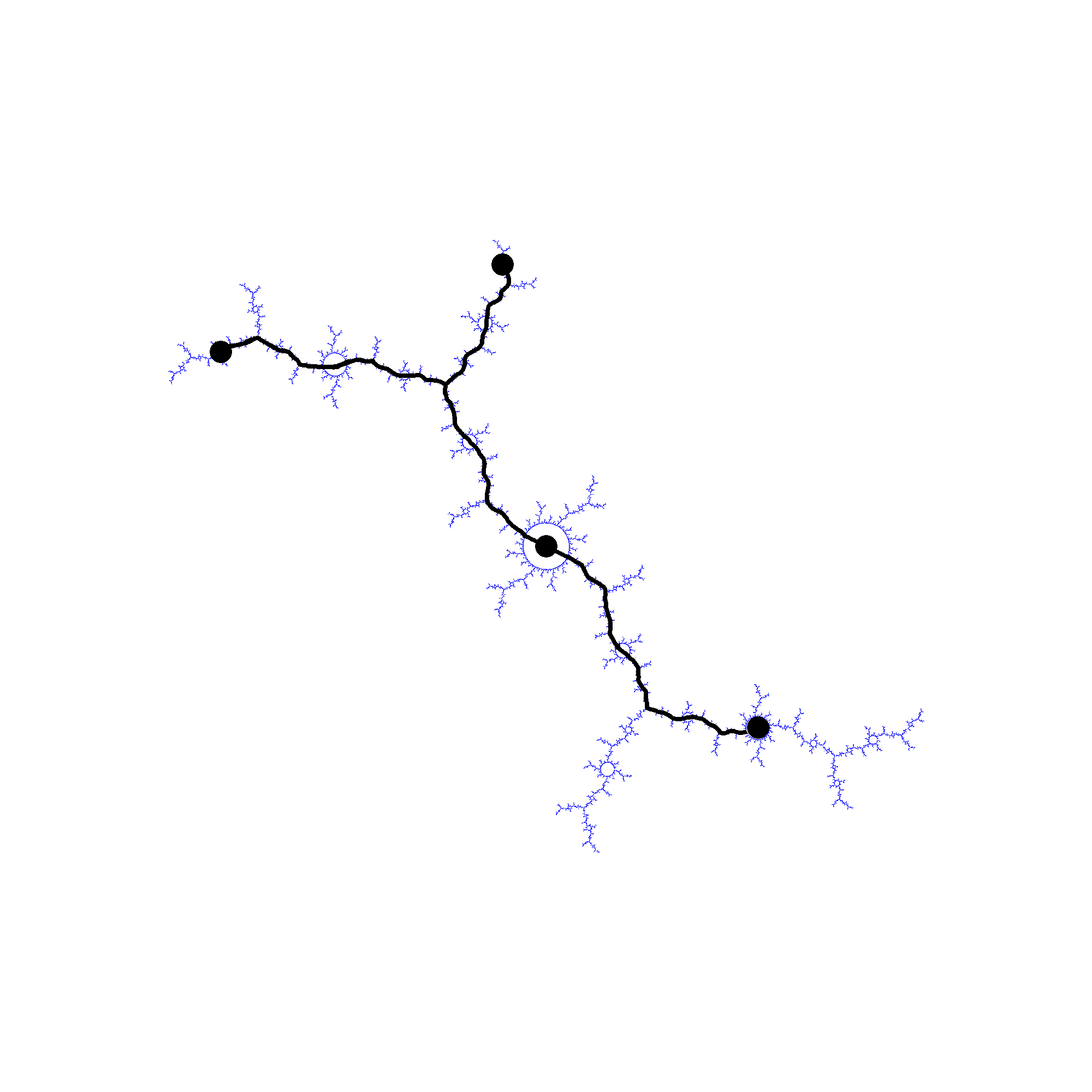}}
\end{minipage}
\hfill
\begin{minipage}{0.44 \textwidth}
\fbox{ \includegraphics[scale=0.13]{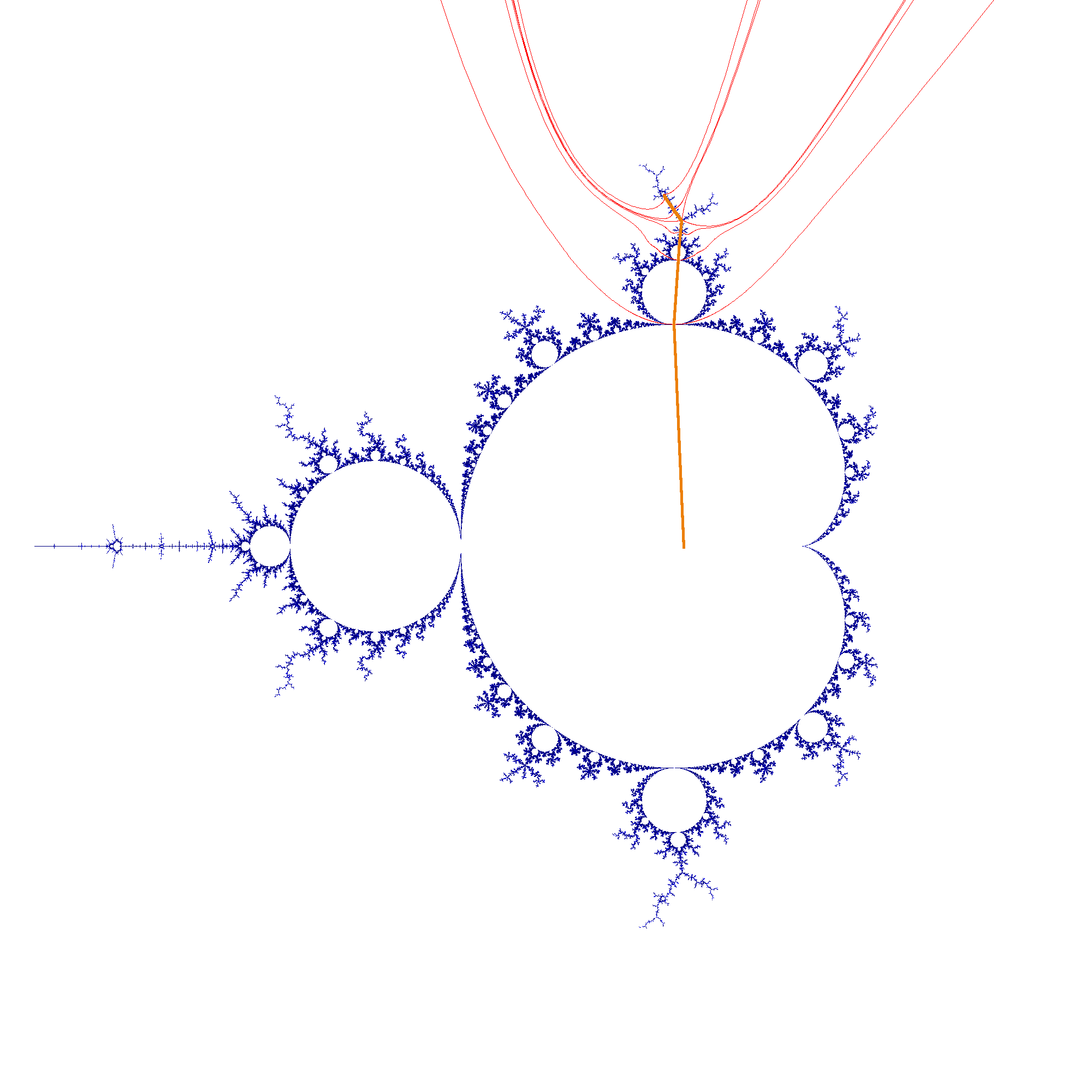}}
\end{minipage}

\caption{To the left: the Hubbard tree of the complex polynomial of period $4$ and characteristic 
angles $\theta = 3/15, 4/15$. To the right: the vein joining the center of the main cardioid with the main antenna in the $1/3$-limb ($\theta = 1/4$), 
and external rays landing on it. }
\label{vein} 
\end{figure}


In parameter space, a generalization of the real slice is a \emph{vein}: a vein $v$ is an embedded arc in 
$\mathcal{M}$, joining a parameter $c \in \partial \mathcal{M}$ with the center of the main cardioid. Given a vein $v$
and a parameter $c$ on $v$, we can define the set $P_c$  as the set of external angles of rays which land 
on $v$ closer than $c$ to the main cardioid:
$$P_c := \{ \theta \in \mathbb{R}/\mathbb{Z} \ : \ R_M(\theta) \textup{ lands on } v \cap [0, c] \}$$ 
where $[0, c]$ means the segment of vein joining $c$ to the center of the main cardioid (see Figure \ref{vein}, right).

Note that the set of topologically finite parameters contain the postcritically finite ones but it is 
much larger: indeed, every parameter $c \in \partial \mathcal{M}$ which is biaccessible (i.e. it belongs 
to some vein) is topologically finite (see section \ref{section:Htree}).

In the $\frac{p}{q}$-limb, there is a unique parameter $c_{p/q}$ such that the critical point lands on the $\beta$ fixed point 
after $q$ iterates (i.e. $f^q(0) = \beta$). The vein $v_{p/q}$ joining $c_{p/q}$ to $c = 0$ will be called the \emph{principal vein} of angle 
$p/q$. Note that $v_{1/2}$ is the real axis, while $v_{1/3}$ is the vein constructed by Branner and Douady \cite{BD}.
We can now extend the result of Theorem \ref{realeq} to principal veins:

\begin{theorem} \label{mainvein} \label{complexeq}
Let $v = v_{p/q}$ be principal vein in the Mandelbrot set, and $c \in v \cap \partial \mathcal{M}$ a parameter
along the vein. Then we have the equality
$$\frac{h_{top}(f_c \mid_{T_c})}{\log 2} = \textup{H.dim } H_c = \textup{H.dim } P_c.$$ 
\end{theorem}
We conjecture that the previous equality holds along any vein $v$.
Note that the statement can be given in more symmetric terms in the following way. If 
one defines for each $A \subseteq \mathcal{M}$, 
$$\Theta_M(A) := \{ \theta \in S^1 \ : \ R_M(\theta) \textup{ lands on }A \}$$
and similarly, for each $A \subseteq J(f_c)$, the set 
$$\Theta_c(A) := \{ \theta \in S^1 \ : \ R_c(\theta) \textup{ lands on }A \}$$
where $R_c(\theta)$ is the external ray at angle $\theta$ in the dynamical plane
for $f_c$, then Theorem \ref{mainvein} is equivalent to the statement
$$\textup{H.dim }\Theta_c([0, c]) = \textup{H.dim }\Theta_M([0, c]).$$

\subsection{Pseudocenters and real hyperbolic windows}

The techniques we use in the proof rely on the combinatorial analysis of the symbolic dynamics, 
and many ideas come from a connection with the dynamics of continued fractions. 
Indeed, on a combinatorial level the structure of the real slice of the Mandelbrot set 
is isomorphic to the structure of the bifurcation set $\mathcal{E}$ for continued fractions \cite{BCIT},
so we can use the combinatorial tools we developed in that case (\cite{CT}, \cite{CT2}) to analyze the quadratic family. 

For instance, in \cite{CT}, a key concept is the \emph{pseudocenter} of an interval, namely the (unique!) 
rational number with the smallest denominator.
When translated to the world of binary expansions, used to describe the parameter space of quadratic polynomials, the 
definition becomes

\begin{definition}
The \emph{pseudocenter} of a real interval $[a, b]$ with $|a - b| < 1$ is the unique dyadic rational 
number with shortest binary expansion.  
\end{definition}

E.g., the pseudocenter of the interval $[\frac{13}{15}, \frac{14}{15}]$ is $\frac{7}{8} = 0.111$, 
since $\frac{13}{15} = 0.\overline{1101}$ and $\frac{14}{15} = 0.\overline{1110}$. 
Recall that a \emph{hyperbolic component} $W \subseteq \mathcal{M}$ is a connected, open subset of parameters $c$ 
for which the critical point of $f_c$ is attracted to a periodic cycle. If $W$ intersects the real axis, 
we define the \emph{hyperbolic window} 
associated to $W$ to be the interval $(\theta_2, \theta_1) \subseteq [0, 1/2]$, where the rays 
$R_M(\theta_1)$ and $R_M(\theta_2)$ land on $\partial W \cap \mathbb{R}$. 

By translating the bisection algorithm of (\cite{CT}, section 2.4) in terms of kneading sequences, we get the 
following algorithm to generate all real hyperbolic windows (see section \ref{section:bisect}).

\begin{theorem} \label{algoint}
The set of all real hyperbolic windows in the Mandelbrot set can be generated as follows.
Let $c_1 < c_2$ be two real parameters on the boundary of $\mathcal{M}$, with external angles 
$0 \leq \theta_2 < \theta_1 \leq \frac{1}{2}$.
Let $\theta^*$ be the dyadic pseudocenter of the interval $(\theta_2, \theta_1)$, and let 
$$\theta^* = 0.s_1 s_2 \dots s_{n-1} s_n$$
be its binary expansion, with $s_n = 1$. Then the hyperbolic window of smallest period in the interval
 $(\theta_2, \theta_1)$ is the interval of external angles $(\alpha_2, \alpha_1)$ with 
$$\begin{array}{ccc}
\alpha_2 & := & 0.\overline{s_1 s_2 \dots s_{n-1}} \\
\alpha_1 & := & 0.\overline{s_1 s_2 \dots s_{n-1} \check{s}_1 \check{s}_2 \dots \check{s}_{n-1} }
\end{array}
$$
where $\check{s_i} := 1 - s_i$. All real hyperbolic windows are obtained by iteration of this algorithm, 
starting with $\theta_2 = 0$, $\theta_1 = 1/2$.
\end{theorem}


\subsection{Thurston's point of view} \label{section:thurston}

The results of this paper relate to recent work of W. Thurston, who looked at the entropy of Hubbard trees 
as a function of the external angle. 
Indeed, every external angle $\theta$ of the Mandelbrot set combinatorially determines a lamination (see section \ref{section:lami})
and the lamination determines an abstract Hubbard tree, of which we can compute the entropy $h(\theta)$.

Thurston produced very interesting pictures (Figure \ref{Thentro}), suggesting that the complexity of the Mandelbrot set is encoded 
in the combinatorics of the Hubbard tree, and the variation in entropy reflects the geometry of $\mathcal{M}$.

In this sense, Theorems \ref{realeq} and \ref{complexeq} contribute to this program: in fact, the entropy grows as one goes 
further from the center of $\mathcal{M}$ (see also \cite{TL}), and our results make precise the relationship
between the increase in entropy and the increased hairiness of the Mandelbrot set.  


\begin{figure} 
\centering
\includegraphics[scale=0.25]{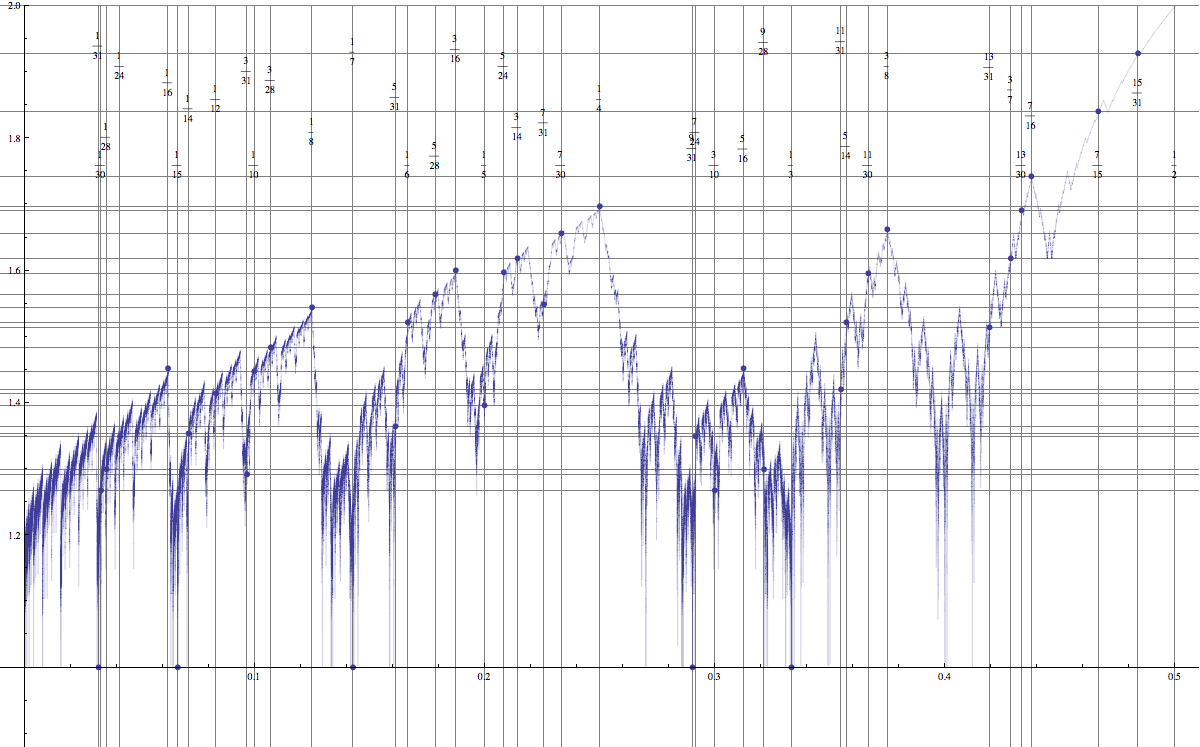}
\caption{Entropy of Hubbard trees as a function of the external angle (by W. Thurston).}
\label{Thentro}
\end{figure}

Bruin and Schleicher \cite{BS} recently proved that entropy is continuous as a function of the external angle. 

Note that Thurston's approach is in some sense dual to ours, since we look at the variation of entropy 
along the veins, i.e. from ``inside'' the Mandelbrot set as opposed to from ``outside'' as a function 
of the external angle.

We point out that the idea of the pseudocenter described in the introduction 
seems also to be fruitful to study the entropy of the Hubbard tree as a function of the external 
angle: indeed, we conjecture that the maximum of the entropy on any wake is achieved precisely at its pseudocenter.
Let us denote by $h(\theta)$ the entropy of the Hubbard tree corresponding to the parameter of external angle $\theta$. 

\begin{conjecture}
Let $\theta_1 < \theta_2$ be two external angles whose rays $R_M(\theta_1)$, $R_M(\theta_2)$ land on the same parameter 
in the boundary of the Mandelbrot set. Then the maximum of entropy on the interval $[\theta_1, \theta_2]$ is attained 
at its pseudocenter:
$$\max_{\theta \in [\theta_1, \theta_2]} h(\theta) = h(\theta^*)$$
where $\theta^*$ is the pseudocenter of the interval $[\theta_1, \theta_2]$.
\end{conjecture}

\subsection{Sketch of the argument}

The proof of Theorem \ref{entropy_fmla} is carried out in two steps. 
We first prove (Theorem \ref{entro_dim} in section \ref{section:entroH}) 
 the relationship between topological entropy $h_{top}(f_c \mid_{T_c})$ of the map restricted to the Hubbard tree and  
the Hausdorff dimension of the set $H_c$ of angles landing on the tree, for all topologically finite polynomials $f_c$.
The bulk of the argument is then proving the 
identity of Hausdorff dimensions between the real Julia set and the slices of $\mathcal{M}$:

\begin{theorem} \label{equaldim}
For any $c \in [-2, \frac{1}{4}]$, we have the equality
$$\textup{H.dim }S_c = \textup{H.dim }P_c.$$
\end{theorem}







It is not hard to show that $P_c \subseteq H_c \subseteq S_c$ for any real parameter $c$ (Corollary \ref{easyincl}); it is much harder to give a
lower bound for the dimension of $P_c$ in terms 
of the dimension of $H_c$; indeed, it seems impossible to include a copy of $H_c$ 
in $P_c$ when $c$ belongs to some \emph{tuning window}, i.e. to some baby Mandelbrot set.
However, for non-renormalizable parameters we can prove the following:

\begin{proposition} \label{embedding1}
Given a non-renormalizable, real parameter $c$ and another real parameter $c' > c$, there exists 
a piecewise linear map $F : \mathbb{R}/\mathbb{Z} \to \mathbb{R}/\mathbb{Z}$ such that 
$$F(H_{c'}) \subseteq P_{c}.$$ 
\end{proposition}


The proposition implies equality of dimension for all non-renormalizable parameters.
By applying tuning operators, we then get equality for all finitely-renormalizable parameters, 
which are dense hence the result follows from continuity.

Proposition \ref{embedding1} will be proved in section \ref{section:embed}. Its proof relies on the definition of a class of parameters, 
which we call \emph{dominant}, which are a subset of the set of non-renormalizable parameters.
We will show that for these parameters (which can be defined purely combinatorially) it is easier to 
construct an inclusion of the Hubbard tree into parameter space; finally, the most technical part (section \ref{proofdensity})
will be proving that such parameters are dense in the set of non-renormalizable angles.

In order to establish the result for complex veins, we first prove continuity of entropy along veins by a version 
of kneading theory for Hubbard trees (section \ref{section:knead}). Finally, 
we transfer the inclusion of Proposition \ref{embedding1} 
from the real vein to the other principal vein via a combinatorial version of the Branner-Douady surgery
(section \ref{section:surgery}). 

\subsection{Remarks and acknowledgements} 
The history of this paper is quite interesting. 
After the discovery of the connection between continued fractions and the real slice of $\mathcal{M}$ \cite{BCIT}, 
the statement for the real case (Theorem \ref{realeq}) came out of discussions with Carlo Carminati in spring 
2011, as an application of our 
combinatorial techniques (indeed, modulo translation to the complex dynamics language, the essential arguments
are contained in \cite{CT2}). At about the same time, I have been informed of the recent work of W. Thurston on the entropy of 
Hubbard trees, which sparked new interest and inspired the generalization to complex veins. 

I especially wish to thank 
A.M. Benini, Tan Lei, and C.T. McMullen for useful conversations, and D. Schleicher for pointing out reference \cite{Ri}.
Some of the pictures have been created with the software \begin{ttfamily}mandel\end{ttfamily} of W. Jung.

\section{External rays} \label{section:extrays}

Let $f(z)$ be a monic polynomial of degree $d$. Recall that the \emph{filled Julia set} $K(f)$ is the set of 
points which do not escape to infinity under iteration: 
$$K(f) := \{z \in \mathbb{C} \  : \ f^n(z) \textup{ does not tend to }\infty \textup{ as }n \to \infty \}.$$
The \emph{Julia set} $J(f)$ is the boundary of $K(f)$. 
If $K(f)$ is connected, then the complement of $K(f)$ in the Riemann sphere is simply connected, so it can 
be uniformized by the Riemann mapping 
$\Phi : \hat{\mathbb{C}} \setminus \overline{\mathbb{D}} \to \hat{\mathbb{C}}\setminus K(f)$
 which maps the exterior of the closed unit disk $\overline{\mathbb{D}}$ to the exterior of $K(f)$.
The Riemann mapping is unique once we impose $\Phi(\infty) = \infty$ and $\Phi'(\infty) = 1$.
With this choice, $\Phi$ conjugates the action of $f$ on the exterior of the filled Julia set
to the map $z \mapsto z^d$, i.e. 
\begin{equation} \label{semiconj1}
f(\Phi(z)) = \Phi(z^d).
\end{equation}
By Carath\'eodory's theorem (see e.g. \cite{Po}), the Riemann mapping extends to a continuous map 
$\overline{\Phi}$ on the boundary 
$\overline{\Phi} : \hat{\mathbb{C}} \setminus \mathbb{D} \to \hat{\mathbb{C}}\setminus \textup{int }K(f) $
if and only if the Julia set is locally connected. If this is the case, the restriction 
of $\overline{\Phi}$ to the boundary is sometimes called the \emph{Carath\'eodory loop} and it will be denoted as 
$$\gamma : \mathbb{R}/\mathbb{Z} \to J(f).$$
As a consequence of the eq. \eqref{semiconj1}, the action of $f$ on the set of angles is 
semiconjugate to multiplication by $d$ (mod $1$):
\begin{equation}
\gamma(d \cdot \theta) = f(\gamma(\theta)) \qquad \textup{for each }\theta \in \mathbb{R}/\mathbb{Z}. 
\end{equation}

In the following we will only deal with the case of quadratic polynomials of the form 
$f_c(z) := z^2 + c$, so $d = 2$ and we will denote as
$$D(\theta) := 2 \cdot \theta \mod 1 $$
the doubling map of the circle. Moreover, we will add the subscript $c$ when we need to make the dependence on the polynomial $f_c$
more explicit.  
Given $\theta \in \mathbb{R}/\mathbb{Z}$, the \emph{external ray} $R_c(\theta)$ is the image of the radial 
arc at angle $2 \pi \theta$ via the Riemann mapping $\Phi_c : \hat{\mathbb{C}} \setminus \overline{\mathbb{D}} \to \hat{\mathbb{C}} 
\setminus K(f_c)$:
$$R_c(\theta) := \{ \Phi_c( \rho e^{2 \pi i \theta} )\}_{\rho > 1}.$$
The ray $R_c(\theta)$ is said to \emph{land} at $x$ if 
$$\lim_{\rho \to 1^+} \Phi_c( \rho e^{2 \pi i \theta} ) = x.$$
If the Julia set is locally connected, then all rays land; in general, by Fatou's theorem, the set of angles for 
which $R_c(\theta)$ does not land has zero Lebesgue measure, and indeed it also has zero capacity and hence zero 
Hausdorff dimension (see e.g. \cite{Po}). It is however known that there exist non-locally connected Julia sets for 
polynomials \cite{MiLC}.
The ray $R_c(0)$ always lands on a fixed point of $f_c$ which is traditionally called the $\beta$ fixed point and denoted as 
$\beta$. The other fixed point of $f_c$ is called the $\alpha$ fixed point. Note that in the case $c = \frac{1}{4}$ 
one has $\alpha = \beta$. Finally, the critical point of $f_c$ will be denoted by $0$, and the critical value 
by $c$.

Analogously to the Julia sets, the exterior of the Mandelbrot set can be uniformized by the Riemann mapping 
$$\Phi_M : \Hat{\mathbb{C}} \setminus \overline{\mathbb{D}} \rightarrow \hat{\mathbb{C}} \setminus \mathcal{M}$$
with $\Phi_M(\infty) = \infty$, and $\Phi'(\infty) = 1$, and images of radial arcs are called \emph{external rays}. 
Every angle $\theta \in \mathbb{R}/\mathbb{Z}$ determines an external ray
$$R_M(\theta) := \Phi_M(\{ \rho e^{2\pi i \theta} : \rho >  1 \})$$ 
which is said to \emph{land} at $x$ if the limit $\lim_{\rho \to 1^+} \Phi_M(\rho e^{2 \pi i \theta})$ exists. 
According to the \emph{MLC conjecture} \cite{DH}, the Mandelbrot set is locally connected, and therefore all rays land on some point of the 
boundary of $\mathcal{M}$.

\subsection{Biaccessibility and regulated arcs}

A point $z \in J(f_c)$ is called \emph{accessible} if it is the landing point of at least one external ray. 
It is called \emph{biaccessible} if it is the landing point of at least two rays, i.e. there exist $\theta_1, 
\theta_2$ two distinct angles such that $R_c(\theta_1)$ and $R_c(\theta_2)$ both land at $z$. This is 
equivalent to say that $J(f_c) \setminus \{z \}$ is disconnected.


Let $K = K(f_c)$ be the filled Julia set of $f_c$. Assume $K$ is connected and locally connected. Then it is also path-connected 
(see e.g. \cite{Wi}, Chapter 8), so given any two points $x, y$ in $K$, there exists an arc in $K$ with endpoints $x, y$.

If $K$ has no interior, then the arc is uniquely determined by its endpoints $x$, $y$. Let us now describe how 
to choose a canonical representative inside the Fatou components in the case $K$ has interior.  
In this case, each bounded Fatou component eventually maps to a periodic Fatou component, which either contains an attracting cycle, 
or it contains a parabolic cycle on its boundary, or it is a periodic Siegel disk. 

Since we will not deal with the Siegel disk case in the rest of the paper, let us assume we are in one of the first two cases.
Then there exists a Fatou component $U_0$ which contains the critical point, and a biholomorphism $\phi_0 : U_0 \to \mathbb{D}$
to the unit disk mapping the critical point to $0$. The preimages $\phi_0^{-1}(\{\rho e^{2 \pi i \theta} \ : \ 0 \leq \rho < 1 \})$ of radial arcs in the unit disk 
are called \emph{radial arcs} in $U_0$. Any other bounded Fatou component $U$ is eventually mapped to $U_0$; let 
$k \geq 0$ be the smallest integer  such that $f_c^k(U) = U_0$. Then the map $\phi := \phi_0 \circ f_c^k$ is a biholomorphism of $U$ 
onto the unit disk, and we define radial arcs to be preimages under $\phi$ of radial arcs in the unit disk.   

An embedded arc $I$ in $K$ is called \emph{regulated} (or \emph{legal} in Douady's terminology \cite{DoCompact}) if the intersection between $I$ 
and the closure of any bounded Fatou component is contained in the union of at most two radial arcs.
With this choice, given any two points $x, y$ in $K$, there exists a unique regulated arc in $K$ with endpoints $x, y$ 
(\cite{Za-biacc}, Lemma 1). Such an arc will be denoted by $[x, y]$, and the corresponding open arc by 
$(x, y) := [x, y] \setminus \{x, y\}$.
A \emph{regulated tree} inside $K$ is a finite tree whose edges are regulated arcs. Note that, in the case $K$ has 
non-empty interior, regulated trees as defined need not be invariant for the dynamics, because $f_c$ need not map radial arcs 
to radial arcs. However, by construction, radial arcs in any bounded Fatou component $U$ different from $U_0$ map 
to radial arcs in $f_c(U)$. In order to deal with $U_0$, we need one further hypothesis. Namely, we will assume that $f_c$ 
has an attracting or parabolic cycle of period $p$ with real multiplier. Then we can find a parametrization $\phi_0 : U_0 \to \mathbb{D}$ 
such that the interval $I := \phi_0^{-1}((-1, 1))$ is preserved by the $p$-th iterate of $f_c$, i.e. $f_c^p(I) \subseteq I$. The interval 
$I$ will be called the \emph{bisector} of $U_0$.
Now note that, if the regulated arc $[x, y]$ does not contain $0$ in its interior and it only intersects 
 the critical Fatou component $U_0$ in its bisector, then we have 
$$f_c([x, y]) = [f_c(x), f_c(y)].$$
The \emph{spine} of $f_c$ is the regulated arc $[-\beta, \beta]$ joining the $\beta$ fixed point 
to its preimage $-\beta$. The biaccessible points are related to the points which lie on the spine by the following lemma.

\begin{lemma}  \label{biacc_lemma}
Let $f_c(z) = z^2 + c$ be a quadratic polynomial whose Julia set is connected and locally connected. Then the set of biaccessible 
points is 
$$\mathcal{B}  = J(f_c) \cap \bigcup_{n \geq 0} f_c^{-n}((-\beta, \beta)).$$
\end{lemma}

\begin{proof}
Let $f = f_c$, and $x \in J(f) \cap (-\beta, \beta)$. The set $V := R_c(0) \cup [-\beta, \beta] \cup R_c(1/2)$ disconnects the plane
in two parts, $\mathbb{C} \setminus V = A_1 \cup A_2$. We claim that $x$ is the limit of points in the basin of infinity $U_\infty$ on 
both sides of $V$, i.e. for each $i = 1, 2$ there exists a sequence $\{x_n\}_{n \in \mathbb{N}} \subseteq A_i \cap U_\infty$ with $x_n \to x$; since the Riemann 
mapping $\Phi$ extends continuously to the boundary, this is enough to prove that there exist two external angles $\theta_1 \in (0, 1/2)$
and $\theta_2 \in (1/2, 1)$ such that $R_c(\theta_1)$ and $R_c(\theta_2)$ both land on $x$. Let us now prove the claim; if it is not true, 
then there exists an open neighborhood $\Omega$ of $x$ and an index $i\in\{1, 2\}$ such that $\Omega \cap A_i$ is connected and contained in the 
interior of the filled Julia set $K(f)$, hence $\Omega \cap A_i$ is contained in some bounded Fatou component. This implies that 
$\Omega \cap V$ lies in the closure of a bounded Fatou component, and $x$ on its boundary. However, this contradicts the definition 
of regulated arc, because if $U$ is a bounded Fatou component intersecting a regulated arc $I$, then $\partial U \cap I$ 
does not disconnect $\overline{U} \cap I$.
Suppose now that $x \in J(f)$ is such that 
 $f^n(x)$ belongs to $(-\beta, \beta)$ for some $n$. Then by the previous argument $f^n(x)$ is biaccessible, and since $f$ is a local 
homeomorphism outside the spine, $x$ is also biaccessible.

Conversely, suppose $x$ is biaccessible, and the two rays at angles $\theta_1$ and $\theta_2$ land on $x$, with $0 < \theta_1 < \theta_2 < 1$. 
Then there exists some $n$ for which $1/2 \leq D^n(\theta_2) - D^n(\theta_1) < 1$, hence $R_c(D^n(\theta_1))$ and $R_c(D^n(\theta_2))$ 
must lie on opposite sides with respect to the spine, and since they both land on $f^n(x)$, then $f^n(x)$ belongs to the spine. 
Since the point $\beta$ is not biaccessible (\cite{Mc}, Theorem 6.10), $f^n(x)$ must belong to $(-\beta, \beta)$. 
\end{proof}

\begin{lemma}
We have that $\alpha \in [0, c]$. 
\end{lemma}

\begin{proof}
Indeed, since $\alpha \in (-\beta, 0)$ (\cite{Za-biacc}, Lemma 5), we have 
$-\alpha \in (\beta, 0)$ and $\alpha = f(-\alpha) \in (\beta, c)$. Thus, since $0 \in (\alpha, \beta)$ 
we have  $\alpha \in (0, c)$.
\end{proof}
 
\begin{lemma} \label{increase}
For $x \in [0, \beta)$, we have $x \in (f(x), \beta)$.
\end{lemma}

\begin{proof}
Let us consider the set $S = \{ x \in [0, \beta] \ : \ x \in (\beta, f(x)) \}$. The set is open 
by continuity of $f$. Since the $\beta$ fixed point is repelling, the set $S$ contains points in a neighborhood of $\beta$, 
so it is not empty. Suppose $S \neq [0, \beta)$ and let $x \in \partial S$, $x \neq \beta$. By continuity of $f$, $x$ must be a fixed point 
of $f$, but the only fixed point of $f$ in the arc is $\beta$.
\end{proof}

For more general properties of biaccessibility we refer to \cite{Za-biacc}.

\section{Laminations} \label{section:lami}

A powerful tool to construct topological models of Julia sets and the Mandelbrot set is 
given by laminations, following Thurston's approach. As we will see, laminations represent equivalence 
relations on the boundary of the disk arising from external rays which land on the same point.
We now give the basic definitions, and refer to \cite{ThLam} for further details.

A \emph{geodesic lamination} $\lambda$ is a set of hyperbolic geodesics in the closed unit disk $\overline{\mathbb{D}}$, called 
the \emph{leaves} of $\lambda$, such that no two leaves intersect in $\mathbb{D}$, and the union of all leaves 
is closed. 

A \emph{gap} of a lamination $\lambda$ is the closure of a component of the complement of the union of all leaves.
In order to represent Julia sets of quadratic polynomials, we need to restrict ourselves to \emph{invariant 
laminations}.

Let $d \geq 2$. The map $g(z) := z^d$ acts on the boundary of the unit disk, hence it induces a dynamics on the set of leaves.
Namely, the image of a leaf $\overline{pq}$ is defined as the leaf joining the images of the endpoint: 
$g(\overline{pq}) = \overline{g(p)g(q)}$. A lamination $\lambda$ is \emph{forward invariant} 
if the image of any leaf $L$ of $\lambda$ still belongs to $\lambda$. Note that the image leaf may be \emph{degenerate}, 
i.e. consist of a single point on the boundary of the disk.
 
A lamination is \emph{invariant} if in addition to being forward invariant it satisfies the additional conditions:
\begin{itemize}
 \item \emph{Backward invariance}: if $\overline{pq}$ is in $\lambda$, then there exists a collection
of $d$ disjoint leaves in $\lambda$, each joining a preimage of $p$ to a preimage
of $q$.
\item 
\emph{Gap invariance}: for any gap $G$, the hyperbolic convex hull of the image of
$G_0 = \overline{G} \cap S^1$ is either a gap, a leaf, or a single point.
\end{itemize}

In this paper we will only deal with quadratic polynomials, so $d = 2$ and the invariant laminations for 
the map $g(z) = z^2$ will be called \emph{invariant quadratic laminations}. 
A leaf of maximal length in a  lamination is called a \emph{major leaf}, and its image a \emph{minor leaf}.
Typically, a quadratic invariant lamination has $2$ major leaves, but the minor leaf is always unique.

If $J(f_c)$ is a Julia set of a quadratic polynomial, one can define the equivalence 
relation $\sim_c$ on the unit circle $\partial \mathbb{D}$ by saying that $\theta_1 \sim_c \theta_2$ if the rays $R_c(\theta_1)$ and
$R_c(\theta_2)$ land on the same point. 

From the equivalence relation $\sim_c$ one can construct a quadratic invariant lamination in the following way.
Let $E$ be an equivalence class for $\sim_c$. If $E = \{ \theta_1, \theta_2 \}$ contains two elements, then we define the leaf $L_E$
as $L_E := (\theta_1, \theta_2)$. If $E = \{\theta\}$ is a singleton, then we define $L_E$ to be the degenerate leaf $L_E:= \{\theta\}$. 
Finally, if $E = \{\theta_1, \dots, \theta_k\}$ contains more than two elements, 
with $0 \leq \theta_1 < \theta_2 < \dots < \theta_k < 1$, then we define $L_E$ to be the 
union of the leaves $L_E := (\theta_1, \theta_2) \cup (\theta_2, \theta_3) \cup \dots \cup (\theta_k, \theta_1)$. Finally, we let 
the associated lamination $\lambda_c$ be
$$\lambda_c := \bigcup_{E \textup{ equiv. class of }\sim_c} L_E.$$ 
The lamination $\lambda_c$ is an invariant quadratic lamination.
The equivalence relation $\sim_c$ can be extended to a relation $\cong_c$ on the closed disk 
$\overline{\mathbb{D}}$ by taking convex hulls, and 
the quotient of the disk by $\cong_c$ is a model for the Julia set:

\begin{theorem}[\cite{DoCompact}]
If the Julia set $J(f_c)$ is connected and locally connected, then it is homeomorphic to the quotient
of $\overline{\mathbb{D}}$ by the equivalence relation $\cong_c$. 
\end{theorem}

We define the the \emph{characteristic leaf} of a quadratic polynomial $f_c$ with Julia set connected and locally connected
to be the minor leaf of the invariant lamination $\lambda_c$. The endpoints of the characteristic leaf are 
called \emph{characteristic angles}.

\subsection{The abstract Mandelbrot set} \label{section:Mabs}

In order to construct a model for the Mandelbrot set, Thurston \cite{ThLam} defined the \emph{quadratic minor lamination} $QML$ 
as the union of the minor leaves of all quadratic invariant laminations (see Figure \ref{qml}). 

\begin{figure}[h!]  
\includegraphics[scale=0.5]{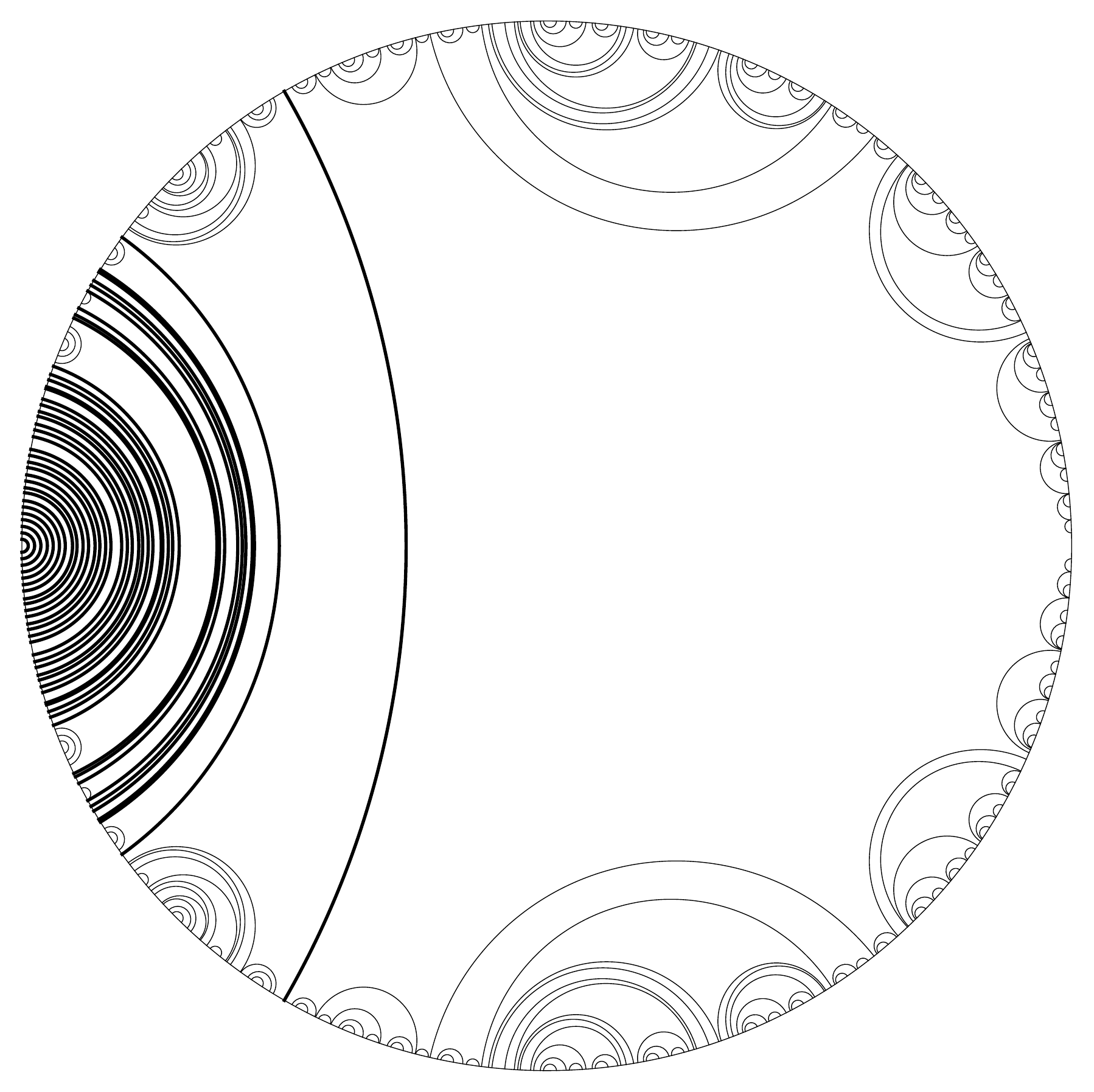}
\caption{Thurston's \emph{quadratic minor lamination}. The quotient of the unit disk by the equivalence relation 
given by the lamination is a topological model for $\mathcal{M}$. Leaves which are symmetric with respect
to complex conjugation (displayed thicker) correspond to rays landing on the real axis. }
\label{qml}
\end{figure}

As in the Julia set case, the lamination determines an equivalence relation $\cong_M$ on $\overline{\mathbb{D}}$ by identifying 
points on the same leaf, and also points in the interior of finite ideal polygons whose sides are leaves. The quotient 
$$\mathcal{M}_{abs} := \overline{\mathbb{D}}/\cong_M$$
is called \emph{abstract Mandelbrot set}. It is a compact, connected and locally connected space.
Douady \cite{DoCompact} constructed a continuous surjection 
$$\pi_M : \mathcal{M} \to \mathcal{M}_{abs}$$
which is injective if and only if $\mathcal{M}$ is locally connected.

The idea behind the construction is that leaves of $QML$ connect external angles whose corresponding rays in parameter space land on the 
same point. However, since we do not know whether $\mathcal{M}$ is locally connected, additional care is required. 
Indeed, let $\sim_M$ denote the equivalence relation on $\partial \mathbb{D}$ induced by the lamination $QML$, and   
$\theta_1 \asymp_M \theta_2$ denote that the external rays $R_M(\theta_1)$ and $R_M(\theta_2)$ land on the same point. 
The following theorem summarizes a few key results comparing the analytic and combinatorial models of the Mandelbrot set:

\begin{theorem} \label{yoccoz}
Let $\theta_1, \theta_2 \in \mathbb{R}/\mathbb{Z}$ be two angles. Then the following are true:
\begin{enumerate}
\item if $\theta_1 \asymp_M \theta_2$, then $\theta_1 \sim_M \theta_2$; 
\item if $\theta_1 \sim_M \theta_2$ and $\theta_1, \theta_2$ are rational, then 
$\theta_1 \asymp_M \theta_2$;
\item if $\theta_1 \sim_M \theta_2$ and $\theta_1, \theta_2$ are not infinitely renormalizable, then 
$\theta_1 \asymp_M \theta_2$.
\end{enumerate}
\end{theorem}

\begin{proof}
(1) and (2) are contained in (\cite{ThLam}, Theorem A.3). (3) follows from Yoccoz's theorem on landing of rays 
at finitely renormalizable parameters (see \cite{Hu} for the proof). Indeed, Yoccoz proves that external rays 
$R_M(\theta)$ with non-infinitely renormalizable combinatorics land, and moreover that the intersections of 
nested parapuzzle pieces contain a single point. Along the boundary of each puzzle piece lie pairs of external rays with rational 
angles (see also \cite{Hu}, sections 5 and 12) which land on the same point, and since the intersection of the nested 
sequence of puzzle pieces is a single point $c \in \partial \mathcal{M}$, the rays $\theta_1$ and $\theta_2$ land on the same point $c$.
\end{proof}

The following criterion makes it possible to check whether a leaf belongs to the quadratic 
minor lamination by looking at its dynamics under the doubling map:

\begin{proposition}[\cite{ThLam}] \label{critQML}
A leaf $m$ is the minor leaf of some invariant quadratic lamination (i.e. it belongs to $QML$) if and only if 
the following three conditions are met:
\begin{itemize}
\item[(a)] all forward images of $m$ have disjoint interiors; 
\item[(b)] the length of any forward image of $m$ is never less than the length of $m$;
\item[(c)] if $m$ is a non-degenerate leaf, then $m$ and all leaves on the forward orbit of $m$ are disjoint 
from the interiors of the two preimage leaves of $m$ of length at least $1/3$.
\end{itemize}
\end{proposition}

For the rest of the paper we shall work with the abstract, locally connected model of $\mathcal{M}$ and study its dimension 
using combinatorial techniques; only at the very end (Proposition \ref{combvsanal}) we shall compare 
the analytical and combinatorial models and prove that our results hold for the actual Mandelbrot set even without assuming 
the MLC conjecture.

\section{Hubbard trees} \label{section:Htree}

Assume now that the polynomial $f = f_c(z) = z^2 + c$ has connected Julia set (i.e. $c\in \mathcal{M}$), and 
no attracting fixed point (i.e. $c$ lies outside the main cardioid). The \emph{critical orbit} of $f$ is the 
set $Crit(f) := \{ f^k(0) \}_{k \geq 0}$. Let us now give the fundamental

\begin{definition}
The \emph{Hubbard tree} $T$ for $f$ is the smallest regulated tree which contains
the critical orbit, i.e. 
$$T := \bigcup_{i, j \geq 0} [f^i(0), f^j(0)].$$
\end{definition}

Note that, according to this definition, the set $T$ need not be closed in general. We shall establish a few fundamental properties of 
Hubbard trees.

\begin{figure}
\fbox{ \includegraphics[width=0.95 \textwidth]{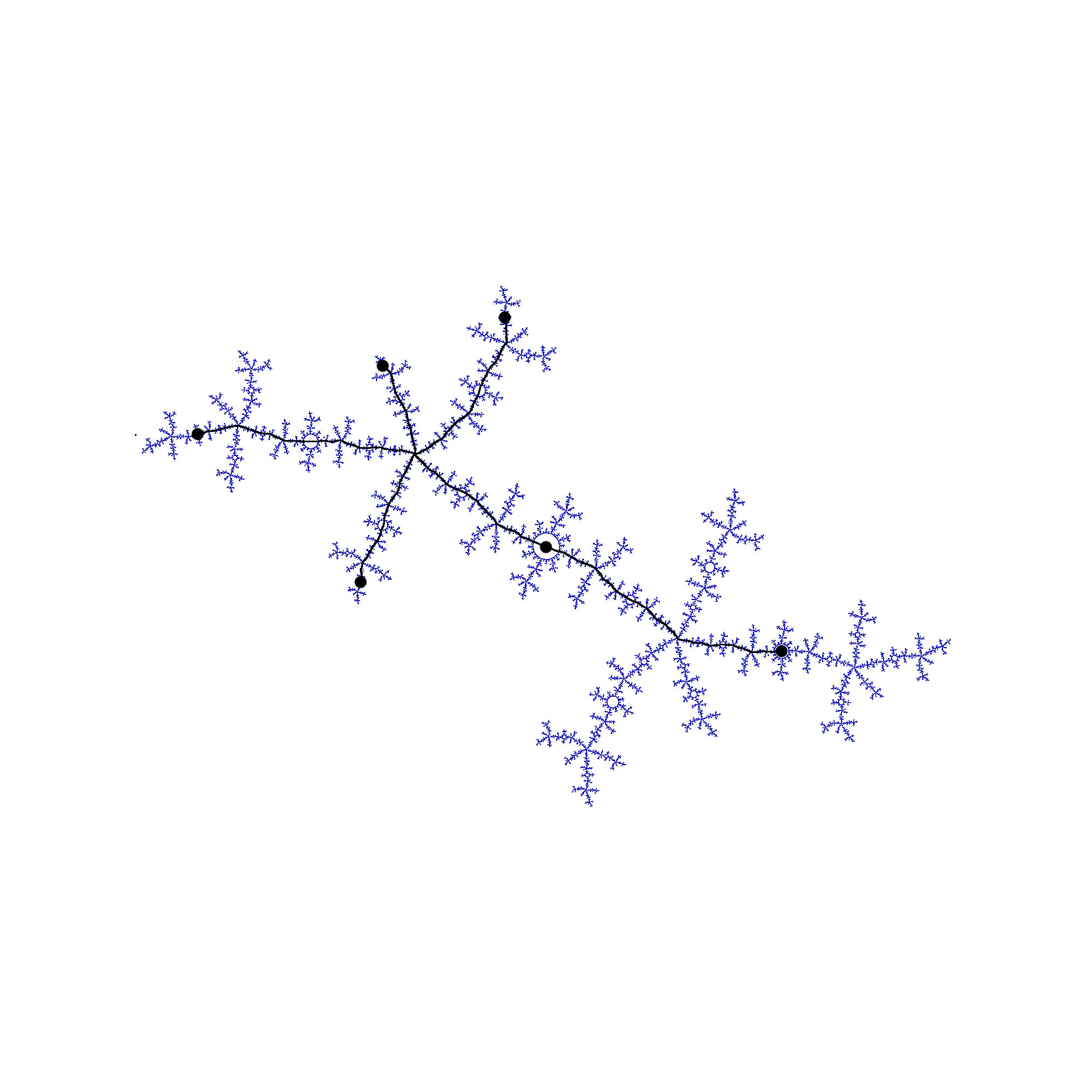}}
\caption{The Hubbard tree of the quadratic polynomial with characteristic leaf $(19/63, 20/63)$.
The map $f_c$ is postcritically finite, and the critical point belongs to a cycle of period $6$.
The parameter $c$ belongs to the principal vein in the $2/5$-limb.}
\end{figure}

\begin{lemma} \label{treebasic}
The following properties hold: 
\begin{enumerate}
 \item $T$ is the smallest forward-invariant set which contains the regulated arc $[\alpha, 0]$; 
\item $T = \bigcup_{n \geq 0}[\alpha, f^n(0)]$. 
\end{enumerate}
\end{lemma}

\begin{proof}

Let now $T_1$ be the smallest forward-invariant set which contains the regulated arc $[\alpha, 0]$. 
By definition, $T$ is forward-invariant and contains $[\alpha, 0]$ since $\alpha \in [0, c]$, so 
$T_1 \subseteq T$.
Let now
$$T_2 := \bigcup_{n \geq 0} [\alpha, f^n(0)].$$ 
Since $[f^i(0), f^j(0)] \subseteq [\alpha, f^i(0)] \cup [\alpha, f^j(0)]$, then $T \subseteq T_2$.
By definition, 
$$T_1 = \bigcup_{n \geq 0} f^n([\alpha, 0]).$$
Since $f^i([\alpha, 0]) \supseteq [\alpha, f^i(0)]$, then $T_2 \subseteq T_1$, hence $T = T_1 = T_2$.
\end{proof}

The tree thus defined need not have finitely many edges. However, in the following we will restrict ourself to the case when $T$ 
is a finite tree. Let us introduce the definition: 

\begin{definition}
A polynomial $f$ is \emph{topologically finite} if the Julia set is locally connected and 
the Hubbard tree $T$ is homeomorphic to a tree with finitely many edges.
\end{definition}

Recall that a polynomial is called \emph{postcritically finite} if the critical orbit is finite. Postcritically 
finite polynomials are also topologically finite, but it turns out that the class of topologically 
finite polynomials is much bigger and indeed it contains all polynomials along the veins 
of the Mandelbrot set (see also section \ref{section:veins}).
 
\begin{proposition} \label{top_finite}
Let $f$ have locally connected  Julia set. 
Suppose there is an integer $n \geq 1$ such that $f^n(0)$ lies on the regulated arc $[\alpha, \beta]$, 
and let $N$ be the smallest such integer. 
Then $f$ is topologically finite, and the Hubbard tree $T$ of $f$ is given by 
$$T =  \bigcup_{i = 0}^N [\alpha, f^i(0)].$$
\end{proposition}

\begin{proof}
Let $T_N := \bigcup_{i = 0}^N [\alpha, f^i(0)]$. By Lemma \ref{treebasic} (2),  $T_N \subseteq T$.
Note now that for each $i$ we have 
$$f([\alpha, f^i(0)]) \subseteq [\alpha, c] \cup [\alpha, f^{i+1}(0)]$$
thus 
$$f(T_N) \subseteq T_N \cup [\alpha, f^{N+1}(0)].$$
Now, either $f^{N}(0)$ lies in $[\alpha, -\alpha]$, or by Lemma \ref{increase}, $f^N(0)$ lies between $\beta$ and $f^{N+1}(0)$.
In the first case, $[\alpha, f^{N+1}(0)] \subseteq [\alpha, c]$ and in the second case 
$[\alpha, f^{N+1}(0)] \subseteq [\alpha, f^N(0)]$; in both cases, $[\alpha, f^{N+1}(0)] \subseteq T_N$, so 
$T_N$ is forward-invariant and it contains $[\alpha, 0]$, so it contains $T$ by Lemma \ref{treebasic} (1). 
\end{proof}



\begin{proposition}
If the Julia set of $f$ is locally connected and the critical value $c$ is biaccessible, 
then $f$ is topologically finite.
\end{proposition}

\begin{proof}
Since $c$ is biaccessible, by Lemma \ref{biacc_lemma} there exists $n \geq 0$ such that $f^n(c)$ belongs to the spine 
$[-\beta, \beta]$ of the Julia set. Then either $f^n(c)$ or $f^{n+1}(c)$ lie on $[\alpha, \beta]$, so $f$ is topologically finite by 
Proposition \ref{top_finite}. 
\end{proof}


Let us define the \emph{extended Hubbard tree} $\widetilde{T}$ to be the union of the Hubbard tree and 
the spine: 

$$\widetilde{T} := T \cup [-\beta, \beta].$$

Note the extended tree is also forward invariant, i.e. $f(\widetilde{T}) \subseteq \widetilde{T}$. Moreover, it is related to the 
usual Hubbard tree in the following way:

\begin{lemma} \label{exttree}
The extended Hubbard tree eventually maps to the Hubbard tree:
$$\widetilde{T} \setminus \{ \beta, - \beta\} \subseteq \bigcup_{n \geq 0} f^{-n}(T).$$
\end{lemma}

\begin{proof}
Since $f([\alpha, -\beta)) = [\alpha, \beta)$, we just need to check that every element $z \in [\alpha, \beta)$ 
eventually maps to the Hubbard tree. Indeed, either there exists $n \geq 0$ such that $f^n(z) \in [\alpha, c] \subseteq T$, 
or, by Lemma \ref{increase}, the sequence $\{f^n(z)\}_{z \geq 0}$ all lies on $[0, \beta)$ and it is ordered along the 
segment, i.e. for each $n$, $f^{n+1}(z)$ lies in between $0$ and $f^n(z)$. Then the sequence must have a limit point, and such
limit point would be a fixed point of $f$. However, $f$ has no fixed points on $[0, \beta)$, contradiction. 
\end{proof}

\subsection{Valence}

If $T$ is a finite tree, then the \emph{degree} of a point $x \in T$ is the 
number of connected components of $T \setminus \{ x \}$, and is denoted by $deg(x)$.
Moreover, let us denote by $deg(T)$ denote the largest degree of a point on the tree:
$$deg(T) := \max\{ deg(x) \ : \ x \in T \}.$$
On the other hand, for each $z \in J(f)$, we call \emph{valence} of $z$ the number of external rays which land on $z$ and 
denote it as 
$$val(z) := \# \{ \theta \in \mathbb{R}/\mathbb{Z} \ : \ R_c(\theta) \textup{ lands on } z \}.$$
The valence of $z$ also equals the number of connected components of $J(f) \setminus \{z\}$ 
(\cite{Mc}, Theorem 6.6), also known as the \emph{Urysohn-Menger index} of $J(f)$ at $z$.



\begin{proposition} \label{valencebound}
Let $T$ be the extended Hubbard tree for a topologically finite quadratic polynomial $f$.  
Then the number of rays $N$ landing on $x \in T$ is bounded above by
$$N \leq 2 \cdot deg(T).$$
\end{proposition}

The proposition follows easily from the 

\begin{lemma} \label{max_val}
Let $T$ be the extended Hubbard tree for $f$, and $x \in T$ a point on the tree which never maps to the critical point. 
Then the number of rays $N$ landing on $x$ is bounded above by
$$N \leq \max \{ deg(f^n(x)) \ : \ n \geq 0 \}.$$
\end{lemma}

\begin{proof}
Note that, since the forward orbit of $x$ does not contain the critical point, $f^n$ is a local homeomorphism in a neighborhood of $x$; 
thus, for each $n \geq 0$, $val(f^n(x)) = val(x)$ and $deg(f^n(x)) \geq deg(x)$. 
Suppose now the claim is false: let $N$ be such that $deg(f^N(x)) = \max \{ deg(f^n(x)) \ : \ n \geq 0 \}
 < val(x)$, and denote $y = f^N(x)$. Then there are two angles $\theta_1$, $\theta_2$ such that the rays 
$R_c(\theta_1)$ and $R_c(\theta_2)$ both land at $y$, and the sector between $R_c(\theta_1)$ and 
$R_c(\theta_2)$ does not intersect the tree. Then, there exists $M \geq 0$ such that the rays 
$R_c(D^M(\theta_1))$ and $R_c(D^M(\theta_2))$ lie on opposite sides of the spine, thus their common 
landing point $z := f^M(y)$ must lie on the spine. Moreover, since $val(z) = val(x) \geq 2$ while only 
one ray lands on the $\beta$ fixed point, $z$ must lie in the interior of the spine. This means that the 
sector between the rays $R_c(D^M(\theta_1))$ and $R_c(D^M(\theta_2))$ intersects the spine, so 
$deg(f^M(y)) > deg(y)$, contradicting the maximality of $N$.
\end{proof}

\begin{proof}[Proof of Proposition \ref{valencebound}]
If $val(x) > 0$, then $x$ lies in the Julia set $J(f_c)$. 
Now, if the forward orbit of $x$ does not contain the critical point, the claim follows immediately from the Lemma. 
Otherwise, let $n \geq 0$ be such that $f^n(x) = 0$ is the critical point. 
Note that this $n$ is unique, because otherwise the critical point would be periodic, so it would not lie in the Julia set. 
Hence, by applying the Lemma to the critical value $f^{n+1}(x)$, we have 
$$val(f^{n+1}(x)) \leq deg(T).$$
Finally, since the map $f_c$ is locally a double cover at the critical point, 
$$val(x) = val(f^n(x)) = 2 \cdot val(f^{n+1}(x)) \leq 2 \cdot deg(T).$$ 
\end{proof}

\section{Topological entropy} \label{section:htop}

Let $f : X \to X$ be a continuous map of a compact metric space $(X, d)$. A measure of the complexity of the 
orbits of the map is given by its \emph{topological entropy}. Let us now recall its definition. Useful 
references are \cite{dMvS} and \cite{CFS}.

Given $x \in X$, $\epsilon > 0$ and $n$ an integer, we define the ball $B_f(x, \epsilon, n)$ as the set of points whose 
orbit remains close to the orbit of $x$ for the first $n$ iterates:
$$B_f(x, \epsilon, n) := \{ y \in X \ : \ d(f^i(x), f^i(y)) < \epsilon \ \forall 0 \leq i \leq n \}.$$
A set $E \subseteq X$ is called $(n, \epsilon)$\emph{-spanning} if every point of $X$ remains close to some 
point of $E$ for the first $n$ iterates, i.e. if $X = \bigcup_{x \in E} B_f(x, \epsilon, n)$.
Let $N(n, \epsilon)$ be the minimal cardinality of a $(n, \epsilon)$-spanning set. The topological entropy is 
the growth rate of $N(n, \epsilon)$ as a function of $n$: 

\begin{definition} The \emph{topological entropy} of the map $f : X \to X$ is defined as
$$h_{top}(f) := \lim_{\epsilon \to 0^+} \lim_{n \to \infty} \frac{1}{n} \log N(n, \epsilon).$$
\end{definition}
   
When $f$ is a piecewise monotone map of a real interval, it is easier to compute the 
entropy by looking at the number of laps. 
Recall the \emph{lap number} $L(g)$ of a piecewise monotone interval map $g : I \to I$ is the smallest 
cardinality of a partition of $I$ in intervals such that the restriction of $g$ to any such interval is monotone. 
The following result of Misiurewicz and Szlenk relates the topological entropy to the growth rate of the lap 
number of the iterates of $f$: 

\begin{theorem}[\cite{MS}]
Let $f : I \to I$ be a piecewise monotone map of a close bounded interval $I$, and let $L(f^n)$
be the lap number of the iterate $f^n$. Then the following equality holds:
$$h(f) = \lim_{n \to \infty} \frac{1}{n} \log L(f^n).$$
\end{theorem}

Another useful property of topological entropy is that it is invariant under dynamical extensions 
of bounded degree:

\begin{proposition}[\cite{Bo}]\label{bdeg}
Let $f: X \to X$ and $g : Y \to Y$ be two continuous maps of compact metric spaces, and let $\pi : X \to Y$ 
a continuous, surjective map such that $g \circ \pi = \pi \circ f$. Then 
$$h_{top}(g) \leq h_{top}(f).$$
Moreover, if there exists a finite 
number $d$ such that for each $y \in Y$ the fiber $\pi^{-1}(y)$ has cardinality always smaller than $d$, then 
$$h_{top}(g) = h_{top}(f).$$ 
\end{proposition}

In order to resolve the ambiguities arising from considering different restrictions of the same map, 
if $K$ is an $f$-invariant set we shall use the notation $h_{top}(f, K)$ to denote the topological entropy 
of the restriction of $f$ to $K$. 

\begin{proposition}[\cite{Do}, Proposition 3] \label{entronoFatou}
Let $f: X \to X$ a continuous map of a compact metric space, and let $Y$ be a closed subset of $X$ such that 
$f(Y) \subseteq Y$. Suppose that, for each $x\in X$, the distance $d(f^n(x), Y)$ tends to zero, uniformly on 
any compact subset of $X \setminus Y$. Then $h_{top}(f, Y) = h_{top}(f, X)$.
\end{proposition}

The following proposition is the fundamental step to relate entropy and Hausdorff dimension of invariant 
subsets of the circle (\cite{Fu}, Proposition III.1; see also \cite{Bi}):
 
\begin{proposition} \label{dim_entro}
Let $d \geq 1$, and $\Omega \subset \mathbb{R}/\mathbb{Z}$ be a closed, invariant set for the map 
$Q(x) := dx \mod 1$. Then the topological entropy of the restriction of $Q$ to $\Omega$ is related 
to the Hausdorff dimension of $\Omega$ in the following way:
$$\textup{H.dim }\Omega = \frac{h_{top}(Q, \Omega)}{\log d}.$$ 
\end{proposition}



\section{Invariant sets of external angles} \label{section:invsets}

Let $f_c$ be a topologically finite quadratic polynomial, and $T_c$ its Hubbard tree. 
One of the main players in the rest of the paper is the set $H_c$ of angles of external rays 
landing on the Hubbard tree:
$$H_c := \{ \theta \in \mathbb{R}/\mathbb{Z} \ : R_c(\theta) \textup{ lands on } T_c \}.$$

Note that, since $T_c$ is compact and the Carath\'eodory loop is continuous by local connectivity, 
$H_c$ is a closed subset of the circle. Moreover, since $T_c \cap J(f_c)$ 
is $f_c$-invariant, then $H_c$ is invariant for the doubling map, i.e. $D(H_c) \subseteq H_c$.

Similarly, we will denote by $S_c$ the set of angles of rays landing on the spine $[-\beta, \beta]$, 
and $B_c$ the set of angles of rays landing on the set of biaccessible points.

\begin{proposition} \label{biaccentro}
Let $f_c$ be a topologically finite quadratic polynomial. Then 
$$\textup{H.dim }H_c = \textup{H.dim }S_c = \textup{H.dim }B_c.$$
\end{proposition}

\begin{proof}
Lemma \ref{biacc_lemma} implies the inclusion 
$$S_c \setminus \{0, 1/2 \} \subseteq B_c \subseteq \bigcup_{n \geq 0}^\infty D^{-n}(S_c)$$
hence 
$$\textup{H.dim }S_c \leq \textup{H.dim }B_c \leq \sup_{n \geq 0} \textup{H.dim } D^{-n}(S_c) = \textup{H.dim }S_c.$$ 
Moreover, it is clear that $H_c \subseteq B_c$, and by Lemma \ref{exttree} one also has 
$$S_c \setminus \{ 0, 1/2 \} \subseteq \bigcup_{n \geq 0}^\infty D^{-n}(H_c)$$
hence $\textup{H.dim }S_c \leq \textup{H.dim }H_c \leq \textup{H.dim }B_c$.
\end{proof}

We will now characterize the set $H_c$ and other similar sets of angles purely in terms of the dynamics
of the doubling map on the circle, as the set of points whose orbit never hits certain open intervals.
 
In order to do so, we will make use of the following lemma: 

\begin{lemma} \label{inv_set}
Let $X \subseteq S^1$ be a closed, forward invariant set for the doubling map $D$, so that $D(X) \subseteq X$, 
and let $U \subseteq S^1$ be an open set, disjoint from $X$. Suppose moreover that 
\begin{enumerate}
\item $D^{-1}(X) \setminus X \subseteq U$;
\item $\partial U \subseteq X$.
\end{enumerate}
Then $X$ equals the set of points whose orbit never hits $U$:
$$X = \{ \theta \in S^1 \ : \ D^n(\theta) \notin U \ \ \forall n \geq 0 \}.$$
\end{lemma}

\begin{proof}
Let $\theta$ belong to $X$. By forward invariance, $D^n(\theta) \in X$ for each $n \geq 0$, and since $X$ and $U$ are disjoint, 
then $D^n(\theta) \notin U$ for all $n$. Conversely, let us suppose that $\theta$ does not belong to $X$, 
and let $V$ be the connected component of the complement of $X$ containing $\theta$; since the doubling map is uniformly expanding,
there exists some $n$ such that $f^n(V)$ is the whole circle, hence there exists an integer $k \geq 1$ such that 
$D^k(V) \cap X \neq \emptyset$, but $D^{k-1}(V) \cap X = \emptyset$; 
then, $D^{k-1}(V)$ intersects $D^{-1}(X) \setminus X$, so by (1) it intersects $U$. 
Moreover, since $\partial U \subseteq X$ we have $D^{k-1}(V) \cap \partial U =\emptyset$, so 
$D^{k-1}(V)$ is an open set which intersects $U$ but does not intersect its boundary, hence $D^{k-1}(V) \subseteq U$ and, since 
$\theta \in V$, we have $D^{k-1}(\theta) \in U$.
\end{proof}


Let us now describe combinatorially the set of angles of rays landing on the Hubbard tree. Let $T_c$ be the Hubbard tree of $f_c$; 
since $T_c$ is a compact set, then $H_c = \gamma^{-1}(T_c)$ is a closed subset of the circle. Among all connected components of 
the complement of $H_c$, there are finitely many $U_1, U_2, \dots, U_r$ which contain rays which land on the preimage $f_c^{-1}(T_c)$.
The angles of rays landing on the Hubbard tree are precisely the angles whose future trajectory for the doubling map never hits the $U_i$: 

\begin{proposition}[\cite{TL}] \label{forbidden_Htree}
Let $T_c$ be the Hubbard tree of $f_c$, and $U_1, U_2, \dots, U_r$ be the connected components of the complement of 
$H_c$ which contain rays landing on $f_c^{-1}(T_c)$. Then the set $H_c$ of angles of rays landing on $T_c$ equals
$$H_c = \{ \theta \in \mathbb{R}/\mathbb{Z} \ : \ D^n(\theta) \notin U_i \ \ \forall n \geq 0 \ \forall i = 1, \dots, r \}.$$ 
\end{proposition}

\begin{proof}
It follows from Lemma \ref{inv_set} applied to $X = H_c$ and $U = U_1 \cup \dots \cup U_r$. Indeed, $D(H_c) \subseteq H_c$ since 
$T_c \cap J(f_c)$ is forward-invariant under $f_c$. The set $U$ is disjoint from $H_c$ by definition of the $U_i$. Moreover, if $\theta$ belongs to 
$D^{-1}(H_c) \setminus H_c$, then $R_c(\theta)$ lands on $f_c^{-1}(T_c)$, so $\theta$ belongs to some $U_i$. Finally, let us check that 
for each $i$ we have the inclusion $\partial U_i \subseteq H_c$. Indeed, if $U$ is non-empty then $H_c$ has no interior (since 
it is invariant for the doubling map and does not coincide with the whole circle), so angles on the boundary of $U_i$ are limits 
of angles in $H_c$, so their corresponding rays land on the Hubbard tree by continuity of the Riemann mapping on the boundary.
\end{proof}

\section{Entropy of Hubbard trees} \label{section:entroH}
 

We are now ready to prove the relationship between the topological entropy of a topologically finite 
quadratic polynomial $f_c$ and the Hausdorff dimension of the set of 
rays which land on the Hubbard tree $T_c$: 

\begin{theorem} \label{entro_dim} 
Let $f_c(z) = z^2 + c$ be a topologically finite quadratic polynomial, let $T_c$ be its Hubbard tree and 
$H_c$ the set of external angles of rays which land on the Hubbard tree. 
Then we have the identity
$$\frac{h_{top}(f_c \mid_{T_c})}{\log 2}  = \textup{H.dim }H_c.$$
\end{theorem}

\begin{proof}
Let $\gamma : \mathbb{R}/\mathbb{Z} \to J(f_c)$ the Carath\'eodory loop. 
We know that
$$\gamma( D( \theta)) = f_c(\gamma(\theta)).$$
By Proposition \ref{valencebound}, the cardinality of the preimage of any point is bounded; 
hence, by Theorem \ref{bdeg}, we have 
$$h_{top}(f_c, J(f_c) \cap T_c) = h_{top}(D, \gamma^{-1}(J(f_c )\cap T_c)) = h_{top}(D, H_c).$$
Moreover, Proposition \ref{entronoFatou} implies 
$$h_{top}(f_c, J(f_c) \cap T_c) = h_{top}(f_c, T_c).$$
Then we conclude, by the dimension formula of Proposition \ref{dim_entro}, that
$$\textup{H.dim } H_c = \frac{h_{top}(D, H_c)}{\log 2}.$$ 
\end{proof}

The exact same argument applies to any compact, forward invariant set $X$ in the Julia set: 

\begin{theorem}  
Let $f_c$ be a topologically finite quadratic polynomial, and $X \subseteq J(f_c)$ compact and 
invariant (i.e. $f_c(X) \subseteq X$). Let define the set
$$\Theta_c(X) := \{\theta \in \mathbb{R}/\mathbb{Z} \ : \ R_c(\theta) \textup{ lands on }X \};$$ 
then we have the equality 
$$\frac{h_{top}(f_c \mid_{X})}{\log 2}  = \textup{H.dim }\Theta_c(X).$$
\end{theorem}


\section{Combinatorial description: the real case} \label{section:symbolic}

Suppose $c \in \partial \mathcal{M} \cap \mathbb{R}$.  
By definition, the \emph{dynamic root} $r_c$ of $f_c$ is the critical value $c$ if $c$ belongs to the Julia set, 
otherwise it is the smallest value of $J(f_c) \cap \mathbb{R}$ larger than $c$. This means that $r_c$ lies 
on the boundary of the bounded Fatou component containing $c$.

Recall that the \emph{impression} of a parameter ray $R_M(\theta)$ is the set of all 
$c \in \partial \mathcal{M}$ for which there is a sequence $\{w_n \}$ such that 
$|w_n | > 1$, $w_n \to e^{2\pi i \theta}$, and $\Phi_M^{-1}(w_n) \to c$. We denote
the impression of $R_M(\theta)$ by $\hat{R}_M(\theta)$. It is a non-empty, compact, connected subset of $\partial \mathcal{M}$.
Every point of $\partial \mathcal{M}$ belongs to the impression of at least one parameter ray. Conjecturally,
every parameter ray $R_M(\theta)$ lands at a well-defined point $c(\theta) \in \partial \mathcal{M}$ and $\hat{R}_M(\theta) = {c(\theta)}$.

In the real case, much more is known to be true. First of all, every real Julia set is locally connected \cite{LvS}. 
The following result summarizes the situation for real maps.

\begin{theorem}[\cite{Za}, Theorem 3.3]
Let $c \in \partial \mathcal{M} \cap \mathbb{R}$. Then there exists a unique angle $\theta_c \in [0, 1/2]$ such that
the rays $R_c(\pm \theta_c)$ land at the dynamic root $r_c$ of $f_c$. In the parameter plane, the
two rays $R_M (\pm \theta_c)$, and only these rays, contain $c$ in their impression. 
\end{theorem}

The theorem builds on the previous results of Douady-Hubbard \cite{DH} and Tan Lei \cite{TanL} for the case of 
periodic and preperiodic critical points and uses density of hyperbolicity in the real quadratic family to get the claim 
for all real maps.

To each angle $\theta \in S^1$ we can 
associate a \emph{length} $\ell(\theta)$ as the length (along the circle) of the chord delimited by the leaf joining $\theta$ to 
$1-\theta$ and containing the angle $\theta = 0$. In formulas, it is easy to check that 
$$\ell(\theta) := \left\{ \begin{array}{ll} 2\theta & \textup{if }0 \leq \theta < \frac{1}{2} \\
                                            2-2\theta &  \textup{if }\frac{1}{2} \leq \theta < 1. \\
                          \end{array}
\right.$$
For a real parameter $c$, we will denote as $\ell_c$ the length of the characteristic leaf 
$$\ell_c := \ell(\theta_c).$$
The key to analyzing the symbolic dynamics of $f_c$ is the following interpretation
in terms of the dynamics of the tent map. Since all real Julia sets are locally connected, for $c$ real all dynamical 
rays $R_c(\theta)$ have a well-defined limit $\gamma_c(\theta)$, which belongs to $J(f_c)$. 
Let us moreover denote by $T$ the full tent map on the interval $[0,1]$, defined as 
$T(x) := \min \{2x, 2-2x\}$. The following diagram is commutative:

$$\xymatrix{
  S^1 \ar@(ul,ur)^D \ar[d]^\ell \ar[r]^{\gamma_c} & J(f_c) \ar@(ul,ur)^{f_c}\\
[0,1] \ar@(ur,r)^T
  }
$$

This means that we can understand the dynamics of $f_c$ on the Julia set in terms of the dynamics of the tent 
map on the space of lengths. 
First of all, the set of external angles corresponding 
to rays which land on the real slice of the Julia set can be given the following characterization:

\begin{proposition} \label{symb_spine}
Let $c \in [-2, \frac{1}{4}]$. Then the set $S_c$ of external angles of rays which land 
on the real slice $J(f_c) \cap \mathbb{R}$ of the Julia set is
$$S_c = \{ \theta \in \mathbb{R}/\mathbb{Z} \ : \ T^n(\ell(\theta)) \leq \ell_c \quad \forall n \geq 1 \}.$$
\end{proposition}

\begin{proof}
Let $X$ be the set of angles of rays landing on the segment $[c, \beta]$. Since $f_c^{-1}([c, \beta]) = [-\beta, \beta]$, 
then $D^{-1}(X)$ is the set of angles landing on the spine. Thus, if we set $U := (\theta_c, 1-\theta_c)$ then the hypotheses
of Lemma \ref{inv_set} hold, hence we get the following description:
$$S_c = \{ \theta \in \mathbb{R}/\mathbb{Z} \ : \ D^n(\theta) \notin (\theta_c, 1-\theta_c) \ \ \forall n \geq 1 \}$$
hence by taking the length on both sides
$$\theta \in S_c \quad \Leftrightarrow \quad \ell(D^n(\theta)) \leq \ell(\theta_c) \ \ \forall n \geq 1 $$
and by the commutative diagram we have $\ell(D^n(\theta)) = T^n(\ell(\theta))$, which, when substituted into the previous 
equation, yields the claim.                                                                                                                      \
\end{proof}

Recall that for a real polynomial $f_c$ the Hubbard tree is the segment $[c, f_c(c)]$. 
Let us denote as $L_c := \ell(D(\theta_c))$ the length of the leaf which corresponds to $f_c(c) = c^2 + c$. The 
set of angles which land on the Hubbard tree can be characterized as:

\begin{proposition} \label{symb_Htree}
The set $H_c$ of angles of external rays which land on the Hubbard tree for $f_c$ is:
$$H_c := \{ \theta \in \mathbb{R}/\mathbb{Z} \ : \ T^n(\ell(\theta)) \geq L_c \ \quad \forall n \geq 0 \}.$$
\end{proposition}

\begin{proof}
Since the Hubbard tree is $[c, f_c(c)]$ and its preimage is $[0, c]$, one can take $U = (D(\theta_c), 1-D(\theta_c))$ 
(where we mean the interval containing zero) and $X = H_c$, and we get by Lemma \ref{inv_set}
$$H_c = \{ \theta \in S^1 \ : \ D^n(\theta) \notin U \ \ \forall n \geq 0 \}$$ 
hence in terms of length 
$$H_c = \{ \theta \in S^1 \ : \ \ell(D^n(\theta)) \geq \ell(D(\theta_c)) \  \ \forall n \geq 0\}$$ 
which yields the result when you substitute $\ell(D^n(\theta)) = T^n(\ell(\theta))$ and $L_c = \ell(D(\theta_c))$.
\end{proof}



 



\subsection{The real slice of the Mandelbrot set}
Let us now turn to parameter space. We are looking for a combinatorial description 
of the set of rays which land on the real axis. However, in order to account 
for the fact that some rays might not land, let us define the set $\mathcal{R}$ of \emph{real parameter angles} as 
the set of angles of rays whose prime-end impression intersects the real axis:

$$\mathcal{R} := \{ \theta \in S^1 \ : \ \hat{R}_M(\theta) \cap \mathbb{R} \neq \emptyset \}.$$ 

The set $\mathcal{R}$ is also the closure (in $S^1$) of the union of the angles of rays landing on the boundaries of all 
real hyperbolic components.
Combinatorially, elements of $\mathcal{R}$ correspond to leaves which are maximal in their orbit under the dynamics 
of the tent map: 
\begin{proposition} \label{symb_Pc}
The set $\mathcal{R}$ of real parameter angles can be characterized as
$$\mathcal{R} = \{ \theta \in S^1 \ : \ T^n(\ell(\theta)) \leq \ell(\theta) \ \ \forall n \geq 0 \}.$$
\end{proposition}

\begin{proof}
Let $\theta_c$ be the characteristic angle of a real quadratic polynomial. Since the corresponding dynamical ray 
$R_c(\theta)$ lands on the spine, by Proposition \ref{symb_spine} applied to $\ell(\theta_c) = \ell_c$ we have for each $n \geq 0$
$$T^n(\ell(\theta_c)) \leq \ell(\theta_c).$$ 
Conversely, if $\theta$ does not belong to $\mathcal{R}$ then it belongs to the opening of some real hyperbolic component $W$. 
By symmetry, we can assume $\theta$ belongs to $[0, 1/2]$: then $\theta$ must belong to the interval $(\alpha, \omega)$, whose
endpoints have binary expansion
$$\begin{array}{ll} \alpha = 0.\overline{s_1 \dots s_n} \\
   \omega = 0.\overline{\check{s_1}\dots \check{s_n} s_1 \dots s_n}
  \end{array}$$
where $n$ is the period of $W$, and $s_1 = 0$ (recall the notation $\check{s_i} := 1- s_i$); in this case it is easy to check that 
both $\ell(\alpha) = 2\alpha$ and $\ell(\omega) = 2 \omega$ are fixed points of $T^n$, 
and $T^n(x) > x$ if $x \in (2 \alpha, 2\omega)$. 
The description is equivalent to the one given in (\cite{Za}, Theorem 3.7).
\end{proof}

Note moreover that the image of characteristic leaves are the shortest leaves in the orbit:

\begin{proposition} \label{postchar}
The set $\mathcal{R} \setminus \{ 0 \}$ of non-zero real parameter angles can be characterized as
$$\mathcal{R} \setminus \{ 0 \} = \{ \theta \in [1/4, 3/4] \ : \ T^n(\ell(D(\theta))) \geq \ell(D(\theta)) \ \ \forall n \geq 0 \}.$$
\end{proposition}

\begin{proof}
Since $\theta \in \mathcal{R} \setminus \{0\}$, then $\ell(\theta) \geq 2/3$, so $\ell(D(\theta)) \leq 1/3$. 
The claim follows then from the previous proposition by noting that $T$ maps $[1/2, 1]$ homeomorphically to 
$[0,1]$ and reversing the orientation.
\end{proof}

In the following it will be useful to introduce the following slice of $\mathcal{R}$, by taking for each 
$c \in [-2, 1/4]$ the set of angles of rays whose impression intersects the real axis to the right of $c$.

\begin{definition}
Let $c \in [-2, 1/4]$. Then we define the set 
$$P_c := \mathcal{R} \cap [1-\theta_c, \theta_c]$$
where $\theta_c \in [0, 1/2]$ is the characteristic ray of $f_c$, and $[1-\theta_c, \theta_c]$ 
is the interval containing $0$.
\end{definition}

A corollary of the previous description is that parameter rays landing on $\partial \mathcal{M} \cap \mathbb{R}$ to the right of 
$c$ also land on the Hubbard tree of $c$:

\begin{corollary}   \label{easyincl}
Let $c \in [-2, 1/4]$. Then the inclusion
$$P_c \setminus \{ 0 \} \subseteq H_c$$
holds.
\end{corollary}

\begin{proof}
Let $\theta \neq 0$ belong to $P_c$. Then $\ell(\theta) \leq \ell(\theta_c)$, hence also 
$\ell(D(\theta)) \geq \ell(D(\theta_c))$. Now, by Proposition \ref{symb_Pc},
$$T^n(\ell(D(\theta))) \geq \ell(D(\theta)) \geq \ell(D(\theta_c))$$
for each $n \geq 0$, hence $\theta$ belongs to $H_c$ by Proposition \ref{symb_Htree}.
\end{proof}

\section{Compact coding of kneading sequences} \label{section:coding}

In order to describe the combinatorics of the real slice, we will now associate to each real external ray an infinite sequence
of positive integers. The notation is inspired by the correspondence with continued 
fractions established in \cite{BCIT}. Indeed, because of the isomorphism, the set of integer sequences which arise from 
parameters on the real slice of $\mathcal{M}$ is exactly the same as the set of sequences of partial quotients of elements 
of the bifurcation set $\mathcal{E}$ for continued fractions.

Let $\Sigma := (\mathbb{N}^+)^\mathbb{N}$ be the space of infinite sequences of positive integers, 
and $\sigma : \Sigma \to \Sigma$ be the shift operator. Sequences of positive integers will also be called \emph{strings}.

Let us now associate a sequence of integers to each angle. Indeed, let $\theta \in \mathbb{R}/\mathbb{Z}$, and 
write $\theta$ as a binary sequence: if $0 \leq \theta < 1/2$, we have 
$$\theta =  0.\underbrace{0\dots0}_{a_1}\underbrace{1\dots1}_{a_2}\underbrace{0\dots0}_{a_3}\dots \qquad a_i \geq 1$$
while if $1/2 \leq \theta < 1$ we have 
$$\theta =  0.\underbrace{1\dots1}_{a_1}\underbrace{0\dots0}_{a_2}\underbrace{1\dots1}_{a_3}\dots \qquad a_i \geq 1.$$
In both cases, let us define the sequence $w_\theta$ by counting the number of repetitions of the same symbol:
$$w_\theta := (a_1, a_2, a_3, \dots).$$
Note moreover that $w_\theta$ only depends on $\ell(\theta)$, which in both cases is given by 
$$\ell(\theta) = 0.\underbrace{0\dots 0}_{a_1-1}\underbrace{1\dots1}_{a_2}\underbrace{0\dots0}_{a_3}\dots\qquad a_i \geq 1.$$
Note that we have the following commutative diagram:
$$\xymatrix{
  \mathbb{R}/\mathbb{Z} \ar@(ul, ur)^D \ar[r]^\ell & [0,1] \ar@(ul,ur)^T \ar[r] &  \Sigma \ar@(ul,ur)^F 
}$$
where $F((a_1, a_2, \dots)) = (a_1 - 1, a_2, \dots)$ if $a_1 > 1$, and $F((1, a_2, \dots)) = (a_2, \dots)$.

If $\theta_c$ is the characteristic angle of a real hyperbolic component, we denote by $w_c$ the string 
associated to the \emph{postcharacteristic leaf} $L_c = (D(\theta_c), 1-D(\theta_c))$.
For instance, the airplane component has root $\theta_c = 3/7 = 0.\overline{011}$, so $D(\theta_c) = 1/7 = 0.\overline{001}$ and 
$w_c = \overline{(2,1)}$.



\subsection{Extremal strings}

Let us now define the \emph{alternate lexicographic order} on the set of strings of positive integers.
Let $S = (a_1, \dots, a_n)$ and $T = (b_1, \dots, b_n)$ be two finite strings of positive integers of equal length, 
and let $k := \min \{ i \geq 1 \ : \ a_i \neq b_i \}$ the first different digit. We will say that $S < T$ if 
$k \leq n$ and either 
$$ k \textup{ is odd  and }a_k > b_k $$
or 
$$ k \textup{ is even  and }a_k < b_k. $$
For instance, in this order $(2, 1) < (1, 2)$, and $(2, 1) < (2, 3)$. 
The order can be extended to an order on the set $\Sigma := (\mathbb{N}^+)^\mathbb{N}$ of infinite strings of positive integers.
Namely, if $S = (a_1, a_2, \dots)$ and $T = (b_1, b_2, \dots)$ are two infinite strings, then $S < T$ 
if there exists some $n \geq 1$ for which $(a_1, a_2, \dots, a_n) < (b_1, b_2, \dots, b_n)$.
We will denote as $\overline{S}$ the infinite periodic string $(S, S, \dots)$.

Note that as a consequence of our ordering we have, for two angles $\theta$ and $\theta'$, 
$$w_{\theta} < w_{\theta'} \Leftrightarrow \ell(\theta) > \ell(\theta')$$
and on the other hand, for two real $c, c' \in \partial \mathcal{M} \cap \mathbb{R}$, 
$$w_{c} < w_{c'} \Leftrightarrow \ell(\theta_c) < \ell(\theta'_c).$$
The following is a convenient criterion to compare periodic strings:

\begin{lemma}[\cite{CT}, Lemma 2.12] \label{stringlemma} 
Let $S$, $T$ be finite strings of positive integers. Then 
\begin{equation} 
ST < TS \Leftrightarrow \overline{S} < \overline{T}.
\end{equation}
\end{lemma}

In order to describe the real kneading sequences, we need the 
\begin{definition}
A finite string of positive integers $S$ is called \emph{extremal} if 
$$XY < YX$$ 
for every splitting $S = XY$ where $X$, $Y$ are nonempty strings.
\end{definition}

For instance, the string $(2, 1, 2)$ is extremal because $(2, 1, 2) < (2, 2, 1) < (1, 2, 2)$. Note that 
a string whose first digit is strictly larger than the others is always extremal.





Extremal strings are very useful because they parametrize purely periodic (i.e. rational with odd denominator)
parameter angles on the real axis:

\begin{lemma}
A purely periodic angle $\theta \in [1/4, 3/4]$ belongs to the set $\mathcal{R}$ if and only if there exists an 
extremal string $S$ for which 
$$w_{D(\theta)} = \overline{S}.$$
\end{lemma}

\begin{proof}
Let $\theta \in [1/4, 1/2]$ be purely periodic for the doubling map. Then we can write its expansion as 
$$\theta = 0.\overline{01^{a_1}0^{a_2}\dots0^{a_n-1}}$$
with $a_i \geq 1$, and $n$ even. Then $x:= \ell(D(\theta)) = 0.\overline{0^{a_1-1} 1^{a_2} \dots 1^{a_n} 0 }$, 
and by Proposition \ref{postchar} the angle $\theta$ belongs to $\mathcal{R}$ if and only if 
$$T^n(x) \geq x \qquad \textup{for all }n \geq 0.$$
By writing out the binary expansion one finds out that this is equivalent to the statement
$$0.\overline{0^{a_k -1} 1^{a_{k+1}} \dots 1^{a_{k-1}} 0} \geq 0.\overline{0^{a_1-1}1^{a_2} \dots 1^{a_n}0} \qquad \textup{for all }1 \leq k \leq n$$
which in terms of strings reads
$$(\overline{a_k, \dots, a_n, a_1, \dots, a_{k-1}}) \geq (\overline{a_1, \dots, a_n}) \qquad \textup{for all }1 \leq k \leq n.$$
The condition is clearly satisfied if $S = (a_1, \dots, a_n)$ is extremal. Conversely, if the condition is satisfied then 
$S$ must be of the form $S = P^k$ with $P$ an extremal string.
\end{proof}

\subsection{Dominant strings}

The order $<$ is a total order on the strings of positive integers of fixed given length; in order to be able to compare strings of different lengths 
we define the partial order 
$$S << T \quad \textup{if }\exists i \leq \min\{|S|, |T|\} \textup{ s.t. }S_1^i < T_1^i$$
where $S_1^i := (a_1, \dots, a_i)$ denotes the truncation of $S$ to the first $i$ characters. 
Let us note that:
\begin{enumerate}
\item if $|S| = |T|$, then $S < T$ if and only if $S << T$;
\item if $S, T, U$ are any strings, $S << T \Rightarrow SU << T, S << TU$; 
\item If $S << T$, then $S \cdot z < T \cdot w$ for any $z, w \in (\mathbb{N}^+)^\mathbb{N}$.
\end{enumerate}

\begin{definition} \label{dom_string}
A finite string $S$ of positive integers is called \emph{dominant} if it has even length and 
$$XY << Y$$
for every splitting $S = XY$ where $X$, $Y$ are finite, nonempty strings.
\end{definition}

Let us remark that every dominant string is extremal, while the converse is not true.
For instance, the strings $(5,2,4,3)$ and $(5,2,4,5)$ are both extremal, but the first is dominant while the second is not.
On the other hand, a string whose first digit is strictly large than the others is always dominant (as a corollary, there exist 
dominant strings of arbitrary length).


\begin{definition}
A real parameter $c$ is \emph{dominant} if there exists a dominant string $S$ such that
$$w_c = \overline{S}.$$
\end{definition}

The airplane parameter $\theta_c = 0.\overline{011}$ is dominant because $w_c = \overline{(2, 1)}$, 
and $(2, 1)$ is dominant.
On the other hand, the period-doubling of the airplane ($\theta_c = 0.\overline{011100}$) is not dominant 
because its associated sequence is $\overline{(3)}$, and dominant strings must be of even length. In general, we 
will see that tuning always produces non-dominant parameters.

However, the key result is that dominant parameters are dense in the set of non-renormalizable angles:

\begin{proposition} \label{densitydominant}
Let $\theta_c \in [0, 1/2]$ be the characteristic angle of a real, non-renormalizable parameter $c$, with $c \neq -1$. Then 
$\theta_c$ is limit point from below of characteristic angles of dominant parameters.
\end{proposition}
 
Since the proof of the proposition is quite technical, it will be postponed to section \ref{proofdensity}.

\subsection{The bisection algorithm} \label{section:bisect}

Let us now describe an algorithm to generate all real hyperbolic windows (see Figure \ref{bisection}). 

\begin{theorem} \label{algo}
The set of all real hyperbolic windows in the Mandelbrot set can be generated as follows.
Let $c_1 < c_2$ be two real parameters on the boundary of $\mathcal{M}$, with external angles 
$0 \leq \theta_2 < \theta_1 \leq \frac{1}{2}$.
Let $\theta^*$ be the dyadic pseudocenter of the interval $(\theta_2, \theta_1)$, and let 
$$\theta^* = 0.s_1 s_2 \dots s_{n-1} s_n$$
be its binary expansion, with $s_n = 1$. Then the hyperbolic window of smallest period in the interval
 $(\theta_2, \theta_1)$ is the interval of external angles $(\alpha_2, \alpha_1)$ with 
$$\begin{array}{ccc}
\alpha_2 & := & 0.\overline{s_1 s_2 \dots s_{n-1}} \\
\alpha_1 & := & 0.\overline{s_1 s_2 \dots s_{n-1} \check{s}_1 \check{s}_2 \dots \check{s}_{n-1} }
\end{array}
$$
(where $\check{s_i} := 1 - s_i$). All hyperbolic windows are obtained by iteration of this algorithm, 
starting with $\theta_2 = 0$, $\theta_1 = 1/2$.
\end{theorem}

\begin{figure}[h!]

\setlength{\fboxsep}{0pt}%
\setlength{\fboxrule}{0.5pt}%
\fbox{\includegraphics[scale=0.75]{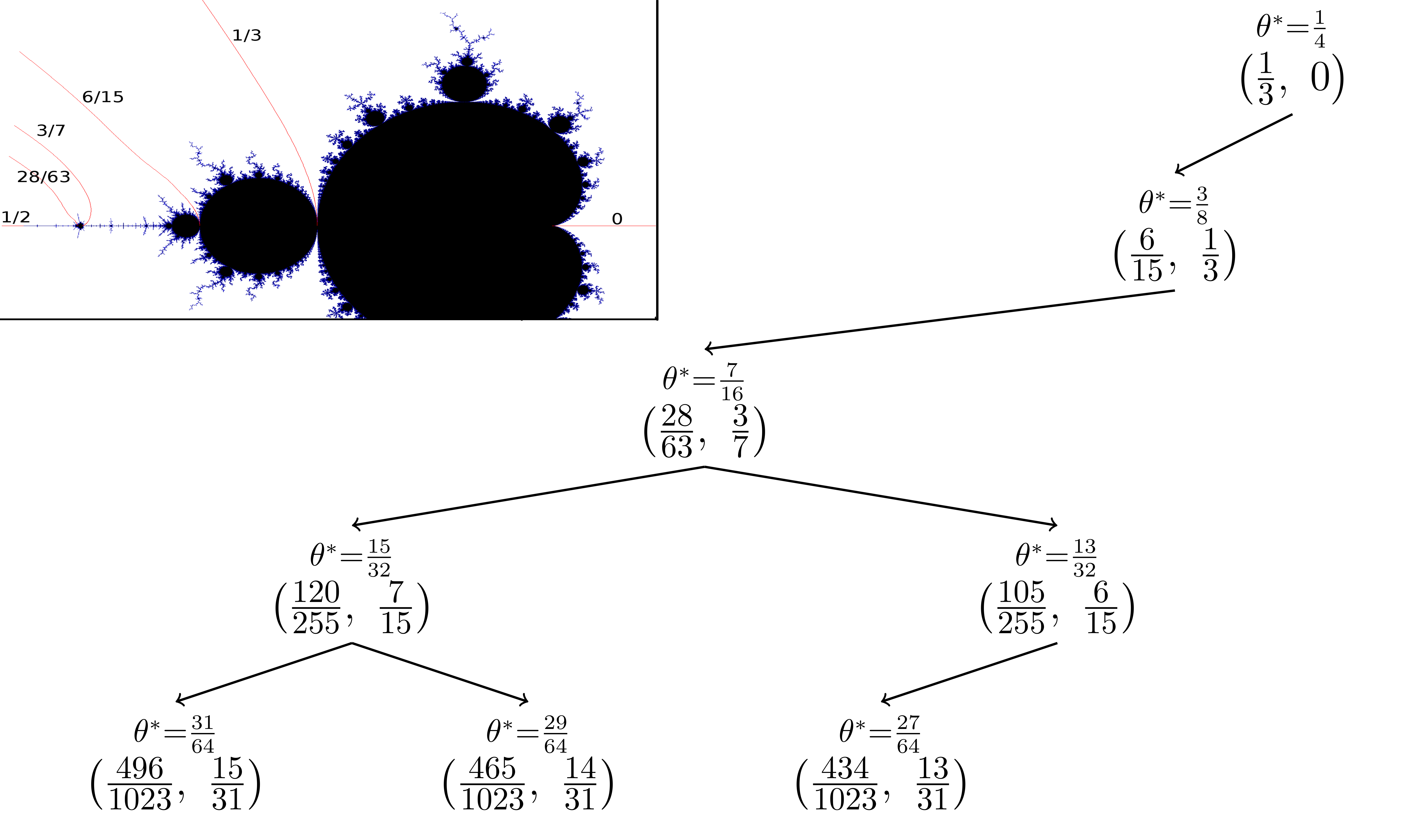}}

\caption{The first few generations of the bisection algorithm which produces all real hyperbolic 
windows between external angles $0$ and $\frac{1}{2}$. Every interval represents a hyperbolic component, 
and we display the angles of rays landing at the endpoints as well as the pseudocenter $\theta^*$. 
The root of the tree ($\theta^* = \frac{1}{4}$) corresponds to the real slice of the main cardioid, its 
child is the ``basilica'' component of period $2$ ($\theta^* = \frac{3}{8}$), then $\theta^* = \frac{7}{16}$ 
corresponds to the ``airplane'' component of period $3$ etc. Some branches 
of the tree do not appear because some pairs of components have an endpoint in common (due to period doubling). }

\label{bisection}

\end{figure}

\begin{proof}[Proof of Theorem \ref{algo}.]
The theorem is a rephrasing, in the language of complex dynamics, of (\cite{BCIT}, Proposition 3).
Indeed, the set $\Lambda$ of \cite{BCIT} is almost precisely the set $\mathcal{R}$ of real parameter angles; 
precisely, we have the equality $\mathcal{R} \cap [0, 1/2] = \frac{1}{2} \Lambda$ (\cite{BCIT}, Proposition 7), 
and the intervals $J_d = (r^-, r^+)$ of (\cite{BCIT}, Section 4.1) determine exactly the hyperbolic windows
 $[\alpha_2, \alpha_1]$ defined in the statement of the theorem, via the translation $\alpha_1 = \frac{r^-}{2}$ and $\alpha_2 = \frac{r^+}{2}$. 
\end{proof}

\noindent \textbf{Example}

\medskip
Suppose we want to find all hyperbolic components between the airplane 
parameter (of period $3$) and the basilica parameter (of period $2$).
The ray landing on the root of the airplane component has angle $\theta_1 = \frac{3}{7}$,
while the ray landing immediately to the left of the basilica has angle $\theta_2 = \frac{2}{5}$.  Let us apply the algorithm: 
$$\begin{array}{lll}
   \theta_2 = \frac{2}{5} & = & 0.011001100110 \dots \\
\theta_1 = \frac{3}{7} & = & 0.011011011011\dots \\
\hline \\
\theta^* & = & 0.01101
  \end{array}$$
hence $\alpha_1 = 0.\overline{0110} = \frac{2}{5}$ and $\alpha_2 = 0.\overline{01101001} = \frac{7}{17}$ and we get the 
component of period $4$ which is the doubling of the basilica. Note we do not always get the doubling of the previous 
component; indeed, the next step would be
$$\begin{array}{lll}
\theta_2 = \frac{7}{17} & = & 0.011010010110 \dots \\
\theta_1 = \frac{3}{7} & = & 0.011011011011\dots \\
\hline \\
\theta^* & = & 0.011011
  \end{array}$$
hence $\alpha_1 = 0.\overline{01101}$ and we get a component of period $5$.
Iteration of the algorithm produces all real hyperbolic components.
We conjecture that a similar algorithm holds in every vein.

\section{A copy of the Hubbard tree inside parameter space} \label{section:embed}

We saw that the set of rays which land on the real axis in parameter space also land in the dynamical plane. 
In order to establish equality of dimensions, we would like to prove the other inclusion. Unfortunately, 
in general there is no copy of $H_c$ inside $P_c$ (for instance, is $c$ is the basilica tuned with itself, 
then the Hubbard tree is a countable set, while only two pairs of rays land in parameter space to the right of $c$).
However, outside of the baby Mandelbrot sets, one can indeed map the combinatorial model for the Hubbard tree 
into the combinatorial model of parameter space:

\begin{proposition} \label{embedding}
Given a non-renormalizable, real parameter $c$ and another real parameter $c' > c$, there exists 
a piecewise linear map $F : \mathbb{R}/\mathbb{Z} \to \mathbb{R}/\mathbb{Z}$ such that 
$$F(H_{c'}) \subseteq P_{c}.$$ 
\end{proposition}

\begin{proof}
Let us denote $\ell := \ell(c)$ and $\ell' := \ell(c')$ the lengths of the characteristic leaves. 
Let us now choose a dominant parameter $c''$ in between $c$ and $c'$ and such that its corresponding 
string $w_{c''} = \overline{S}$ with $S$ dominant, in such a way that $S$ is a prefix of $w_c$ and not a prefix 
of $w_{c'}$. Let us denote by $\ell'' := \ell(c'')$ the length of the characteristic leaf of $c''$.

If $S = (s_1, s_2, \dots, s_n)$ (recall $n$ must be even), let us define the dyadic number 
$$s := 0.01^{s_1}0^{s_2}\dots1^{s_{n-1}} 0^{s_n}$$ and the ``length'' of $S$ to be $N := s_1+s_2+ \dots + s_n$.
Then, let us construct the map 
\begin{equation} \label{defF}
F(\theta) := \left\{ \begin{array}{ll} 
                        s + \frac{1-\theta}{2^{N+1}} & \textup{ if } 0 \leq \theta < \frac{1}{2} \\                        
			(1-s) + \frac{\theta}{2^{N+1}} & \textup{ if } \frac{1}{2} \leq \theta < 1 \\
                       \end{array} 
		\right.
\end{equation}
Let us now check that $F$ maps $[0, \frac{1}{2}) \cap H_{c'}$ into $P_{c''} \subseteq P_c$ (then the other 
half follows by symmetry). In order to verify the claim, 
let us pick $\theta \in H_{c'}$, $0 < \theta < \frac{1}{2}$. We need to check that $\phi := F(\theta)$ satisfies:
\begin{enumerate}
\item $\ell(\phi) \leq \ell''$; 
\item $T^n(\ell(\phi)) \leq \ell(\phi) \qquad \forall n \geq 0$.
\end{enumerate}

(1) 
Since $\theta$ belongs to $H_{c'}$, by Proposition \ref{symb_Htree} we have 
$$\ell(\theta) \geq L_{c'} \geq L_{c''}.$$
Moreover, equation \eqref{defF} implies
$$\ell(\phi) = 2s + 2^{-N} ( 1 - \ell(\theta)/2 )$$
while by the definition of $s$ one has
$$\ell'' = 2s + 2^{-N} ( 1 - L_{c''}/2  )$$
hence combining with the previous inequality we get $\ell(\phi) \leq \ell''$.

(2) If $1 \leq n < N$, then 
either $T^n(\ell(\phi)) \leq \frac{1}{2} < \ell(\phi)$, or $T^n(\ell(\phi))$ is of the form 
$$0.1^{s_k}0^{s_{k+1}}\dots0^{s_n}\dots$$
which is less than $0.1^{s_1}0^{s_2}\dots1^{s_n}$ because of dominance. If instead $n > N$,
$T^n(\ell(\phi)) = T^{n-N-1}(\ell(\theta)) \leq \ell'$, 
and $\ell' < \ell(\phi)$ because $\ell(\phi)$ begins with $0.1^{s_1}0^{s_2}\dots0^{s_n}$, 
and $S$ is not a prefix of $w_{\theta'}$. Finally, let $\hat{\theta} := \max\{\theta, 1-\theta\}$ and 
analyze the $N^{th}$ iterate: we have
$$T^N(\ell(\phi)) = \hat{\theta} \leq 2s + \frac{\hat{\theta}}{2^N} = \ell(\phi)$$
because $\hat{\theta}$ belongs to $H_{c'} \subseteq H_{c''}$, and $\max\{\theta \in [0,1] \ : \theta \in H_{c''} \} = 2s/(1-2^{-N})$.  
\end{proof}










\section{Renormalization and tuning} \label{section:renorm}

The Mandelbrot set has the remarkable property that near every point of its boundary there are infinitely 
many copies of the whole $\mathcal{M}$, called \emph{baby Mandelbrot sets}.
A \emph{hyperbolic component} $W$ of the Mandelbrot set is a connected component of the interior of $\mathcal{M}$
such that all $c \in W$, the orbit of the critical point is attracted to a periodic cycle under iteration
of $f_c$.

Douady and Hubbard \cite{DH} related the presence of baby copies of $\mathcal{M}$ to renormalization 
in the family of quadratic polynomials. 
More precisely, they associated to any hyperbolic component $W$ a \emph{tuning map} $\iota_W : \mathcal{M} \rightarrow \mathcal{M}$ which 
maps the main cardioid of $\mathcal{M}$ to $W$, and such that the image of the whole $\mathcal{M}$ under $\iota_W$ is a baby copy of $\mathcal{M}$.

The tuning map can be described in terms of external angles in the following terms \cite{DoAlgo}. Let $W$ be a hyperbolic component, and 
$\eta_0$, $\eta_1$ the angles of the two external rays which land on the root of $W$. 
Let $\eta_0 = 0.\overline{\Sigma_0}$ and $\eta_1 = 0.\overline{\Sigma_1}$ be the (purely periodic) binary expansions of the two angles 
which land at the root of $W$. 
Let us define the map $\tau_W : \mathbb{R}/\mathbb{Z} \rightarrow \mathbb{R} /\mathbb{Z}$ in the following way:
$$\theta = 0.\theta_1\theta_2\theta_3 \dots \mapsto \tau_W(\theta) = 0.\Sigma_{\theta_1}\Sigma_{\theta_2}\Sigma_{\theta_3}\dots$$
where $\theta = 0.\theta_1 \theta_2 \dots$ is the binary expansion of $\theta$, and its image is given by substituting the binary 
string $\Sigma_0$ to every occurrence of $0$ and $\Sigma_1$ to every occurrence of $1$.

\begin{proposition}[\cite{Do}, Proposition 7]
The map $\tau_W$ has the property that, if $\theta$ is a characteristic angle of the parameter $c \in \partial \mathcal{M}$, then 
$\tau_W(\theta)$ is a characteristic angle of the parameter $\iota_W(c)$.
\end{proposition}


If $W$ is a real hyperbolic component, then $\iota_W$ preserves the real axis. The image of the tuning operator 
is the \emph{tuning window} $\Omega(W)$ with
$$ \Omega(W) := [\omega(W), \alpha(W)]$$
where 
$$\begin{array}{lll}
  \alpha(W) & := & 0.\overline{\Sigma_0} \\
 \omega(W) & := & 0.\Sigma_0\overline{\Sigma_1}.
  \end{array}
$$
The point $\alpha(W)$ will be called the \emph{root} of the tuning window. Overlapping tuning windows are nested, 
and we call \emph{maximal tuning window} a tuning window which is not contained in any other tuning window.

Let us describe the behavior of Hausdorff dimension with respect to the tuning operator:

\begin{proposition} \label{tuneddim}
Let $W$ be a hyperbolic component of period $p$ with root $r(W)$, and let $c \in \mathcal{M}$. Then 
we have the equalities 
$$\textup{H.dim }H_{\tau_W(c)} = \max \left\{ \textup{H.dim }H_{r(W)}, \textup{H.dim }\tau_W(H_c) \right\}$$
$$\textup{H.dim }P_{\tau_W(c)} = \max \left\{ \textup{H.dim }P_{r(W)}, \textup{H.dim }\tau_W(P_c) \right\}.$$
Moreover, 
$$\textup{H.dim }\tau_W(H_c) = \frac{1}{p} \textup{H.dim }H_c.$$
\end{proposition}

\begin{proof}
Let $c' := \tau_W(c)$. The Julia set of $f_{c'}$ is constructed by taking the Julia set of 
$f_{r(W)}$ and inserting a copy of the Julia set of $f_c$ inside every bounded Fatou component. 
Hence in particular, the extended Hubbard tree of $J(f_{c'})$ contains a topological copy $T_1$ of the 
extended Hubbard tree of $f_{r(W)}$ which contains the critical value $c'$. The set of angles which land on $T_1$ 
are precisely the image $\tau_W(H^{ext}_{c})$ via tuning of the set $H^{ext}_c$ of angles which land on the extended 
Hubbard tree of $H_c$.
Let $\theta \in H_{c'}$ be an angle whose ray lands on the Hubbard tree of $f_{c'}$. Then 
either $\theta$ also belongs to $H_{r(W)}$ or it lands on a small copy of the extended Hubbard tree 
of $f_{r(W)}$, hence it eventually maps to $T_1$. Hence we have the inclusions
$$H_{r(W)} \cup \tau_W(H_{c}) \subseteq H_{c'} \subseteq H_{r(W)} \cup \bigcup_{n \geq 0} D^{-n}(\tau_W(H^{ext}_{c})) $$
from which the claim follows, recalling that $H^{ext}_c \setminus\{-\beta, \beta\} \subseteq \bigcup_{n \geq 0} D^{-n}(H_c)$. 

In parameter space, one notices that the set of rays landing on the vein $v$ for $c'$
either land between $0$ and $r(W)$, or between $r(W)$ and $c'$. In the latter case, 
they land on the small copy of the Mandelbrot set with root $r(W)$, so they are in the image of $\tau_W$. 
Hence
$$P_{c'} = P_{r(W)} \cup \tau_W(P_c)$$
and the claim follows. The last claim follows by looking at the commutative diagram 
$$\xymatrix{
H_c \ar@(ul,ur)^D \ar[r]^{\tau_W} & \tau_W(H_c). \ar@(ul,ur)^{D^p}
}$$
Since $\tau_W$ is injective and continuous restricted to $H_c$ (because $H_c$ does not contain dyadic rationals) we have by Proposition \ref{bdeg}
$$h_{top}(D, H_c) = h_{top}(D^p, \tau_W(H_c))$$
and, since $H_c$ is forward invariant we can apply Proposition \ref{dim_entro} and get
$$\textup{H.dim }\tau_W(H_c) = \frac{h_{top}(D^p, \tau_W(H_c))}{p \log 2} = \frac{1}{p} \frac{h_{top}(D, H_c)}{\log 2} = \frac{1}{p} \textup{H.dim }H_c$$
from which the claim follows.

\end{proof}





\subsection{Scaling and continuity at the Feigenbaum point}

Among all tuning operators is the operator $\tau_W$ where $W$ is the basilica component 
of period $2$ (the associated strings are $\Sigma_0 = 01$, $\Sigma_1 = 10$). We will denote this particular operator 
simply with $\tau$. The fixed point of $\tau$ is the external angle of the \emph{Feigenbaum point} $c_{Feig}$.

Let us explicitly compute the dimension at the Feigenbaum parameter. Indeed, let $c_0$ be the airplane 
parameter of angle $\theta_0 = 3/7$, and consider the sequence of parameters of angles $\theta_n := \tau^n(\theta_0)$ 
given by successive tuning.

The set $H_{c_0}$ is given by all angles with binary sequences which do not contain $3$ consecutive equal symbols, 
hence the Hausdorff dimension is easily computable (see example 4 in the introduction):
$$\textup{H.dim }H_{\theta_0} = \log_2 \frac{\sqrt{5} + 1}{2}.$$
Now, by repeated application of Proposition \ref{tuneddim} we have 
$$\textup{H.dim }H_{\theta_n} = \frac{\textup{H.dim }H_{\theta_0}}{2^n}.$$
Note that the angles $\theta_n$ converge from above to the Feigenbaum angle $\theta_F$, also $\textup{H.dim }H_{c_{Feig}} = 0$; 
moreover, since $\theta_n$ is periodic of period $2^n$, 
$$ \theta_n - \theta_F \asymp 2^{-2^n} $$
and together with 
\begin{equation} \label{eqfeig}
\textup{H.dim }H_{\theta_n} - \textup{H.dim }H_{\theta_F} = \frac{\textup{H.dim }H_{\theta_0}}{2^n} 
\end{equation}
we have proved the

\begin{proposition} \label{feig}
For the Feigenbaum parameter $c_{Feig}$ we have
$$\textup{H.dim }S_{c_{Feig}} = 0$$
and moreover, the entropy function $\theta \mapsto h(\theta)$ is not H\"older-continuous at the Feigenbaum point.
Similarly, the dimension of the set of biaccessible angles for the Feigenbaum parameter is $0$. 
\end{proposition}

Note that it also follows that the entropy $h(c):= h_{top}(f_c, [-\beta, \beta])$ as a function of the parameter $c$ 
has vertical tangent at $c = c_{Feig}$, as shown in Figure \ref{MTentropy}. Indeed, if $c_n \to c_{Feig}$ is the 
sequence of period doubling parameters converging to the Feigenbaum point, it is a deep result 
\cite{Ly2} that $|c_n - c_{Feig}| \asymp \lambda^{-n}$, 
where $\lambda \cong 4.6692\dots$ is the Feigenbaum constant; hence, by equation \eqref{eqfeig}, we have 
$$\frac{h(c_n) - h(c_{Feig})}{|c_n - c_{Feig}|} \asymp \left(\frac{\lambda}{2}\right)^n \to \infty.$$

\subsection{Proof of Theorem \ref{equaldim}}

Let us now turn to the proof of equality of dimensions between $H_c$ and $P_c$.
Recall we already established $P_c \subseteq H_c$, hence we are left with proving 
that for all real parameters $c \in \partial \mathcal{M} \cap \mathbb{R}$, 
$$\textup{H.dim }H_c \leq \textup{H.dim }P_c.$$
By Proposition \ref{feig}, the inequality holds for the Feigenbaum point and for all $c > c_{Feig}$. 
Moreover, by Proposition \ref{embedding} and continuity of entropy (\cite{MT}, see also section \ref{section:knead}), we have 
the inequality for any $c \in \partial \mathcal{M} \cap \mathbb{R}$ which is non-renormalizable.
Let now $\tau$ be the tuning operator whose fixed point is the Feigenbaum point: since the root of its tuning window is 
the basilica map which has zero entropy, by Proposition \ref{tuneddim} we have, for each $n \geq 0$ and each $c \in \mathcal{M}$,
\begin{equation} \label{tt}
\textup{H.dim }H_{\tau^n(c)} = \textup{H.dim }\tau^n(H_c) 
\qquad \textup{H.dim }P_{\tau^n(c)} = \textup{H.dim }\tau^n(P_c). 
\end{equation}
Now, each renormalizable parameter $c \in \mathcal{M} \cap (-2, c_{Feig})$ is either of the form $c = \tau^n(c_0)$
with $c_0$ non-renormalizable, or $c = \tau^n(\tau_W(c_0))$ with $W$ a real hyperbolic component such that its root $r(W)$
is outside the baby Mandelbrot set determined by the image of $\tau$.
\begin{enumerate}
 \item 
In the first case we note that (since tuning operators behave well under the operation of concatenation of binary 
strings), by applying the operator $\tau^n$ to both sides of the inclusion of Proposition \ref{embedding} we get for each 
$c' > c_0$ a piecewise linear map $F_0$ such that
$$F_0(\tau^n(H_{c'})) \subseteq \tau^n(P_{c_0})$$
hence, by continuity of entropy and of tuning operators,
$$\textup{H.dim }H_{c} = \sup_{c' > c_0} \textup{H.dim }H_{\tau^n(c')} = 
\textup{H.dim }\tau^n(H_{c'}) \leq \textup{H.dim }\tau^n (P_{c_0}) = \textup{H.dim }P_{c}.$$
\item
In the latter case $c = \tau^n(\tau_W(c_0))$, by Proposition \ref{tuneddim} we get
$$\tau^n(P_{\tau_W(c_0)}) = \tau^n(P_{r(W)}) \cup \tau^n(\tau_W(P_{c_0}))$$
and since the period of $W$ is larger than $2$ we have the inequality 
$$\textup{H.dim } \tau^n(\tau_W(P_{c_0})) \leq \textup{H.dim } \tau^{n+1}(P_{c_0}) \leq \textup{H.dim } \tau^{n+1}(\mathcal{R}) \leq \tau^n(P_{r(W)})$$
where in the last inequality we used the fact that the set of rays $\tau(\mathcal{R})$ land to the right of the root $r(W)$.
Thus we proved that 
$$\textup{H.dim }\tau^n(P_{\tau_W(c_0)}) = \textup{H.dim } \tau^n(P_{r(W)})$$
and the same reasoning for $H_c$ yields
$$\textup{H.dim }\tau^n(H_{\tau_W(c_0)}) = \textup{H.dim } \tau^n(H_{r(W)}).$$
Finally, putting together the previous equalities with eq.\eqref{tt} and applying the case (1) to $\tau^n(r(W))$ (recall 
$r(W)$ is non-renormalizable), we have the equalities
$$\textup{H.dim }P_{c} = \textup{H.dim }\tau^n(P_{\tau_W(c_0)}) = \textup{H.dim } \tau^n(P_{r(W)}) = \textup{H.dim } P_{\tau^n(r(W))} = $$
$$= \textup{H.dim } H_{\tau^n(r(W))}  = \textup{H.dim } \tau^n(H_{r(W)}) = \textup{H.dim }\tau^n(H_{\tau_W(c_0)}) =  \textup{H.dim }H_{c}.$$
\end{enumerate}


\subsection{Density of dominant parameters} \label{proofdensity}

In order to prove Proposition \ref{densitydominant}, we will need the following definitions:
given a string $S$, the set of its \emph{prefixes-suffixes} is 
$$\begin{array}{lll} PS(S) & := & \{ Y :  Y \textup{ is both a prefix and a suffix of }S\} = \\
   & = & \{ Y :  Y \neq \emptyset, \exists \ X, Z\textup{ s.t. } S = XY = YZ \}.
  \end{array}$$
Note that an extremal string $S$ of even length is dominant if and only if
$PS(S)$ is empty. Moreover, let us define the set of \emph{residual suffixes} as
$$RS(S) := \{ Z \ : \ S = YZ, \ Y \in PS(S) \}.$$
\textit{Proof of Proposition \ref{densitydominant}.}
By density of the roots of the maximal tuning windows in the set of non-renormalizable angles, it is enough 
to prove that every $\theta \in (0, \frac{1}{2})$ which is root of a maximal tuning window, 
$\theta \neq 1/3$, can be approximated from the right by dominant points. Hence we can assume $w_\theta = \overline{S}$, 
$S$ an extremal string of even length, and $1$ is not a prefix of $S$. If $S$ is dominant, a sequence of approximating 
dominant parameters is given by the strings 
$$S^n11, \qquad n \geq 1.$$
The rest of the proof is by induction on $|S|$. If $|S| = 2$, then $S$ itself is dominant and we are in the previous case.
If $|S| > 2$, either $S$ is dominant and we are done, or $PS(S) \neq \emptyset$ and also $RS(S) \neq \emptyset$. Let us choose $Z_\star \in RS(S)$ 
such that $$\overline{Z_\star} := \min \{ \overline{Z} \ : \ Z \in RS(S) \}$$
and $Y_\star \in PS(S)$ such that $S = Y_\star Z_\star$. Let $\alpha(Y_\star)$ be the root 
of the maximal tuning window $\overline{Y_\star}$ belongs to. Then by Lemma  \ref{largerthanprim},  
$\overline{Z_\star} > \alpha(Y_\star)$, and by minimality
$$\alpha(Y_\star) < \overline{Z} \quad \forall Z \in RS(S).$$
Now, since $Y_\star$ has odd length and belongs to the window of root $\alpha(Y_\star)$, 
then one can write $\alpha(Y_\star) = \overline{P}$ with $Y_\star << P$, hence also $S << P$.
Moreover, 
$$|P| \leq |Y_\star| + 1 \leq |S|$$
and actually $|Y_\star| + 1 < |S|$ because otherwise the first digit of $Y_\star$ would appear twice at the beginning of $S$, contradicting the 
fact that $S$ is extremal. Suppose now $\alpha(Y_\star)  \neq \overline{1}$. Then 
 $|P| < |S|$ and by induction there exists $\gamma = \overline{T}$ such that $T$ is dominant, 
$$ \alpha(Y_\star) <  \overline{T} < \overline{Z} \quad \forall Z \in RS(S)$$
and $\gamma$ can also be chosen close enough to $\alpha(Y_\star)$ so that $P$ is prefix of $T$, which implies
$$S << T.$$
By Lemma \ref{inductivestep}, $S^nT^m$ is a dominant string for $m$ large enough, of even length if $m$ is even, and arbitrarily close to $\overline{S}$ 
as $n$ tends to infinity. If $\alpha(Y_\star) = \overline{1}$, the string $S^n1^{2m}$ is also dominant for $n$, $m$ large enough.
\qed

\begin{lemma} \label{prefsuffisodd}
If $S$ is an extremal string and $Y \in PS(S)$, then $Y$ is an extremal string of odd length. 
\end{lemma}

\begin{proof}
Suppose $S = XY = YZ$. Then by extremality $XY < YX$, hence $XYY < YXY$ and, by substituting $YZ$ for $XY$, $YZY < YYZ$.
If $|Y|$ were even, it would follow that $ZY < YZ$, which contradicts the extremality of $S = YZ < ZY$. Hence $|Y|$ is odd.
Suppose now $Y = AB$, with $A$ and $B$ non-empty strings. Then $S = XAB < BXA$. By considering the first $k := |Y|$ characters 
on both sides of this equation, $Y = AB = S_1^k \leq (BXA)_1^k = BA$. If $Y = AB = BA$, then $Y = P^k$ for some string $P$, 
hence by Lemma \ref{stringlemma} we have $P Z P^{k-1} < P^kZ = S$, which contradicts 
the extremality of $S$, hence $AB < BA$ and $Y$ is extremal.
\end{proof}

\begin{lemma}\label{inductivestep}
Let $S$ be an extremal string of even length, and $T$ be a dominant string. Suppose moreover that
\begin{enumerate}
  \item  $S << T$;
  \item $\overline{T} < \overline{Z}$ $\ \forall Z \in RS(S)$.
\end{enumerate}
Then, for any $n \geq 1$ and for $m$ sufficiently large, $S^n T^m$ is a dominant string. 
\end{lemma}

\begin{proof} Let us check that $S^nT^m$ by checking all its splittings. We have four cases:
\begin{enumerate}
 \item From (1), we have 
$$S^n T^m << T^a, \quad a \geq 1$$
$$S^n T^m << S^b T^m, \quad b < n.$$
\item If $S = XY$, $XY << YX$ by extremality, hence
$$S^n T^m << Y S^b T^m \quad \forall b \geq 1.$$
\item Since $T$ is dominant, $T << U$ whenever $T = QU$, thus
$$S^n T^m << T << U.$$
\item One is left to prove that $S^n T^m << Y T^m$ whenever $S = XY$.
If $Y \notin PS(S)$, then $XY << Y$ and the proof is complete. Otherwise, 
$S = XY = YZ$, $|Y| \equiv 1 \mod 2$ by Lemma \ref{prefsuffisodd}. Moreover, since 
$YZ < ZY$, by a few repeated applications of Lemma \ref{stringlemma}, we have $\overline{Z S^{n-1}} > \overline{Z}$, hence (2) 
implies $ \overline{T} < \overline{Z S^{n-1}}$, and by Lemma \ref{newstringlemma} we have
$Z S^{n-1} \overline{T} > \overline{T}$, hence for $m$ large enough
$Z S^{n-1} T^m >> T^m$ and then 
$$S^nT^m << Y T^m.$$ 
\end{enumerate}

\end{proof}

\begin{lemma} \label{newstringlemma}
Let $Y$, $Z$ be finite strings of positive integers such that $\overline{Y} < \overline{Z}$. Then 
$$Z \overline{Y} > \overline{Y}.$$ 
\end{lemma}

\begin{proof}
By Lemma \ref{stringlemma}, for any $k \geq 0$ we have
$$\overline{Y^k} < \overline{Z} \Rightarrow Y^kZ< ZY^k$$
hence, by taking the limit as $k \to \infty$, $Z \overline{Y} \geq \overline{Y}$. Equality cannot hold
because otherwise $Y$ and $Z$ have to be multiple of the same string, which contradicts the strict inequality 
$\overline{Y} < \overline{Z}$.
\end{proof}

\begin{lemma}\label{largerthanprim}
Let $\theta$ be a non-renormalizable, real parameter angle such that $w_\theta = \overline{S}$
and $S$ is an extremal string of even length, and let $Y \in PS(S)$, $S = YZ$. 
Let $\phi$ the parameter angle such that $w_\phi = \overline{Y}$, and let $\Omega = [\omega, \alpha]$ 
be the maximal tuning window which contains $\phi$. Then
if $w_\alpha = \overline{S_0}$, we have  
$$\overline{Z} > \overline{S_0}.$$ 
\end{lemma}

\begin{proof}
Since $\phi$ lies in the tuning window $\Omega$, $Y$ is a concatenation of the strings $S_0$ and $S_1$. 
As a consequence, $Y\overline{S_0}$ is also a concatenation of strings $S_0$ and $S_1$, 
so $Y \overline{S_0} \geq S_1 \overline{S_0}$.
Moreover, by Lemma \ref{stringlemma}, $\overline{S} < \overline{Y} < \overline{S_0}$. 
We now claim that
$$\beta := \overline{ZY} > \overline{S_0}.$$
Indeed, suppose $\beta \leq \overline{S_0}$; then, $\overline{S} = Y\beta \geq Y\overline{S_0} \geq S_1\overline{S_0}$, 
which combined with the fact that $\overline{S} < \overline{S_0}$ implies $\theta$ lies in the tuning window $\Omega$, 
contradicting the fact that $\theta$ is non-renormalizable.

Now, suppose $\overline{Z} \leq \overline{S_0}$; then $\overline{Z} \leq \overline{S_0} \leq \overline{ZY}$, 
which implies $Z$ has to be prefix of $\overline{S_0}$, hence $Z = S_0^k V$ with $V$ prefix of $S_0$, $V \neq \emptyset$ since $|Z|$ is odd. If $S_0 \neq (1, 1)$, 
then $S_0$ is extremal and, by Lemma \ref{stringlemma}, $\overline{Z} = \overline{S_0^k V} > \overline{S_0}$, contradiction. 
In the case $S_0 = (1, 1)$, then $Z$ must be just a sequence of $1$'s of odd length, which forces $\overline{S} = \overline{1}$,
hence $S$ cannot be extremal.
\end{proof}

\section{The complex case}

In the following sections we will develop in detail the tools needed to prove Theorem \ref{mainvein}.
In particular, in section \ref{section:knead} we prove continuity of entropy along 
principal veins by developing a generalization of kneading theory to tree maps. Then 
(section \ref{section:surgery}) we develop the combinatorial surgery map, which maps the combinatorial 
model of real Hubbard trees to Hubbard trees along the vein. Finally (section \ref{section:proofvein}), we use the surgery 
to transfer the inclusion of Hubbard tree in parameter space of section \ref{section:embed} from 
the real vein to the other principal veins. 




\subsection{Veins} \label{section:veins}

A \emph{vein} in the Mandelbrot set is a continuous, injective arc inside $\mathcal{M}$.
Branner and Douady \cite{BD} showed that there exists a vein joining the parameter at angle $\theta = 1/4$ 
to the main cardiod of $\mathcal{M}$. In his thesis, J. Riedl \cite{Ri} showed existence of veins 
connecting any tip at a dyadic angle $\theta = \frac{p}{2^q}$ to the main cardioid. Another proof of this fact is due to 
 J. Kahn (see \cite{DoCompact}, Section V.4, and  \cite{Sch}, Theorem 5.6). Riedl also shows that 
the quasiconformal surgery preserves local connectivity of Julia sets, hence by using the local connectivity 
of real Julia sets \cite{LvS} one concludes that all Julia sets of maps along the dyadic veins are locally connected
(\cite{Ri}, Corollary 6.5) .

Let us now see how to define veins combinatorially just in terms of laminations.
Recall that the quadratic minor lamination $QML$ is the union of all \emph{minor leaves} of all invariant laminations
corresponding to all quadratic polynomials.  
The degenerate leaf $\{0\}$ is the natural \emph{root} of $QML$. No other leaf of $QML$ 
contains the angle $0$ as its endpoint.
Given a rooted lamination, we define a partial order on the set of leaves by 
saying that $\ell_1 < \ell_2$ if $\ell_1$ separates $\ell_2$ from the root.

\begin{definition}
Let $\ell$ be a minor leaf. Then the \emph{combinatorial vein} defined by $\ell$ is the set 
$$P(\ell) := \{ \ell' \in QML \ :  \{0 \} < \ell' \leq \ell \}$$
of leaves which separate $\ell$ from the root of the lamination.
\end{definition}

\subsection{Principal veins}
  
Let $\frac{p}{q}$ be a rational number, with $0 < p < q$ and $p, q$ coprime. The $\frac{p}{q}$-limb in the Mandelbrot set is the set of 
parameters which have rotation number $\frac{p}{q}$ around the $\alpha$ fixed point. In each limb, there exists a
unique parameter $c = c_{p/q}$ such that the critical point maps to the $\beta$ fixed point after exactly $q$ steps, i.e.
$f^q_c(0) = \beta$. For instance, $c_{1/2} = -2$ is the Chebyshev polynomial. These parameters represent the 
``highest antennas'' in the limbs of the Mandelbrot set.
The \emph{principal vein} $v_{p/q}$ is the vein joining $c_{p/q}$ to the main cardioid. We shall denote 
by $\tau_{p/q}$ the external angle of the ray landing at $c_{p/q}$ in parameter space.

\begin{proposition} \label{qstar}
Each parameter $c \in v_{p/q}$ is topologically finite, and the Hubbard tree $T_c$ is a 
$q$-pronged star. Moreover, the valence of any point $x \in T_c$ is at most $2 q$. 
\end{proposition}

\begin{proof}
Let $\tau$ be the point in the Julia set of $f_c$ where the ray at angle $\tau_{p/q}$ lands.
Since $c \in [\alpha, \tau]$, then $f^{q-1}(c) \in [\alpha, \beta]$, hence by Lemma \ref{top_finite}
the extended Hubbard tree is a $q$-pronged star. The unique point with degree 
larger than $1$ is the $\alpha$ fixed point, which has degree $q$, so the second claim follows 
from Lemma \ref{max_val}.
\end{proof}

Note that, by using combinatorial veins, the statement of Theorem \ref{mainvein} can be given in purely combinatorial form
as follows.
Given a set $\lambda$ of leaves in the unit disk, let us denote by $\textup{H.dim }\lambda$ the Hausdorff dimension of the 
set of endpoints of (non-degenerate) leaves of $\lambda$.
Moreover, if the leaf $\ell$ belongs to $QML$ we shall denote as $\lambda(\ell)$ the invariant quadratic lamination which has $\ell$ as minor leaf.
The statement of the theorem then becomes that, for each $\ell \in P(\tau_{p/q})$, the following equality holds:
$$\textup{H.dim }P(\ell) = \textup{H.dim }\lambda(\ell).$$
We conjecture that the same equality holds for every $\ell \in QML$. 

\subsection{A combinatorial bifurcation measure}

The approach to the geometry of the Mandelbrot set via entropy of Hubbard trees 
allows one  to define a transverse measure on the quadratic minor lamination $QML$. 
Let $\ell_1 < \ell_2$ be two ordered leaves of $QML$, corresponding to two parameters $c_1$ and 
$c_2$, and let $\gamma$ be a tranverse arc connecting $\ell_1$ and $\ell_2$. Then one can assign the measure of the 
arc $\gamma$ to be the difference between the entropy of the two Hubbard trees:
$$\mu(\gamma) := h(f_{c_2}\mid_{T_{c_2}}) - h(f_{c_1}\mid_{T_{c_1}}).$$
By the monotonicity result of \cite{TL}, such a measure can be interpreted as a \emph{transverse bifurcation measure}: 
in fact, as one crosses more and more leaves from the center of the Mandelbrot set to the periphery, 
i.e. as the map $f_c$ undergoes more and more bifurcations, one picks up more and more measure. 
The measure can also be interpreted as the derivative of the entropy in the direction transverse to the leaves: note also that, 
since period doubling bifurcations do not change the entropy, $\mu$ is non-atomic.

The dual to the lamination is an $\mathbb{R}$-tree, and the transverse measure $\mu$ defines a metric on such a tree. 
By pushing it forward to the actual Mandelbrot set, one endows the union of all veins in $\mathcal{M}$ with the structure 
of a metric $\mathbb{R}$-tree. It would be very interesting to analyze the properties of such transverse measure, and also 
comparing it to 
the other existing notions of bifurcation measure.

\smallskip

In the following sections we will 
develop the proof of Theorem \ref{mainvein}.





\section{Kneading theory for Hubbard trees} \label{section:knead}
 
In this section we will analyze the symbolic dynamics of some continuous maps of trees, in order 
to compute their entropy as zeros of some power series. As a consequence, we will see that the entropy 
of Hubbard trees varies continuously along principal veins. 
Our work is a generalization to tree maps of Milnor and Thurston's \emph{kneading theory} \cite{MT} for interval maps. 
The general strategy is similar to \cite{BdC}, but our view is towards application to Hubbard trees. Moreover, since
we are mostly interested in principal veins, we will treat in detail only the case of trees with a particular topological 
type. An alternative, independent approach to continuity is in \cite{BS}.

\subsection{Counting laps and entropy}

Let $f : T \to T$ be a continuous map of a finite tree $T$. We will assume $f$ is a local homeomorphism onto its image except at 
one point, which we call the \emph{critical point}. At the critical point, the map is a branched cover of degree $2$.
Let us moreover assume $T$ is a \emph{rooted tree}, i.e. it has a distinguished end $\beta$.
The choice of a root defines a partial ordering on the tree; namely, $x < y$ if $x$ disconnects $y$ from the root.

Let $C_f$ be a finite set of points of $T$ such that $T \setminus C_f$ is a union of disjoint open intervals $I_k$, 
and the map $f$ is monotone on each $I_k$ with respect to the above-mentioned ordering. The critical point and the branch 
points of the tree are included in $C_f$. 

For each subtree $J \subseteq T$, the \emph{number of laps} of the restriction of $f^n$ to $J$ is defined as 
$\ell(f^n \mid_{J}) := \#(J \cap  \bigcup_{i = 0}^{n-1} f^{-i}(C_f) ) + \#\textup{Ends}(J) - 1$, in analogy with the real case.
Denote $\ell(f^n) := \ell(f^n \mid_T)$. The \emph{growth number} $s$ of the map $f : T \to T$ is the exponential growth rate of the number of laps:

\begin{equation} \label{entrodef}
s := \lim_{n \to \infty} \sqrt[n]{\ell(f^n)}.
\end{equation}

\begin{lemma}[\cite{BdC}, Lemma 4.1]
The limit in eq. \eqref{entrodef} exists, and it is related to the topological entropy $h_{top}(f\mid_T)$
in the following way: 
$$s = e^{h_{top}(f\mid_T)}.$$ 
\end{lemma}

The proof is the same as in the analogous result of Misiurewicz and Szlenk for interval maps (\cite{dMvS}, Theorem II.7.2).
In order to compute the entropy of $f$, let us define the generating function 
$$\mathcal{L}(t) := 1 + \sum_{n = 1}^\infty \ell(f^n) t^{n}$$
where $\ell(f^n)$ is the number of laps of $f^n$ on all $T$. 
Moreover, for $a, b \in T$, let us denote as 
$\ell(f^n\mid_{[a, b]})$ the number of laps of the restriction of $f^n$  to the interval $[a, b]$.
Thus we can construct for each $x \in T$ the function 
$$\mathcal{L}(x, t) := 1 + \sum_{n = 1}^\infty \ell(f^n \mid_{[\beta, x]}) t^{n}$$
and for each $n$ we shall denote $L_{n, x} := \ell(f^n\mid_{[\beta, x]})$.
Let us now relate the generating function $\mathcal{L}$ to the kneading sequence.

Before doing so, let us introduce some notation;  for $x \notin C_f$, the sign 
$\epsilon(x) \in \{ \pm 1\}$ is defined according as to whether $f$ preserves or 
reverses the orientation of a neighbourhood of $x$. Finally, let us define
$$\eta_{k}(x) := \epsilon(x) \cdots \epsilon(f^{k-1} (x))$$
for $k \geq 1$, and $\eta_0(x) := 1$.
Moreover, let us introduce the notation 
$$\chi_{k}(x) := \left\{ \begin{array}{ll}
			      1 & \textup{if }f(x) \in I_k \\
			      0 & \textup{if }f(x) \notin I_k 
                            \end{array}
		    \right.
$$
and $\hat{\chi}_{k}(x) := 1 - \chi_{k}(x)$.

Let us now focus on the case when $T$ is the Hubbard tree of a quadratic polynomial along the principal vein $v_{p/q}$.
Then we can set $C_f := \{ \alpha, 0 \}$ the union of the $\alpha$ fixed point and the critical point, so that 
$$T \setminus C_f = I_0 \cup I_1 \cup \dots \cup I_q$$
where the critical point separates $I_0$ and $I_1$, and the $\alpha$ fixed point separates $I_1, I_2, \dots, I_q$. 
The dynamics is the following:
\begin{itemize}
 \item $f: I_k \mapsto I_{k+1}$  homeomorphically, for $1 \leq k \leq q-1$;
\item $f: I_q \mapsto I_0 \cup I_1$ homeomorphically;
\item $f(I_0) \subseteq I_0 \cup I_1 \cup I_2$.
\end{itemize}

We shall now write a formula to compute the entropy of $f$ on the tree as a function of the itinerary of the 
critical value. 

\begin{proposition} \label{c_fmla}
Suppose the critical point for $f$ is not periodic. Then we have the equality 
$$\mathcal{L}(c, t)\left[ 1 - 2t \Theta_1(t) + \frac{4t^2}{1+t} \Theta_2(t) \right] = \Theta_3(t)$$
as formal power series, where 
$$\Theta_1(t) := \sum_{k = 0}^\infty \eta_{k}(c) \hat{\chi}_{0}(f^k(c)) t^k$$
$$\Theta_2(t) := \sum_{k = 0}^\infty \eta_{k}(c) \chi_{2}(f^k(c)) t^k$$
depend only on the itinerary of the critical value $c$, and $\Theta_3(t)$ is some power series with real, non-negative, bounded coefficients.
(Note that, in order to deal with the prefixed case, we extend the definitions of $\epsilon$, $\hat{\chi}_0$ and $\chi_2$ by setting 
$\epsilon(\alpha) = \hat{\chi}_0(\alpha) = \chi_2(\alpha) = 1$.)
\end{proposition}

\begin{proof}
We can compute the number of laps recursively. Let us suppose $x \in T$ such that $f^n(x) \neq 0$ for all $n \geq 0$.
Then for $n \geq 2$ we have the following formulas: 
$$\ell(f^n\mid_{[\beta, x]}) = \left\{ \begin{array}{ll} \ell(f^{n-1}\mid_{[\beta, f(x)]}) & \textup{if } x \in I_0 \cup \{0\} \\
							 - \ell(f^{n-1}\mid_{[\beta, f(x)]}) + 2 \ell(f^{n-1}\mid_{[\beta, c]}) + 1  & \textup{if } x\in I_1 \\
				\ell(f^{n-1}\mid_{[\beta, f(x)]}) + 2 \ell(f^{n-1}\mid_{[\beta, c]}) - 2 \ell(f^{n-1}\mid_{[\beta, \alpha]}) & \textup{if } x \in I_2 \cup \dots \cup I_{q-1} \cup \{\alpha\} \\
				- \ell(f^{n-1}\mid_{[\beta, f(x)]}) + 2 \ell(f^{n-1}\mid_{[\beta, c]}) + 1 & \textup{if } x \in I_q 
\end{array} \right. $$
Now, recalling the notation $L_{n,x} := \ell(f^n\mid_{[\beta, x]})$, the previous formula can be rewritten as 
$$L_{n, x} = \epsilon(x)L_{n-1, f(x)} + 2 \hat{\chi}_{0}(x) L_{n-1, c} - 2 \chi_{2}(x) L_{n-1, \alpha} + \frac{1- \epsilon(x)}{2}.$$
Moreover, for $n = 1$ we have 
$$L_{1, x} = \epsilon(x) + 2 \hat{\chi}_{0}(x) +\frac{1-\epsilon(x)}{2} + R(x) $$
where 
$$R(x) := \left\{ \begin{array}{ll}
                   1 & \textup{if } x\in I_q \\
		   -1 & \textup{if } x = \alpha \\	
		   0 & \textup{otherwise.} 
                  \end{array} \right.$$
Hence by multiplying every term by $t^n$ and summing up we get
$$\mathcal{L}(x, t) = t \epsilon(x) \mathcal{L}(f(x), t) + 2t \hat{\chi}_0(x) \mathcal{L}(c, t) - 2t \chi_2(x) \tilde{\mathcal{L}}(\alpha, t) + S(x, t) $$
with $S(x,t) :=  \frac{1-\epsilon(x)}{2}\frac{t}{1-t} + tR(x) +1 $. If we now apply the formula to $f^k(x)$ and multiply everything by $\eta_k(x) t^k$ 
we have for each $k \geq 0$

$$\eta_k(x) t^k \mathcal{L}(f^k(x), t) - \eta_k(x) \epsilon(f^k(x)) t^{k+1} \mathcal{L}(f^{k+1}(x), t) =  $$
$$ = 2t^{k+1} \eta_k(x) \hat{\chi}_0(f^{k}(x)) \mathcal{L}(c, t)
 - 2t^{k+1} \eta_k(x) \chi_2(f^{k}(x)) \tilde{\mathcal{L}}(\alpha, t) + \eta_k(x) t^k S(f^k(x), t) $$
so, by summing over all $k \geq 0$, the left hand side is a telescopic series and we are left with 
\begin{equation} \label{eq_Lxt}
\mathcal{L}(x, t) = 2t \Theta_1(x, t) \mathcal{L}(c, t) - 2t \Theta_2(x, t) \tilde{\mathcal{L}}(\alpha, t)
+ \Theta_3(x, t)
\end{equation}
where we used the notation $\tilde{\mathcal{L}}(x, t):= \sum_{n = 1}^\infty \ell(f^n\mid_{[\beta, x]}) t^n$ and

$$ \Theta_3(x, t) :=  \sum_{k = 0}^\infty \eta_k(x) S(f^k(x), t) t^k = 1 + \sum_{k = 1}^\infty \frac{1 + \eta_{k-1}(x)(\epsilon(f^{k-1}(x)) + 2 R(f^{k-1}(x)))}{2} t^k$$
is a power series whose coefficients are all real and lie between $0$ and $1$.
The claim now follows by plugging in the value $x = c$ in eq. \eqref{eq_Lxt}, and using Lemma \ref{alpha_fmlas} to 
write $\tilde{\mathcal{L}}(\alpha, t)$ in terms of $\mathcal{L}(c, t)$.
\end{proof}

\begin{lemma} \label{alpha_fmlas}
We have the following equalities of formal power series:
\begin{enumerate}
 \item 
$$\tilde{\mathcal{L}}(\alpha, t) = \frac{ 2 t \mathcal{L}(c, t)}{1 + t}$$
\item
$$\mathcal{L}(t) t^{q-1} = \frac{(1-t^q) \mathcal{L}(c, t)}{1+t} + P(t)$$
where $P(t)$ is a polynomial.
\end{enumerate}
\end{lemma}

\begin{proof}
(1) We can compute $\ell(f^n \mid_{[\beta, \alpha]})$ recursively, since we have for $n \geq 2$
$$\ell(f^n\mid_{[\beta, \alpha]}) = 2 \ell(f^{n-1}\mid_{[\beta, c]}) - \ell(f^{n-1}\mid_{[\beta, \alpha]})$$ 
while $\ell(f\mid_{[\beta, \alpha]}) = 2$, hence by multiplying each side by $t^n$ and summing over $n$ we get
$$ \tilde{\mathcal{L}}(\alpha, t) = 2t \mathcal{L}(c, t) - t \tilde{\mathcal{L}}(\alpha, t)$$
and the claim holds.

(2) If we let $\mathcal{L}_{[\alpha, c]}(t) := 1 + \sum_{n = 1}^\infty \ell(f^n\mid_{[\alpha, c]}) t^n$, we have by (1) that
$$\mathcal{L}_{[\alpha, c]}(t) = \frac{(1-t) \mathcal{L}(c, t)}{1+t}.$$
Now, since the Hubbard tree can be written as the union $T = \bigcup_{i = 0}^{q-1} [\alpha, f^i(c)]$, for each $n \geq 1$ we have 
$$\ell(f^n\mid_T) = \sum_{i = 0}^{q-1} \ell(f^n \mid_{[\alpha, f^i(c)]}) = \sum_{i = 0}^{q-1} \ell(f^{n+i}\mid_{[\alpha, c]})$$
hence multiplying both sides by $t^{n+q-1}$ and summing over $n$ we get
$$\mathcal{L}(t) t^{q-1} = (1 + t + \dots + t^{q-1})\mathcal{L}_{[\alpha, c]}(t) + P(t)$$
for some polynomial $P(t)$. The claim follows by substituting $\mathcal{L}_{[\alpha, c]}(t)$ using (1).
\end{proof}

\begin{proposition} \label{entro_root}
Let $s$ be the growth number of the tree map $f : T \to T$. If $s > 1$, then the smallest 
positive, real zero of the function
$$\Delta(t) :=  1+t  - 2 t (1+t) \Theta_1(t) + 4t^2 \Theta_2(t)$$
lies at $t = \frac{1}{s}$. If $s = 1$, then $\Delta(t)$ has no zeros inside the interval $(0, 1)$.
\end{proposition}

\begin{proof}
Recall $s := \lim_{n \to \infty} \sqrt[n]{\ell(f^n) }$, so the convergence radius of the series $\mathcal{L}(t)$ 
is precisely $r = \frac{1}{s}$. By Proposition \ref{c_fmla}, $$\mathcal{L}(c, t) = \frac{\Theta_3(t)(1+t)}{\Delta(t)}$$
can be continued to a meromorphic function in the unit disk, and by Lemma \ref{alpha_fmlas}, 
also $\mathcal{L}(t)$ can be continued to a meromorphic function in the unit disk, and 
the set of poles of the two functions inside the unit disk coincide (note both power series expansions begin 
with $1$, hence they do not vanish at $0$).

Let us now assume $s > 1$. Then $\mathcal{L}(c, t)$ must have a pole on the circle $|t| = \frac{1}{s}$, and since 
the coefficients of its power series are all positive, it must have a pole on the positive real axis.
This implies $\Delta(1/s) = 0$. Moreover, since $\Theta_3(t)$ has real non-negative coefficients, it cannot vanish 
on the positive real axis, hence $\Delta(t) \neq 0$ for $0 < t < 1/s$.  

If instead $s = 1$, $\mathcal{L}(c, t)$ is holomorphic on the disk, so for the same reason 
$\Delta(t)$ cannot vanish inside the interval $(0,1)$.
\end{proof}

\subsection{Continuity of entropy along veins}

\begin{theorem} \label{continuous}
Let $v = v_{p/q}$ be the principal vein in the $p/q$-limb of the Mandelbrot set. Then the entropy 
$h_{top}(f_c \mid_{T_c})$ of $f_c$ restricted to its Hubbard tree depends continuously, as $c$ moves along the vein, 
on the angle of the external ray landing at $c$.
\end{theorem}

\begin{proof}
Let $\ell \in P(\tau_{p/q})$ be the minor leaf associated to the parameter $c \in \partial \mathcal{M}$, $\ell = (\theta^-, \theta^+)$.
Since the entropy does not change under period doubling, we may assume that $c$ is not the period doubling of some other 
parameter along the vein; thus, there exist $\{\ell_n\}_{n \geq 1} \subseteq P(\tau_{p/q})$ a sequence of leaves of $QML$ 
which tends to $\ell$.
Since $c \in \partial \mathcal{M}$, the orbit $f_c^n(0)$ never goes back to $0$, so we can apply Propositions \ref{c_fmla} and 
\ref{entro_root}. 
Thus we can write
\begin{equation} \label{quotient}
\mathcal{L}(c, t) = \frac{F(t)}{\Delta(t)} 
\end{equation}
and the entropy $h_{top}(f_c\mid_{T_c})$ is then $\log s$, where $1/s$ is the smallest real positive root 
of $\Delta(t)$. Finally note that both $F(t)$ and $\mathcal{L}(c,t)$ have real non-negative coefficients, and do not vanish 
at $t = 0$. 
The coefficients of $\Delta(t)$ and $F(t)$ depend on the coefficients of $\Theta_1(t)$, $\Theta_2(t)$ and $\Theta_3(t)$, which in turn
 depend only on the itinerary of the angle $\theta^-$ with respect 
to the doubling map $D$ and the partition given by the complement, in the unit circle, of the set
$$\{\theta_1, \dots, \theta_q, \tau_{p/q}, \tau_{p/q} + 1/2 \}$$  
where $\theta_1, \dots, \theta_q$ are the angles of rays landing on the $\alpha$ fixed point. Let $\Delta_n(t), F_n(t)$ denote 
the functions $\Delta(t), F(t)$ of equation \ref{quotient} relative to the parameter corresponding to the leaf $\ell_n$.
If $f_c^n(0) \neq \alpha$ for all $n \geq 0$, then $D^n(\theta^-)$ always lies in the interior of the partition, 
so if $\theta_n^-$ is sufficiently close to $\theta^-$, its itinerary will share a large initial part with the itinerary 
of $\theta^-$, hence the power series for $\Delta(t)$ and $\Delta_n(t)$ share 
arbitrarily many initial coefficients and their coefficients are uniformly bounded, 
so $\Delta_n(t)$ converges uniformly on compact subsets of the disk to $\Delta(t)$, and similarly $F_n(t) \to F(t)$.
Let us now suppose, possibly after passing to a subsequence, that $s_n^{-1} \to s_*^{-1}$. Then by uniform convergeence
on compact subsets of $\mathbb{D}$, $s_*^{-1}$ is either $1$ or a real, non-negative root of $\Delta(t)$, so in either case 
$$\liminf_{n \to \infty} s_n^{-1} \geq s^{-1}.$$
Now, if we have $s_*^{-1} < s^{-1}$, then by Rouch\'e's theorem $\Delta_n$ must have a non-real zero $z_n$ inside the 
disk of radius $s_n^{-1}$ with $z_n \to s_*^{-1}$, hence by definition of $s_n$ and equation \ref{quotient} one also has 
$F_n(z_n) = 0$, but since $F$ has real coefficients then also its conjugate $\overline{z_n}$ is a zero of $F_n$, hence 
in the limit $s_*^{-1}$ is a real, non-negative zero of $F$ with multiplicity two, but this is a contradiction because 
the derivative $F'(t)$ also has real, non-negative coefficients so it does not vanish on the interval $[0, 1)$.
This proves the claim 
$$\lim_{n \to \infty} s_n^{-1} = s^{-1}$$
and continuity of entropy follows.

Things get a bit more complicated when some iterate $f_c^n(0)$ maps to the $\alpha$ fixed point. In this case, the 
iterates of $\theta$ under the doubling map hit the boundary of the partition, hence its itinerary is no longer stable under perturbation. 
However, a simple check proves that even in this case the coefficients for the function $\Delta_n(t)$ still converge 
to the coefficients of $\Delta(t)$. Indeed, if $n$ is the smallest step $k$ such that $f_c^k(c) = \alpha$, then for each $k \geq n$ 
we have $\epsilon(f_c^k(c)) = \hat{\chi}_0(f_c^k(c)) = \chi_2(f_c^k(c)) = 1$. On the other hand, as $\theta_n^-$ tends to $\theta^-$, 
the itinerary of the critical value with respect to the partition $I_0 \cup I_1 \cup \dots \cup I_q$ approaches a preperiodic cycle of period $q$,
where the period is either $(\overline{I_2, I_2, \dots, I_2, I_3, I_1})$ or $(\overline{I_1, I_2, I_2, \dots, I_2, I_3})$. 
In both cases one can check by explicit computation that 
the coefficients in the power series expansion of $\Delta_n(t)$ converge to the coefficients of $\Delta(t)$. 
\end{proof}

\section{Combinatorial surgery} \label{section:surgery}

The goal of this section is to transfer the result about the real line to 
the principal veins $v_{p/q}$; in order to do so, we will define
a surgery map (inspired by the construction of Branner-Douady \cite{BD} for the $1/3$-limb)
which carries the combinatorial principal vein in the real limb to the combinatorial 
principal vein in the $p/q$-limb.

\subsection{Orbit portraits}

Let $0 < p < q$, with $p, q$ coprime. There exists a unique set $C_{p/q}$ of $q$ points on the unit circle which is invariant for the doubling map $D$ and
such that the restriction of $D$ on $C_{p/q}$ preserves the cyclic order of the elements and acts as a rotation of angle $p/q$.
That is $C_{p/q} = \{x_1, \dots, x_q \}$, where $0 \leq x_1 < x_2 < \dots < x_q < 1$ 
are such that $D(x_i) = x_{i+p}$ (where the indices are computed mod $q$).

The $p/q$\emph{-limb} in the Mandelbrot set is the set of parameters $c$ for which the set of angles 
of rays landing on the $\alpha$ fixed point in the dynamical plane for $f_c$ is precisely $C_{p/q}$ 
 (for a reference, see \cite{Mi}). In Milnor's terminology, the set $C_{p/q}$ is an \emph{orbit portrait}: we shall call it 
the \emph{$\alpha$ portrait}.

Given $p/q$, there are exactly two rays landing on the intersection of the $p/q$-limb with the main cardioid:
let us denote these two rays as $\theta_0$ and $\theta_1$.
The angle $\theta_0$ can be found by computing the symbolic coding of the point $p/q$ with respect to the 
rotation of angle $p/q$ on the circle and using the following partition:
$$A_0 := \left(0, 1-\frac{p}{q}\right] \quad A_1 := \left(1-\frac{p}{q}, 1\right].$$
For instance, if $p/q = 2/5$, we have that the orbit is $(2/5, 4/5, 1/5, 3/5, 0)$, hence 
the itinerary is $(0, 1, 0, 0, 1)$ and the angle is $\theta_0 = 0.\overline{01001} = 9/31$.
The other angle $\theta_1$ is obtained by the same algorithm but using the partition:
$$A_0 := \left[0, 1-\frac{p}{q}\right) \quad A_1 := \left[1-\frac{p}{q}, 1\right)$$
(hence if $p/q = 2/5$, we have the itinerary $(0, 1, 0, 1, 0)$ and $\theta_1 = 0.\overline{01010} = 10/31$.)
Let us denote as $\Sigma_0$ the first $q-1$ binary digits of the expansion of $\theta_0$, 
and $\Sigma_1$ the first $q-1$ digits of the expansion of $\theta_1$.

\subsection{The surgery map}

Branner and Douady \cite{BD} constructed a continuous embedding of the $1/2$-limb of the Mandelbrot set into 
the $1/3$-limb, by surgery in the dynamical plane. The image of the real line under this surgery map
is a continuous arc inside the Mandelbrot set, joining the parameter at angle $\theta = 1/4$ with the cusp of $\mathcal{M}$.
Let us now describe, for each $p/q$-limb, the surgery map on a combinatorial level.

In order to construct the surgery map, let us first define the following coding for external angles:
for each $\theta \neq \frac{1}{3}, \frac{2}{3}$, we set
$$A_{p/q}(\theta) := \left\{ \begin{array}{ll} 0 & \textup{if }0 \leq \theta < \frac{1}{3} \\
                              \Sigma_0 & \textup{if }\frac{1}{3} < \theta < \frac{1}{2} \\
			      \Sigma_1 & \textup{if }\frac{1}{2} \leq \theta < \frac{2}{3} \\
				     1 & \textup{if }\frac{2}{3} < \theta < 1.
                             \end{array} \right.$$
Then we can define the following map on the set of external angles:

\begin{definition} 
Let $0 < p < q$, with $p, q$ coprime. The \emph{combinatorial surgery map} $\Psi_{p/q} : \mathbb{R}/\mathbb{Z} \to \mathbb{R}/\mathbb{Z}$ 
is defined on the set of external angles as follows.
\begin{itemize} 
\item If $\theta$ does not land on a preimage of the $\alpha$ fixed point 
(i.e. $D^k(\theta) \neq \frac{1}{3}, \frac{2}{3}$ for all $k \geq 0$), we define $\Psi_{p/q}(\theta)$ as the number with binary expansion
$$\Psi_{p/q}(\theta) := 0. s_1 s_2 s_3 \dots \qquad \textup{with }s_k := A_{p/q}(D^k (\theta)).$$
\item Otherwise, let $h$ be the smallest integer such that $D^h(\theta) \in \{\frac{1}{3}, \frac{2}{3}\}$. Then we define
$$\Psi_{p/q}(\theta) := 0. s_1 s_2 \dots s_{h-1} s_h$$
with $s_k := A_{p/q}(D^k (\theta))$ for $k < h$ and $s_h := \left\{ \begin{array}{ll} \overline{\Sigma_0 1} & \textup{if }D^h(\theta) = \frac{1}{3} \\ 
                                               \overline{\Sigma_1 0} & \textup{if }D^h(\theta) = \frac{2}{3}.  
                                              \end{array} \right.$
\end{itemize}
\end{definition}
Intuitively, the surgery takes the Hubbard tree of a real map, which is a segment, breaks it into two parts
$[c, \alpha]$ and $[\alpha, f(c)]$ and maps them to two different branches of a $q$-pronged star (see Figure \ref{surgery}). 

\begin{figure}[h] 
\includegraphics[width=0.95 \textwidth]{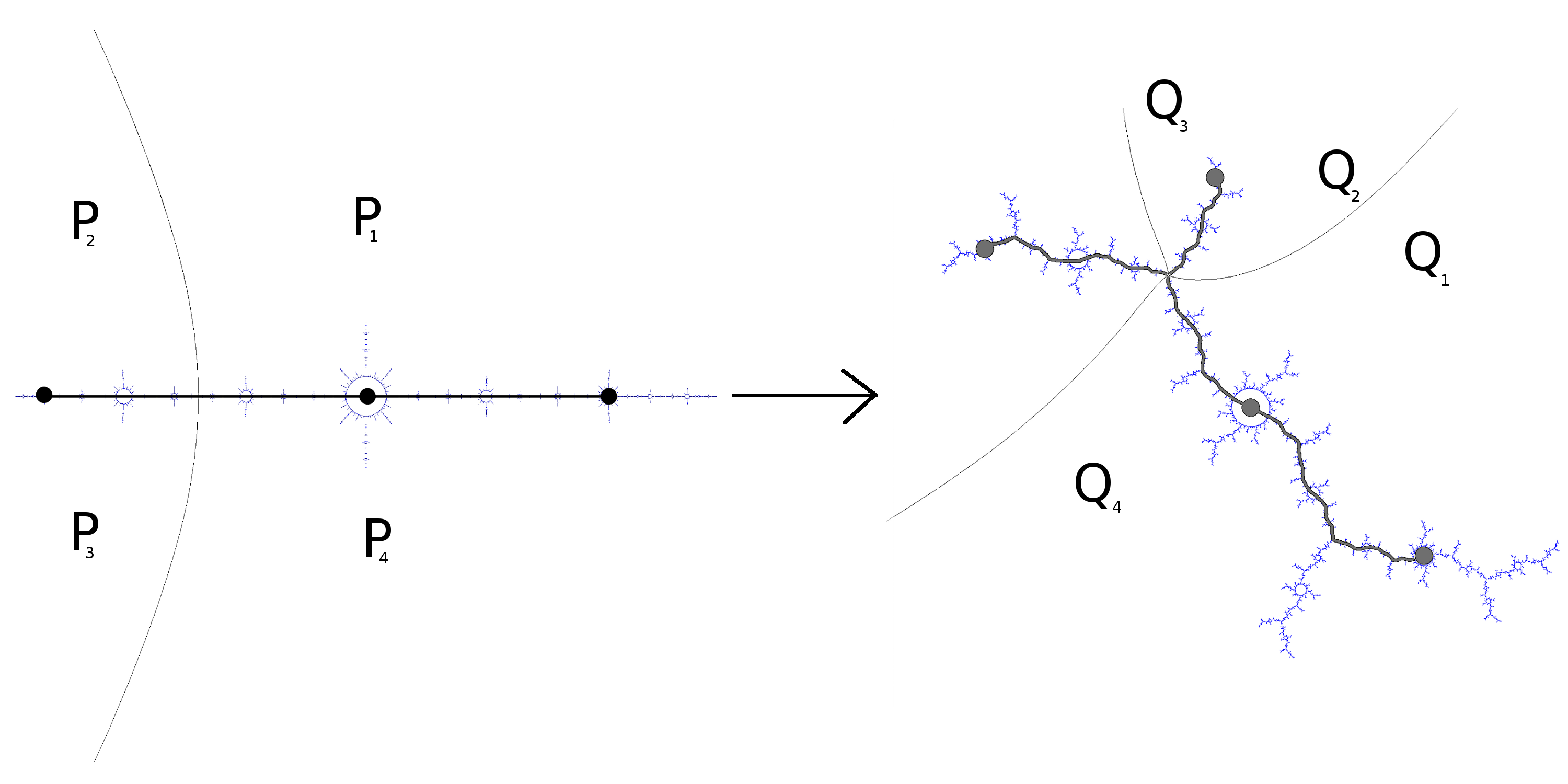}
\caption{The surgery map $\Psi_{1/3}$. The original tree (left) is a segment, which gets ``broken'' at the $\alpha$ fixed point
and a new branch is added so as to form a tripod (right). External rays belonging to the sectors $P_1, P_2, P_3, P_4$ 
are mapped to sectors $Q_1, Q_2, Q_3, Q_4$ respectively.
}   
\label{surgery}
 \end{figure}

The image of $1/2$ under $\Psi_{p/q}$ is the external angle of the ``tip of the highest antenna'' inside the 
$p/q$-limb and is denoted as $\tau_{p/q} := \Psi_{p/q}(1/2) = 0.\Sigma_1$.



Let us now fix a rotation number $p/q$ and denote the surgery map $\Psi_{p/q}$ simply as $\Psi$. 

\begin{lemma}  
The map $\Psi$ is strictly increasing (hence injective), in the sense that if $0 \leq \theta < \theta' < 1$, then 
$0 \leq \Psi(\theta) < \Psi(\theta') < 1$. 
\end{lemma}

\begin{proof}
Let us consider the partitions 
$P_1 := [0, 1/3), P_2 := [1/3, 1/2), P_3 := [1/2, 2/3)$, $P_4 := [2/3, 1)$ and $Q_1 := [0, 0.\overline{0\Sigma_1})$, 
$Q_2 := [0.\overline{\Sigma_0 1}, 0.\Sigma_1)$, $Q_3 := [0.\Sigma_1, 0.\overline{\Sigma_1 0})$,
$Q_4 := [0.\overline{1 \Sigma_0}, 1)$.
It is elementary (even though a bit tedious) to check that the map $\Psi$ respects the partitions, 
in the sense that $\Psi(P_i) \subseteq Q_i$ for each $i = 1, 2, 3, 4$. Indeed, we know 
$$\begin{array}{lll}
D(P_1) & \subseteq & P_1 \cup P_2 \cup P_3 \\
D(P_2) & = & P_4 \\
D(P_3) & = & P_1 \\
D(P_4) & \subseteq & P_2 \cup P_3 \cup P_4
\end{array}$$
so the binary expansion of any element $\Psi(\theta)$ is represented by an infinite path in the graph
$$\xymatrix{
 & \Sigma_0 \ar[dr] & \\
0 \ar@(ul,ur) \ar[ur] \ar[dr] &  & 1 \ar@(ul,ur) \ar[ul] \ar[dl] \\
& \Sigma_1 \ar[ul] &}$$
Let us now check for instance that $\Psi(P_1) \subseteq Q_1$. 
Indeed, if $\theta \in P_1$ then in the above graph the coding of $\varphi(\theta)$ 
starts from $0$ and hence by looking at the graph can be either of the form 
$$\Psi(\theta) = 0.(0\Sigma_1)^k 0^n \Sigma_0 \dots < 0.\overline{0 \Sigma_1} \qquad k \geq 0, n \geq 1$$
or
$$\Psi(\theta) = 0.(0\Sigma_1)^k 0^n \Sigma_1 \dots < 0.\overline{0 \Sigma_1} \qquad k \geq 0, n \geq 2$$
so in both cases $0 \leq \Psi(\theta) < 0.\overline{0 \Sigma}$ and the claim is proven.

Then, given $0 \leq \theta < \theta' < 1$, let $k$ the smallest integer such that $D^k(\theta)$ and $D^k(\theta')$
lie in two different elements of the partition $\bigcup_i P_i$.
Since the map $D^k$ is increasing and the preimage of $0$ lies on the boundary of the 
partition, we have $D^k(\theta) \in P_i$ and $D^k(\theta') \in P_j$ with $i < j$, so 
$\Psi(D^k(\theta)) < \Psi(D^k(\theta'))$ because the first one belongs 
to $Q_i$ and the second one to $Q_j$, hence we have 
$$\Psi(\theta) = 0.s_1 s_2 \dots s_{k-1} \Psi(D^k(\theta)) < 0.s_1 s_2 \dots s_{k-1} \Psi(D^k(\theta'))  = \Psi(\theta').$$
\end{proof}

We can also define the map $\Psi$ on the set of real leaves by defining the image of a leaf to be  
the leaf joining the two images (if $\ell = (\theta_1, \theta_2)$, we set $\Psi(\ell) := (\Psi(\theta_1), 
\Psi(\theta_2))$). From the previous lemma it follows monotonicity on the set of leaves:

\begin{lemma} \label{incr_leaves}
The surgery map $\Psi = \Psi_{p/q}$ is strictly increasing on the set of leaves. 
Indeed, if $\{0 \} \leq \ell_1 < \ell_2 \leq \{1/2\}$, 
then $\{0\} \leq \Psi(\ell_1) < \Psi(\ell_2) \leq \{\tau_{p/q} \}$.
\end{lemma}

Let us now denote by $\Theta_0 := 0.\overline{1 \Sigma_0}$ and $\Theta_1 := 0.\overline{0 \Sigma_1}$
the two preimages of $\theta_0$ and $\theta_1$ which lie in the portrait $C_{p/q}$. Note 
that $D(\Theta_i) = \theta_i$ for $i = 0, 1$.

\subsection{Forbidden intervals}

The leaves $(\theta_0, \theta_1)$ and $(\Theta_0, \Theta_1)$ divide the circle in three parts. Let us 
denote by $\Delta_0$ the part containing $0$, and as $\Delta_1$ the part containing $\tau_{p/q}$.
Moreover, for $2 \leq i \leq q-1$, let us denote $\Delta_i := D^{i-1}(\Delta_1)$. 
With this choice, the intervals $\Delta_0, \Delta_1, \dots, \Delta_{q-1}$ are the connected components of the complement of 
the  $\alpha$ portrait $C_{p/q}$.

Let us also denote by $\hat{C}_{p/q} := C_{p/q} + \frac{1}{2}$ the set of angles of rays landing on the preimage 
of the $\alpha$ fixed point, and $\hat{\Delta}_i := \Delta_i + \frac{1}{2}$ for $0 \leq i \leq q-1$, so that 
$\hat{\Delta}_0, \hat{\Delta}_1, \dots, \hat{\Delta}_{q-1}$ are the connected components of the complement of $\hat{C}_{p/q}$.

\begin{figure} 
\begin{minipage}[b]{0.45 \textwidth}
\includegraphics[scale=0.8]{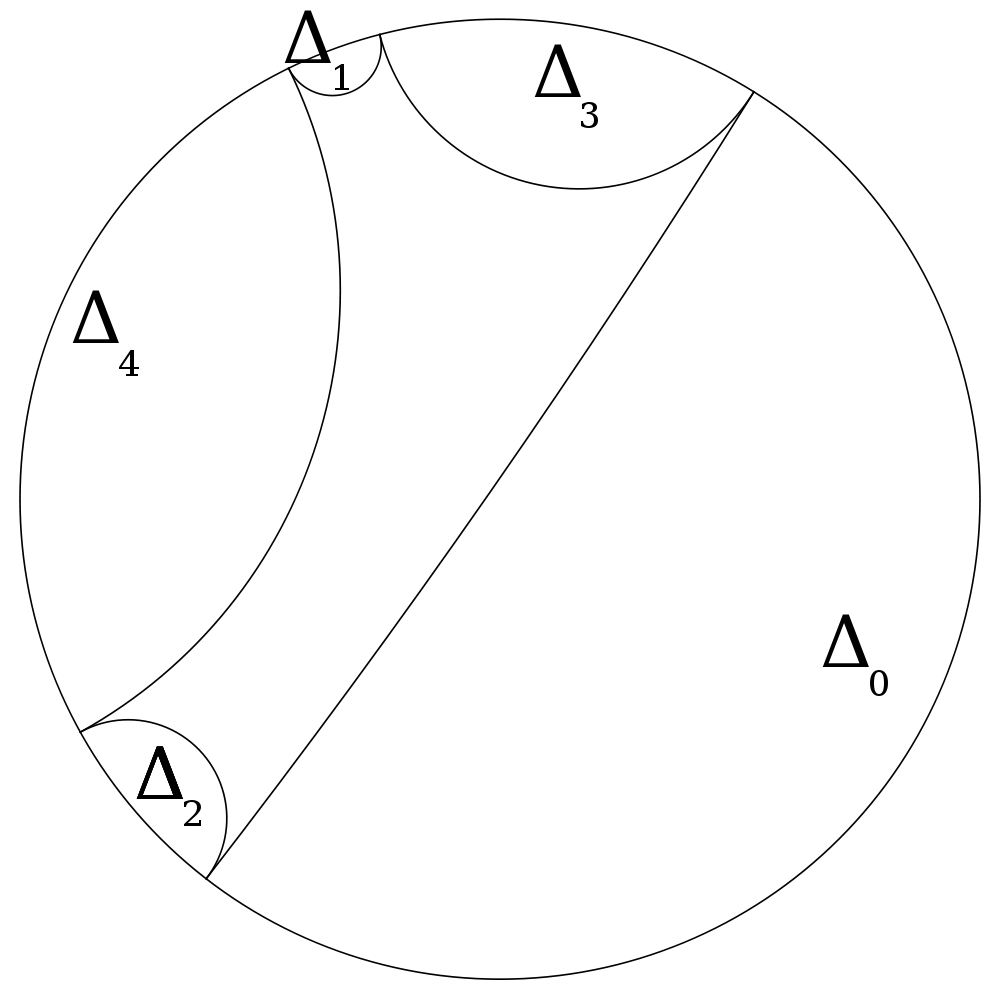}
\end{minipage}
\hfill
\begin{minipage}[b]{0.45 \textwidth}
 \includegraphics[scale=0.8]{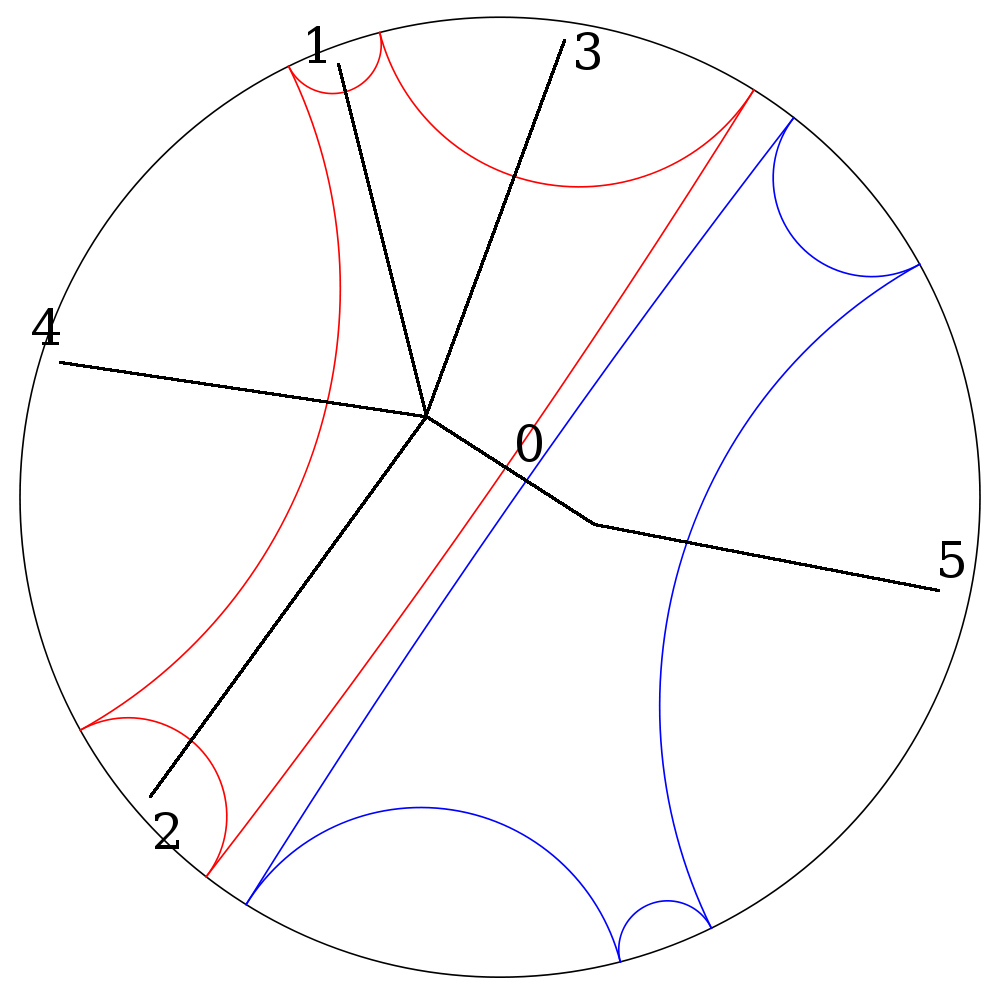}
\end{minipage}
\caption{Left: the $\alpha$ fixed portrait $C_{p/q}$ when $p/q = 2/5$, with the complementary intervals $\Delta_i$. 
Right: The portraits $C_{p/q}$ and $\hat{C}_{p/q}$, with the Hubbard tree drawn as dual to the lamination. The 
numbers indicate the position of the iterates of the critical value.}
\label{for_interval}
\end{figure}

The \emph{forbidden interval} $I_{p/q}$ is then defined as 
$$I_{p/q} := \bigcup_{i = 1}^{q-2} \hat{\Delta}_i.$$
The name ``forbidden interval'' arises from the fact that this interval is avoided by the trajectory of an angle landing 
on the Hubbard tree of some parameter on the vein $v_{p/q}$. Indeed, the following characterization is true:

\begin{proposition} \label{Htree_forbid}
Let $\ell \in P(\tau_{p/q})$ be the characteristic leaf of a parameter $c$ on the principal vein $v_{p/q}$, with 
$\ell = (\theta^-, \theta^+)$, and let $J := (D^{q-1}(\theta^-), D^{q-1}(\theta^+))$ the interval delimited by $D^{q-1}(\ell)$ 
and containing $0$. Then the set of rays landing on the Hubbard tree of $c$ is characterized as
$$H_c := \{ \theta \in S^1 \ : \ D^n(\theta) \notin I_{p/q} \cup J \ \ \forall n \geq 0 \}.$$  
\end{proposition}

\begin{proof}
It follows from the description of $H_c$ in Proposition \ref{symb_Htree} together with the fact that the Hubbard tree is a $q$-pronged star. 
\end{proof}

The explicit characterization also immediately implies that the sets $H_c$ are increasing along principal veins:

\begin{proposition}
Let $\ell < \ell'$ be the characteristic leaves of parameters $c, c'$ which belong to the principal vein $v_{p/q}$.   
\begin{enumerate}
\item Then we have the inclusion
$$H_c \subseteq H_{c'};$$
\item if $T_c$ and $T_{c'}$ are the respective Hubbard trees, we have 
$$h_{top}(f_c\mid_{T_c}) \leq h_{top}(f_c\mid_{T_{c'}}).$$
\end{enumerate}
\end{proposition}

\begin{proof}
(1) Let $J$ be the interval containing $0$ delimited by $D^{q-1}(\ell)$, and $J'$ the interval delimited by $D^{q-1}(\ell')$.
Since $\ell < \ell' < \{\tau_{p/q}\}$, one has $\{0 \} < D^{q-1}(\ell') < D^{q-1}(\ell)$, so $J' \subseteq J$.
If $\theta \in H_c$, then by Proposition \ref{Htree_forbid} its orbit avoids $I_{p/q} \cup J$, hence it also avoids $I_{p/q} \cup J'$
so it must belong to $H_{c'}$.  

(2) From (1) and Theorem \ref{entro_dim}, 
$$h_{top}(f_c\mid_{T_c}) = \textup{H.dim }H_c \cdot \log 2 \leq \textup{H.dim }H_{c'} \cdot \log 2 =  h_{top}(f_{c'}\mid_{T_{c'}}).$$

\end{proof}

Monotonicity of entropy along arbitrary veins is proven, for postcritically finite parameters, in Tao Li's thesis \cite{TL}.
Recently, the following very elegant argument, proving monotonicity along veins without the restriction to postcritically finite parameters, 
was suggested by Tan Lei. 

\begin{proposition}
Let $\ell < \ell'$ be the characteristic leaves of $c, c'$, with $\ell'$ non-degenerate. Then the entropies of the respective
Hubbard trees satisfy the inequality
$$h_{top}(f_c\mid_{T_c}) \leq h_{top}(f_{c'}\mid_{T_{c'}}).$$
\end{proposition}

\begin{proof}
Fix $\ell = (\theta^-, \theta^+)$ a leaf of $QML$, and consider the sets
$$U_1 := \left(\frac{\theta^+}{2}, 1+\frac{\theta^-}{2}\right)$$
$$U_2 := \left(1+\frac{\theta^+}{2}, \frac{\theta^-}{2}\right).$$
Define the set $\Sigma(\ell)$ to be the set of angle pairs whose forward orbit never lies outside $U_1 \cup U_2$, 
and which always have the same itinerary with respect to the partition $U_1 \cup U_2$:
$$\begin{array}{ll} 
\Sigma(\ell) := \{ (\xi, \eta) \in S^1 \times S^1 \ : \ \xi \neq \eta,\   \forall n \geq 0
 & D^n(\eta) \in U_1 \cup U_2,  \\ 
& D^n(\xi) \in U_1 \cup U_2, \\
& D^n(\xi) \in U_i \Leftrightarrow D^n(\eta) \in U_i \}.\\
                                                          \end{array}$$
and denote $U(\ell) := \{ \theta \in S^1 \ : \exists \eta \in S^1 \textup{ s.t. } (\theta, \eta) \in \Sigma(\ell)\}$
the set of endpoints of angle pairs in $\Sigma(\ell)$.
Since the intervals $(\theta^-/2, \theta^+/2)$ and $(1+\theta^-/2, 1+\theta^+/2)$ do not contain 
rays landing on the Hubbard tree, then we have (up to a countable set of angles which hit 
the boundary of the partition) the inclusion
$$H_c \subseteq U(\ell).$$
On the other hand, the two elements of an angle pair $(\xi, \eta)$ which belongs to $\Sigma(\ell)$ 
have the same itinerary with respect to the partition $(\theta^-/2, 1+\theta^-/2) \cup (1+\theta^-/2, \theta^-/2)$, 
so the rays at angles $\xi$ and $\eta$ must land at the same point, which is then biaccessible. This proves 
the inclusion
$$U(\ell) \subseteq B_c,$$
which combined with the previous one implies that $\textup{H.dim }H_c = \textup{H.dim }U(\ell)$.
Now, from the definition it is clear that if $\ell < \ell'$ then $\Sigma(\ell) \subseteq \Sigma(\ell')$, 
since the sets $U_1$ and $U_2$ become larger as $\ell$ gets larger. Hence we have $U(\ell) \subseteq U(\ell')$ and 
$\textup{H.dim }H_c \leq \textup{H.dim }H_{c'}$, which proves the claim.

\end{proof}

\subsection{Surgery in the dynamical and parameter planes}

The usefulness of the surgery map comes from the fact that it maps the real vein in parameter space to the other principal veins, 
and also the Hubbard trees of parameters along the real vein to Hubbard trees along the principal veins. As we will see in this 
subsection, the correspondence is almost bijective. 

Let $Z$ denote the set of angles which never map to the endpoints of fixed leaf $\ell_0 = (1/3, 2/3)$:
$$Z:= \{\theta \in S^1 \ : \ D^n(\theta) \neq 1/3, 2/3 \ \ \forall n \geq 0 \}.$$
Moreover, we denote by $\Omega$ the set of angles which never map to either the forbidden interval $I_{p/q}$ 
or the $\alpha$ portrait $C_{p/q}$:
$$\Omega := \{ \theta \in \Delta_0 \cup \Delta_1 \ : \   D^n(\theta) \notin  I_{p/q} \cup C_{p/q} \ \ \forall n \geq 0 \}.$$
It is easy to check the following 

\begin{lemma} \label{Psi_prop}
The map $\Psi$ is continuous on $Z$, and the image $\Psi(Z)$ is contained in $\Omega$. Given $\theta \in \Omega$, let 
$0 = n_0 < n_1 < n_2 < \dots$ be the return times of $\theta$ to $\Delta_0 \cup \Delta_1$. Then the map
$$\Phi(\theta) := 0.s_0 s_1 s_2 \dots \qquad \textup{with }s_k = \left\{ \begin{array}{ll} 0 & \textup{if }D^{n_k}(\theta) \in [0, \Theta_1) \cup (\theta_0 , \tau_{p/q}) \\
                                                                   1 & \textup{if }D^{n_k}(\theta) \in [\tau_{p/q}, \theta_1) \cup (\Theta_0 , 1) 
                                                                   \end{array} \right.
$$ 
defined on $\Omega$ is an inverse of $\Psi$, in the sense that $\Phi \circ \Psi(\theta) = \theta$ for all $\theta \in Z$. 

\end{lemma}

\begin{proposition} \label{par_surgery}
The surgery map $\Psi = \Psi_{p/q}$ maps the real combinatorial vein 
bijectively onto the principal combinatorial vein $P(\tau_{p/q})$ in the $p/q$-limb, up to a countable set of 
prefixed parameters; indeed, one has the inclusions 
$$P(\tau_{p/q}) \setminus \bigcup_{n \geq 0} D^{-n}(C_{p/q}) \subseteq \Psi(P(1/2)) \subseteq P(\tau_{p/q}).$$
\end{proposition}

\begin{proof}
Let $m \in P(1/2)$ be a minor leaf, and $M_1$, $M_2$ its major leaves. 
By the criterion of Proposition \ref{critQML}, all the elements of the forward orbit 
of $m$ have disjoint interior, and their interior is also disjoint from $m$, $M_1$ and $M_2$, 
so the set of leaves $\{ D^n(m)  \ : \ n\geq 0\} \cup \{ M_1, M_2 \}$ (which may be finite or infinite) is totally ordered, 
and they all lie between $\{0\}$ and $\{1/2\}$. Indeed, they are all smaller than $m$, which is also the shortest 
leaf of the set.
Now, by Lemma \ref{incr_leaves}, the set 
$$\{\Psi(D^n(m)) \ : \ n \geq 0 \} \cup \{ \Psi(M_1), \Psi(M_2) \}$$
is also totally ordered, and all its elements have disjoint interior and lie between $\{0\}$ and $\Psi(m)$. 
Note that all leaves smaller than $\ell_0 := (1/3, 2/3)$ map under $\Psi$ to leaves smaller than 
$(\Theta_0, \Theta_1)$, and all leaves larger than $\ell_0$ map to leaves larger than 
$\Psi(\ell_0) = (\theta_0, \theta_1)$. Note moreover that if a leaf $\mathcal{L}$ is larger than $(\theta_0, \theta_1)$, 
then its length increases under the first $q-1$ iterates (i.e. until it comes back to $\Delta_0$):
$$\ell(D^k(\mathcal{L})) = 2^k \ell(\mathcal{L}) \qquad 0 \leq k \leq q-1.$$
As a consequence, the shortest leaf in the set 
$$S := \{D^n(\Psi(m)) \ : \ n \geq 0 \} \cup \{ \Psi(M_1), \Psi(M_2) \}$$
is $\Psi(m)$, and its images all have disjoint interiors, hence by Proposition \ref{critQML} we have that $\Psi(m)$ 
belongs to $QML$, and it is smaller than $\tau_{p/q}$ by monotonicity of $\Psi$.
Conversely, any leaf $\ell$ of $P(\tau_{p/q})$ whose endpoints never map to the fixed orbit portrait $C_{p/q}$ belongs to $\Omega$, 
hence $\Psi(\ell)$ is well-defined and, since $\Psi$ preserves the ordering, it belongs to $P(1/2)$ by Proposition \ref{critQML}.
\end{proof}

\begin{proposition} \label{dyn_surgery}
Let $c \in [-2, 1/4]$ be a real parameter, with characteristic leaf $\ell$, and let $c'$ be 
a parameter with characteristic leaf $\ell' = \Psi(\ell)$. Moreover, let us set 
$\tilde{H}_{c'} := H_{c'} \cap (\Delta_0 \cup \Delta_1) \setminus \bigcup_{n} D^{-n}(C_{p/q})$. Then
the inclusions 
$$\tilde{H}_{c'} \subseteq \Psi(H_c) \subseteq H_{c'}$$ 
hold. As a consequence, 
$$\textup{H.dim }\Psi(H_c) = \textup{H.dim }H_{c'}.$$
\end{proposition}

\begin{proof}
Let $\theta \in H_c$ and $\ell := (\theta, 1-\theta)$ be its associated real leaf and  
let $\ell_c$ the postcharacteristic leaf for $f_c$. Let us first assume 
$D^n(\theta) \neq 1/3, 2/3$ for all $n$. Then by Lemma \ref{Psi_prop} $\Psi(\theta)$ lies in $\Omega$, 
so its orbit always avoids $I_{p/q}$. Moreover, by Proposition \ref{symb_Htree}
$$D^n(\ell) \geq \ell_c \qquad \textup{for all }n \geq 0.$$
Then, by monotonicity of the surgery map (Lemma \ref{incr_leaves})
$$\Psi(D^n(\ell)) \geq \Psi(\ell_c) \qquad \textup{for all }n \geq 0.$$
Moreover, given $N \geq 0$ either 
$$D^N(\Psi(\ell)) \notin \Delta_0 \cup \Delta_1$$
or one can write
$$D^N(\Psi(\ell)) = \Psi(D^n(\ell))$$
for some integer $n$, so the orbit of $\Psi(\theta)$ always avoids the interval delimited by the leaf $\Psi(\ell_c)$, 
hence by Proposition \ref{symb_Htree} we have $\Psi(\theta) \in H_{c'}$.
The case when $D^n(\theta)$ hits $\{1/3, 2/3\}$ is analogous, except that the leaf $\ell$ 
is eventually mapped to the leaf $(\theta_0, \theta_1)$ which belongs to the $\alpha$ portrait. 

Conversely, let $\theta' \in \tilde{H}_{c'}$ and $\ell'$ be its corresponding leaf. Then by Proposition 
\ref{Htree_forbid} it never maps to $I_{p/q}$, 
so by Lemma \ref{Psi_prop} there exists $\theta \in Z$ such that $\theta' = \Psi(\theta)$. Let $\ell := (\theta, 1-\theta)$ 
be its corresponding real leaf. Moreover, also by Proposition \ref{Htree_forbid}
all iterates of $\ell'$ are larger than $\Psi(\ell_c)$, so by monotonicity of the surgery map all iterates 
of $\ell$ are larger than $\ell_c$, so, by Proposition \ref{symb_Htree}, $\theta$ lies in $H_c$.
The equality of dimensions arises from the fact that for $2 \leq i \leq q-1$ one has
$$H_{c'} \cap \Delta_i = D^{q-1}(H_{c'} \cap \Delta_1)$$
and the doubling map preserves Hausdorff dimension.
\end{proof}

Finally, we need to check that the surgery map behaves well under renormalization. Indeed we have the 

\begin{lemma} \label{surg_renorm}
Let $W$ be a real hyperbolic component, and $\Psi$ the surgery map. Then for each $\theta \in \mathcal{R}$,
$$\Psi(\tau_W(\theta)) = \tau_{\Psi(W)}(\theta)$$
where $\Psi(W)$ is the hyperbolic component whose endpoints are the images via surgery of the endpoints of $W$.
\end{lemma}

\begin{proof}
Let $\theta = 0.\theta_1 \theta_2 \dots$ be the binary expansion of $\theta$. Denote as $\theta^- = 0.\overline{S_0}$, 
$\theta^+ = 0.\overline{S_1}$ the angles of parameter rays landing at the root of $W$, and 
as $\Theta^- := \Psi(\theta^-) = 0.\overline{T_0}$ and $\Theta^+ := \Psi(\theta^+) = 0.\overline{T_1}$ the angles landing at the 
root of $\Psi(W)$. Finally, let $p := |S_0|$ denote the the period of $W$.
Then $\tau_W(\theta)$ has binary expansion 
$$\tau_W(\theta) = 0.S_{\theta_1}S_{\theta_2}\dots$$
By using the fact that either $\theta^- \leq \theta^+ < 1/3$ or $2/3 < \theta^- \leq \theta^+$, 
one checks that for each $0 \leq k < p$, the points
$$D^k(0.S_{\theta_1}S_{\theta_2}\dots)$$
and
$$D^k(0.\overline{S_{\theta_1}})$$
lie in the same element of the partition $\bigcup_{i =1}^4 P_i$. As a consequence, by definition of the surgery map $\Psi$, 
we get that 
$$\Psi(\tau_W(\theta)) = 0.T_{\theta_1}T_{\theta_2}\dots$$
and the claim follows.

\end{proof}

\subsection{Proof of Theorem \ref{mainvein}} \label{section:proofvein}

\begin{definition}
The set $\mathcal{D}_{p/q}$ of \emph{dominant parameters along} $v_{p/q}$ is the image of the set of (real) 
dominant parameters $\mathcal{D}$ under the surgery map:
$$\mathcal{D}_{p/q} := \Psi_{p/q}(\mathcal{D}).$$
\end{definition}




We can now use the surgery map to transfer the inclusion of the Hubbard trees of real maps in the real slice of 
the Mandelbrot set to an inclusion of the Hubbard trees in the set of angles landing on the vein in parameter space.

\begin{proposition} \label{vein_inc}
Let $c \in v_{p/q}$ be a parameter along the vein with non-renormalizable combinatorics, and $c'$ another parameter along 
the vein which separates $c$ from the main cardioid (i.e. if $\ell$ and $\ell'$ are the characteristic leaves, 
$\ell' < \ell \leq \{\tau_{p/q}\}$). Then there exists a piecewise linear map $F : \mathbb{R}/\mathbb{Z} \to 
\mathbb{R}/\mathbb{Z}$ such that 
$$F(\tilde{H}_{c'}) \subseteq P_c.$$
\end{proposition}

\begin{proof}
Let $\theta \in [0, \tau_{p/q}]$ be a characteristic angle for $c$. Let us first assume that the forward orbit of $\theta$ never hits 
$C_{p/q}$. Then by Proposition \ref{par_surgery} there exists an angle $\theta_R \in [0, 1/2] \cap \mathcal{R}$ such that $\theta = \Psi(\theta_R)$, 
and by Lemma \ref{surg_renorm} $\theta_R$ is not renormalizable. Then, by Proposition \ref{embedding},
there exist a $\theta'_R < \theta_R$ arbitrarily close to $\theta_R$ (and by continuity of $\Psi$ we can choose it so that 
$\Psi(\theta'_R)$ lands on the vein closer to $c$ than $c'$) and a piecewise linear map $F_R$ of the circle such that 
\begin{equation} \label{F_inc}
F_R(H_{\theta'_R}) \subseteq P_{\theta_R}. 
\end{equation}
We claim that the map $F := \Psi \circ F_R \circ \Psi^{-1}$ satisfies the claim. Indeed, if $\xi \in [0, 1/2)$ recall
that the map $F_R$ constructed in Proposition \ref{embedding} has the form 
$$F_R(\xi) = s + \xi \cdot 2^{-N}$$
where $s$ is a dyadic rational number and $N$ is some positive integer. Thus, $D^N(F_R(\xi)) = \xi$, so also 
$$\Psi(\xi) = \Psi(D^N(F_R(\xi))) = D^M(\Psi(F_R(\xi)))$$
for some integer $M$. Thus we can write for $\xi \in H_{\theta'_R} \cap Z$
$$\Psi(F_R(\xi)) = t + \Psi(\xi) \cdot 2^{-M}$$
where $t$ is a dyadic rational number, and $t$ and $M$ only depend on $s$ and the element of the partition $\bigcup P_i$ 
to which $\xi$ belongs. Thus we have proven that $F  = \Psi \circ F_R \circ \Psi^{-1}$ is piecewise linear.
Now, by Proposition \ref{dyn_surgery}, eq. \eqref{F_inc}, and Proposition \ref{par_surgery} we have the chain of inclusions
$$\Psi \circ F_R \circ \Psi^{-1}(\tilde{H}_{c'}) \subseteq \Psi \circ F_R (H_{\theta'_R}) \subseteq \Psi (P_{\theta_R}) \subseteq P_c.$$
Finally, if the forward orbit of $\theta$ hits $C_{p/q}$, then by density one can find an angle $\tilde{\theta} \in (\theta', \theta)$ 
such that its forward orbit does not hit $C_{p/q}$, and apply the previous argument to the parameter $\tilde{c}$ with characteristic angle 
$\tilde{\theta}$, thus getting the inclusion
$$F(\tilde{H}_{c'}) \subseteq P_{\tilde{c}} \subseteq P_c.$$
\end{proof}

\begin{proof}[Proof of Theorem \ref{mainvein}.]
Let $c$ be a parameter along the vein $v_{p/q}$. Then by Theorem \ref{dim_entro}
$$\frac{h_{top}(f_c\mid_{T_c})}{\log 2} = \textup{H.dim }H_c.$$
We shall prove that the right hand side equals $\textup{H.dim }P_c$.
Now, since $P_c \subseteq H_c$, it is immediate that 
$$\textup{H.dim }P_c \leq \textup{H.dim }H_c$$
hence we just have to prove the converse inequality.
Let us now assume $c \in v_{p/q}$ non-renormalizable. Then by Proposition \ref{vein_inc}
for each $c' \in [0, c]$ we have the inclusion
$$F(\tilde{H}_{c'}) \subseteq P_c$$ 
so, since $F$ is linear hence it preserves Hausdorff dimension, we have 
$$\textup{H.dim }H_{c'} = \textup{H.dim }\tilde{H}_{c'} \leq \textup{H.dim }P_c$$
and as a consequence 
$$\textup{H.dim }P_c \geq \sup_{c' \in [0, c]} \textup{H.dim }H_{c'}$$
where $[0, c]$ is the segment of the vein $v_{p/q}$ joining $0$ with $c$.
Now by continuity of entropy (Theorem \ref{continuous})
$$\sup_{c' \in [0, c]} \textup{H.dim }H_{c'} = \textup{H.dim }H_c$$
hence the claim is proven for all non-renormalizable parameters along the vein. Now, 
the general case follows as in the proof of Theorem \ref{equaldim} by successively renormalizing and 
using the formulas of Proposition \ref{tuneddim}.
\end{proof}

So far we have worked with the combinatorial model for the veins, which conjecturally coincide with the 
set of angles of rays which actually land on the vein.
Finally, the following proposition proves that the vein and its combinatorial model actually have the 
same dimension, independently of the MLC conjecture.

\begin{proposition} \label{combvsanal}
Let $c \in v_{p/q} \cap \partial \mathcal{M}$ and $\ell$ its characteristic leaf. 
Let
$$\overline{P}_c := \{\theta \in S^1 \ : \ R_M(\theta) \textup{ lands on }v \cap [0, c] \}$$ 
be the set of angles of rays landing on the vein $v$ closer than $c$ to the main cardioid, and 
$$P_c := \{\theta \in S^1 \ : \ \theta \textup{ is endpoint of some } \ell' \in QML, \  \ell' \leq \ell \}$$
its combinatorial model. Then the two sets have equal dimension: 
$$\textup{H.dim }\overline{P}_c = \textup{H.dim }P_c.$$ 
\end{proposition}

\begin{proof}
Fix a principal vein $v_{p/q}$, and let $\tau_W$ be the tuning operator relative to the hyperbolic component of period $q$ in $v_{p/q}$; 
moreover, denote as $\tau$ the tuning operator relative to the hyperbolic component of period $2$. 
Let $P_c^{fr}$ the set of angles which belong to the $P_c$ with finitely renormalizable combinatorics; then Proposition 
\ref{yoccoz} yields the inclusions  
$$\textup{H.dim }P_c^{fr} \subseteq \textup{H.dim }\overline{P}_c \subseteq \textup{H.dim }P_c$$ 
hence to prove the proposition it is sufficient to prove the equality 
$$\textup{H.dim }P_c^{fr} = \textup{H.dim }P_c.$$
Let now $c_n := \tau_W(\tau^n(-2))$ the tips of the chain of nested baby Mandelbrot sets which converge to the 
Feigenbaum parameter in the $p/q$-limb, and let $\ell_n$ be the characteristic leaf of $c_n$. Then if 
$\textup{H.dim }P_c > 0$, there exists a unique $n \geq 1$ such that 
$\ell_{n} < \ell \leq \ell_{n-1}$, hence by monotonicity and by Theorem \ref{mainvein} we know 
$$\textup{H.dim }P_c \geq \textup{H.dim }P_{c_{n}} = \frac{1}{2^n q}.$$
Now, each element of $P_c$ is either of the form $\tau_W\tau^{n-1}(c')$ with $c'$ non-renormalizable, 
or of the form $\tau_W(\tau^{n-1}(\tau_V(c')))$ where $V$ is some  hyperbolic window of period larger than $2$.
However, we know by Proposition \ref{tuneddim} that the image of $\tau_W \circ \tau^{n-1} \circ \tau_V$ has 
Hausdorff dimension at most $\frac{1}{q \cdot 2^{n-1} \cdot 3} < \textup{H.dim }P_c$, hence one must have 
$$\textup{H.dim }P_c = \textup{H.dim }\{\theta \in \overline{P}_c  \ : \ \theta = \tau_W \tau^{n-1}(\theta'), \ \theta' 
\textup{ non-renormalizable} \} \leq \textup{H.dim }P_c^{fr}$$
which yields the claim.
\end{proof}

\end{document}